\documentclass[12pt,notitlepage]{amsart}%
\usepackage{amssymb}
\usepackage{amsfonts}
\usepackage{graphicx}
\usepackage{amscd}
\usepackage{graphicx}
\usepackage{amsmath}%
\newtheorem{theorem}{Theorem}
\theoremstyle{plain}

\newtheorem{definition}{Definition}

\newtheorem{lemma}{Lemma}
\newtheorem{notation}{Notation}

\newtheorem{remark}{Remark}

\numberwithin{equation}{section}
\numberwithin{theorem}{section}
\numberwithin{lemma}{section}
\numberwithin{proposition}{section}
\numberwithin{corollary}{section}

\ifx\pdfoutput\relax\let\pdfoutput=\undefined\fi
\newcount\msipdfoutput
\ifx\pdfoutput\undefined\else
\ifcase\pdfoutput\else
\ifx\paperwidth\undefined\else
\ifdim\paperheight=0pt\relax\else\pdfpageheight\paperheight\fi
\ifdim\paperwidth=0pt\relax\else\pdfpagewidth\paperwidth\fi
\fi\fi\fi
\begin{document}
\title[Reaction-diffusion Equations on Complex Networks via $p$-Adic Analysis]{Reaction-diffusion Equations on Complex Networks and Turing Patterns, via
$p$-Adic Analysis}
\author{W. A. Z\'{u}\~{n}iga-Galindo}
\address{Centro de Investigaci\'{o}n y de Estudios Avanzados del Instituto
Polit\'{e}cnico Nacional\\
Departamento de Matem\'{a}ticas, Unidad Quer\'{e}taro\\
Libramiento Norponiente \#2000, Fracc. Real de Juriquilla. Santiago de
Quer\'{e}taro, Qro. 76230\\
M\'{e}xico.}
\email{wazuniga@math.cinvestav.edu.mx}
\thanks{The author was partially supported by Conacyt Grant No. 250845.}

\begin{abstract}
Nakao and Mikhailov proposed using continuous models (mean-field models) to
study reaction-diffusion systems on networks and the corresponding Turing
patterns. This work aims to show that $p$-adic analysis is the natural tool to
carry out this program. This is possible due to the fact that the discrete
Laplacian attached to a network turns out to be a particular case of a
$p$-adic Laplacian. By embedding the graph attached to the system into the
field of $p$-adic numbers, we construct a continuous $p$-adic version of a
reaction-diffusion system on a network. The existence and uniqueness of the
Cauchy problems for these systems are established. We show that Turing
criteria for the $p$-adic continuous reaction-diffusion system remain
essentially the same as in the classical case. However, the properties of the
emergent patterns are very different. The classical Turing patterns consisting
of alternating domains do not occur in the $p$-adic continuous case. Instead
of this, several domains (clusters) occur. Multistability, that is,
coexistence of a number of different patterns with the same parameter values
occurs. Clustering and multistability have been observed in the computer
simulations of large reaction-diffusion systems on networks. Such behavior can
be naturally explained in the $p$-adic continuous model, in a rigorous
mathematical way, but not in the original discrete model.

\end{abstract}
\maketitle

\section{Introduction}

Pattern-forming, reaction-diffusion systems in continuous media, are typically
described by a system of PDEs of the form%
\begin{equation}
\left\{
\begin{array}
[c]{cc}%
\frac{\partial u(x,t)}{\partial t}= & f\left(  u,v\right)  +\varepsilon\Delta
u(x,t)\\
& \\
\frac{\partial v(x,t)}{\partial t}= & g\left(  u,v\right)  +\varepsilon
d\Delta v(x,t),
\end{array}
\right.  \label{EQ_0}%
\end{equation}
where $x\in\mathbb{R}^{n}$, $t\geq0$, and $u(x,t)$, $v(x,t)$ are local
densities of two chemical species, the functions $f$ and $g$ specify the local
dynamics of $u$ and $v$, and $\varepsilon$, $\varepsilon d$ are the
corresponding diffusion coefficients. Typically $u$ corresponds to an
activator, which autocatalytically enhances its own production, and $v$ an
inhibitor that suppresses $u$. \ The system is initially considered to be at a
steady state\ $\left(  u_{0},v_{0}\right)  $ where $f\left(  u_{0}%
,v_{0}\right)  =g\left(  u_{0},v_{0}\right)  =0$. The Turing instability sets
up once the system is subjected to heterogeneous perturbations. If the
perturbation is homogeneous, the system will reach again the steady state
$(u_{0},v_{0})$. Moreover, one has to take care of the difference between a
Turing instability and the possibly resulting Turing patterns. Indeed, a
Turing instability could give rise to waves or even to another homogeneous
equilibrium, \cite{Turing}, \cite[Chapter 2]{Murray}. The existence of Turing
patterns requires that the parameter $d$ exceeds a certain threshold $d_{c}$.
This event drives to a spontaneous development of a spatial pattern formed by
alternating activator-rich and activator-poor patches. Turing instability in
activator-inhibitor systems establishes a paradigm of non-equilibrium
self-organization, which has been extensively studied for biological and
chemical processes.

In the 70s, Othmer and Scriven started the study of the Turing instability in
network-organized systems \cite{Othmer et al 1}-\cite{Othmer et al 2}. Since
then, reaction-diffusion models on networks have been studied intensively, see
e.g. \cite{Ambrosio et al}, \cite{Boccaletti et al}, \cite{Chung}, \cite{Ide},
\cite{Mocarlo}, \cite{Mugnolo}, \cite{Nakao-Mikhailov}, \cite{Othmer et al 1},
\cite{Othmer et al 2}, \cite{Slavova et al}, \cite{Van Mieghem}, \cite{von
Below}, \cite{Zhao}, and the references therein. In the discrete case, the
continuous media is replaced by a network (an unoriented graph $\mathcal{G}$,
which plays the role of discrete media) composed by $\#V(\mathcal{G})$
independent nodes (vertices) that interact via diffusive transport on
$\#E(\mathcal{G})$ links (edges). The analog of operator $\Delta$ is the
Laplacian of the graph $\mathcal{G}$, which is defined as%
\begin{equation}
\left[  L_{JI}\right]  _{J,I\in V(\mathcal{G})}=\left[  A_{JI}-\gamma
_{I}\delta_{JI}\right]  _{J,I\in V(\mathcal{G})}, \label{EQ_Matrix_L}%
\end{equation}
where $\left[  A_{JI}\right]  _{J,I\in V(\mathcal{G})}$ is the adjacency
matrix of $\mathcal{G}$ and $\gamma_{I}$ is the degree of $I$. The network
analogue of (\ref{EQ_0}) is
\begin{equation}
\left\{
\begin{array}
[c]{l}%
\frac{\partial u_{J}}{\partial t}=f(u_{J},v_{J})+\varepsilon%
%TCIMACRO{\tsum \limits_{I}}%
%BeginExpansion
{\textstyle\sum\limits_{I}}
%EndExpansion
L_{JI}u_{I}\\
\\
\frac{\partial v_{J}}{\partial t}=g(u_{J},v_{J})+\varepsilon d%
%TCIMACRO{\tsum \limits_{I}}%
%BeginExpansion
{\textstyle\sum\limits_{I}}
%EndExpansion
L_{JI}v_{I}.
\end{array}
\right.  \label{EQ_1}%
\end{equation}
In the last fifty years, Turing patterns produced by reaction-diffusion
systems on networks have been studied intensively, mainly by physicists,
biologists and engineers, see e.g. \cite{Boccaletti et al}, \cite{Ide},
\cite{Mocarlo}, \cite{Nakao-Mikhailov}, \cite{Othmer et al 1}, \cite{Othmer et
al 2}, \cite{Zhao}\ and the references therein. Nowadays, there is a large
amount of experimental results about the behavior of these systems, obtained
mainly via computer simulations using large random networks. The
investigations of the Turing patterns for large random networks have revealed
that, whereas the Turing criteria remain essentially the same, as in the
classical case, the properties of the emergent patterns are very different. In
\cite{Nakao-Mikhailov}, by using a physical argument, Nakao and Mikhailov
establish\ that Turing patterns with alternating domains cannot exist in the
network case, and only several domains (clusters) occur. Multistability, that
is, coexistence of a number of different patterns with the same parameter
values, is typically found and hysteresis phenomena are observed. They used
mean-field approximation to understand the Turing patterns when $d>d_{c}$, and
proposed that the mean-field approximation is the natural framework to
understand the peculiar behavior of the Turing patterns on networks.

The central goal of this work is to show the existence of a new $p$-adic
continuous version of the system (\ref{EQ_1}) which provides a very good
approximation of the original system. This is a consequence of the fact that
the discrete Laplacian attached to a network turns out to be a particular case
of a $p$-adic Laplacian. We establish from a mathematical perspective, that
the program of Nakao and Mikhailov proposed in \cite{Nakao-Mikhailov} for
understanding the peculiar behavior of the Turing patterns on networks using
mean-field approximations (i.e. a continuous version of (\ref{EQ_1})) is
possible, but it requires $p$-adic analysis. By `embedding the graph'
$\mathcal{G}$ into $\mathbb{Q}_{p}$, the field of $p$-adic numbers, we
construct a family of continuous $p$-adic versions of the system (\ref{EQ_1}),
which can be studied rigorously by using the classical semigroup theory, see
e.g. \cite{Milan}, \cite{Pazy}, in this way, we are able to study the original
system (\ref{EQ_1}) and to obtain a new $p$-adic continuous version of it,
which corresponds to a `mean-field approximation' of the original system
(\ref{EQ_1}).

From now on $p$ denotes a fixed prime number. A $p$-adic number is a series of
the form%
\begin{equation}
x=x_{-k}p^{-k}+x_{-k+1}p^{-k+1}+\ldots+x_{0}+x_{1}p+\ldots,\text{ with }%
x_{-k}\neq0\text{,} \label{p-adic-number}%
\end{equation}
where the $x_{j}$s \ are $p$-adic digits, i.e. numbers in the set $\left\{
0,1,\ldots,p-1\right\}  $. The set of all possible series of the form
(\ref{p-adic-number}) constitutes the field of $p$-adic numbers $\mathbb{Q}%
_{p}$. There are natural field operations, sum and multiplication, on series
of the form (\ref{p-adic-number}), see e.g. \cite{Koblitz}. There is also a
natural norm in $\mathbb{Q}_{p}$ defined as $\left\vert x\right\vert
_{p}=p^{k}$, for a nonzero $p$-adic number of the form (\ref{p-adic-number}).
The field of $p$-adic numbers with the distance induced by $\left\vert
\cdot\right\vert _{p}$ is a complete ultrametric space. The ultrametric
property refers to the fact that $\left\vert x-y\right\vert _{p}\leq
\max\left\{  \left\vert x-z\right\vert _{p},\left\vert z-y\right\vert
_{p}\right\}  $ for any $x$, $y$, $z\in\mathbb{Q}_{p}$. We denote by
$\mathbb{Z}_{p}$ the unit ball, which consists of all series with expansions
of the form (\ref{p-adic-number}) with $k\geq0$.

We identify each vertex of $\mathcal{G}$ with a $p$-adic number of the form
\begin{equation}
I=I_{0}+I_{1}p+\ldots+I_{N-1}p^{N-1}, \label{p_adic_integer}%
\end{equation}
where the $I_{j}$s are $p$-adic digits. The parameter $N$ is fixed along the
article. We denote by $G_{N}^{0}$ the set of all $p$-adic integers of the form
(\ref{p_adic_integer}) which correspond to the vertices of $\mathcal{G}$. In
this way, we construct a $p$-adic parametrization of the adjacency matrix of
$\mathcal{G}$. Our constructions depend only on the adjacency matrix of the
graph $\mathcal{G}$, for this reason, the tree-like structure of
$\mathbb{Q}_{p}$ does not play any role in our construction.

We denote by $\Omega\left(  p^{N}\left\vert x-I\right\vert _{p}\right)  $ the
characteristic function of the ball centered at $I$ with radius $p^{-N}$,
which corresponds to the set $I+p^{N}\mathbb{Z}_{p}$. Now, we attach to
$\mathcal{G}$ the open compact subset \ $\mathcal{K}_{N}$ defined as the
disjoint union of the balls $I+p^{N}\mathbb{Z}_{p}$ for $I\in G_{N}^{0}$,
\ and a finite dimensional real vector space $X_{N}$ generated by the
functions $\left\{  \Omega\left(  p^{N}\left\vert x-I\right\vert _{p}\right)
\right\}  _{I\in G_{N}^{0}}$. This is the space of continuous functions on
$\mathcal{G}$. There exists a kernel $J_{N}(x,y)$, which is a linear
combination of functions of type $\Omega\left(  p^{N}\left\vert x-I\right\vert
_{p}\right)  \Omega\left(  p^{N}\left\vert y-J\right\vert _{p}\right)  $, $I$,
$J\in G_{N}^{0}$, such that the operator $\boldsymbol{L}_{N}:X_{N}\rightarrow
X_{N}$ is defined as
\begin{equation}
\boldsymbol{L}_{N}\varphi\left(  x\right)  =%
%TCIMACRO{\tint \limits_{\mathcal{K}_{N}}}%
%BeginExpansion
{\textstyle\int\limits_{\mathcal{K}_{N}}}
%EndExpansion
\left(  \varphi\left(  y\right)  -\varphi\left(  x\right)  \right)
J_{N}(x,y)dy, \label{Operator_L}%
\end{equation}
where $dy$ denotes the normalized Haar measure of the locally compact group
$(\mathbb{Q}_{p},+)$, is represented by the matrix $\left[  L_{JI}\right]
_{J,I\in G_{N}^{0}}$, see (\ref{EQ_Matrix_L}).

The space $X_{N}$ (endowed with the supremum norm) plays the role of a mesh,
which can be refined as much as we want. Given $M>N$, we can subdivide each
ball $I+p^{N}\mathbb{Z}_{p}$, with $I\in G_{N}^{0}$, into $p^{M-N}$ disjoint
balls $I_{j}+p^{M}\mathbb{Z}_{p}$. In this way we construct new functions of
type $\sum_{I_{j}}c_{I_{j}}\Omega\left(  p^{M}\left\vert x-I_{j}\right\vert
_{p}\right)  $, which form an $\mathbb{R}$-vector space, denoted as $X_{M}$,
of dimension $p^{M-N}\left(  \#G_{N}^{0},\right)  $. We endow $X_{M}$ with the
supremum norm. Then\ $X_{N}$ is continuously embedded, as a Banach space, into
$X_{M}$. Furthermore, operator $\boldsymbol{L}_{N}$ has a natural extension
$\boldsymbol{L}_{M}$ to $X_{M}$ given by the right-hand side of formula
(\ref{Operator_L}).

We set $X_{\infty}$ for the vector space of real-valued, continuous functions
on $\mathcal{K}_{N}$, endowed with the supremum norm. Then \ $X_{M}$ is
continuously embedded, as a Banach space, into $X_{\infty}$, and $\cup_{M\geq
N}X_{M}$ is dense in $X_{\infty}$. Furthermore, operator $\boldsymbol{L}_{M}$
has an extension\ $\boldsymbol{L}$\ to $X_{\infty}$ given by the right-hand
side of formula (\ref{Operator_L}), which is a linear bounded operator. In
this way on each $X_{\bullet}$, we have an operator $\boldsymbol{L}_{\bullet}%
$, where the dot means $N$, $M$ with $M>N$ or $\infty$, and a continuous
version of the system (\ref{EQ_1}):%

\begin{equation}
\left\{
\begin{array}
[c]{l}%
\frac{\partial}{\partial t}\left[
\begin{array}
[c]{l}%
u^{\left(  \bullet\right)  }\left(  t\right) \\
\\
v^{\left(  \bullet\right)  }\left(  t\right)
\end{array}
\right]  =\left[
\begin{array}
[c]{l}%
f(u^{\left(  \bullet\right)  }\left(  t\right)  ,v^{\left(  \bullet\right)
}\left(  t\right)  )\\
\\
g(u^{\left(  \bullet\right)  }\left(  t\right)  ,v^{\left(  \bullet\right)
}\left(  t\right)  )
\end{array}
\right]  +\left[
\begin{array}
[c]{l}%
\varepsilon\boldsymbol{L}_{\bullet}u^{\left(  \bullet\right)  }\left(
t\right) \\
\\
\varepsilon d\boldsymbol{L}_{\bullet}v^{\left(  \bullet\right)  }\left(
t\right)
\end{array}
\right]  ,\\
\\
t\in\left[  0,\tau\right)  \text{, }x\in\mathcal{K}_{N}.
\end{array}
\right.  \label{EQ_System_2}%
\end{equation}
We study the Cauchy problem attached to (\ref{EQ_System_2}), when the initial
datum belongs to a sufficiently small open set containing a steady
state\ $\left(  u_{0},v_{0}\right)  $ where $f\left(  u_{0},v_{0}\right)
=g\left(  u_{0},v_{0}\right)  =0$, assuming that $\nabla f\left(  x\right)
\neq0$ and $\nabla g\left(  x\right)  \neq0$ for $x$ sufficiently close to
$\left(  u_{0},v_{0}\right)  \in\mathbb{R}^{2}$. Under these hypotheses we
establish that (simultaneously) all the Cauchy problems attached to
(\ref{EQ_System_2}) have a unique solution, with the same maximal interval of
existence, see Theorem \ref{Theorem2} and Lemma \ref{Lemma_mild_sol}. By using
standard techniques based on semigroup theory, one shows the uniqueness of all
the initial value problems (\ref{EQ_System_2}), as well as the continuous
dependence of the initial conditions. In Section \ref{Section_examples}, we
study the Cauchy problem attached to the Brusselator\ and the CIMA reaction.

In the case $X_{\infty}$, we call the system (\ref{EQ_System_2}) the
\textit{mean-field model} (or approximation) of the original system
(\ref{EQ_1}). For $M$ sufficiently large, the solution of the Cauchy problem
attached to the mean-field model is arbitrarily close to the solution of
system (\ref{EQ_System_2}) in $X_{M}$, see Theorem \ref{Theorem3}.

The matrix $\mathbb{A}^{\left(  M\right)  }$ of the operator $\boldsymbol{L}%
_{M}$ acting on $X_{M}$ (after renaming the elements of the basis of $X_{M}%
$)\ is a diagonal-type\ matrix of the form%
\[
\mathbb{A}^{\left(  M\right)  }=\left[
\begin{array}
[c]{ccccc}%
\mathbb{A}_{N;M} &  &  &  & \\
& \ddots &  &  & \\
&  & \mathbb{A}_{N;M} &  & \\
&  &  & \ddots & \\
&  &  &  & \mathbb{A}_{N;M}%
\end{array}
\right]  _{p^{M-N}\times p^{M-N}},
\]
where
\[
\mathbb{A}_{N;M}=\left[  p^{N-M}A_{JI}-\gamma_{I}\delta_{JI}\right]  _{J,I\in
G_{N}^{0}}=p^{N-M}\left[  A_{JI}-p^{M-N}\gamma_{I}\delta_{JI}\right]  _{J,I\in
G_{N}^{0}},
\]
see Lemma \ref{Lemma3}. The matrix $\mathbb{A}^{\left(  M\right)  }$, which
corresponds to a network constructed by using $p^{M-N}$ \textit{replicas of
the original network}, each of these replicas\ corresponds to a network having
a diffusion operator of type $\mathbb{A}_{N;M}$ and the corresponding $p$-adic
diffusion equation is
\begin{equation}
\frac{\partial f^{\left(  N\right)  }\left(  x,t\right)  }{\partial
t}=\varepsilon^{\prime}\boldsymbol{L}_{N,\lambda}f^{\left(  N\right)  }\left(
x,t\right)  \label{System_X_N}%
\end{equation}
where $\varepsilon^{\prime}=p^{N-M}\varepsilon$, $\lambda=p^{M-N}$, and
$\boldsymbol{L}_{N,\lambda}:X_{N}\rightarrow X_{N}$ is defined as
\[
L_{N,\lambda}\varphi\left(  x\right)  =\int\limits_{\mathcal{K}_{N}}\left\{
\varphi(y)-\lambda\varphi\left(  x\right)  \right\}  J_{N}\left(  x,y\right)
dy.
\]
The equations of type (\ref{System_X_N}) form a parametric family indexed by
$\left(  \varepsilon^{\prime},\lambda\right)  $, which is invariant under a
scale change of type $\left(  \varepsilon^{\prime},\lambda\right)
\rightarrow\left(  \delta\varepsilon^{\prime},\lambda\delta^{-1}\right)  $,
for $\delta\in\left(  0,1\right)  $, see Theorems \ref{Theorem1_fractal},
\ref{Theorem1_fractal_A}.

In conclusion, the mean-field approximation is the `limit' of the system
(\ref{EQ_System_2}) in $X_{M}$ when $M$ tends to infinity. In turn, any
solution of the system (\ref{EQ_System_2}) in $X_{M}$ ($M>N$) is made of
$p^{M-N}$ solutions of $p^{M-N}$ systems of type (\ref{EQ_System_2}) in
$X_{N}$. Each of them is a scaled version (a scaled replica) of the original
system (\ref{EQ_1}).

In order to understand the `physical contents' of the mean-field model, it is
completely necessary to study its diffusion mechanism, which means to study
the following Cauchy problem:%
\begin{equation}
\left\{
\begin{array}
[c]{l}%
\frac{\partial f\left(  x,t\right)  }{\partial t}=\varepsilon\boldsymbol{L}%
f\left(  x,t\right)  \text{, \ }x\in\mathcal{K}_{N}\text{, }t>0\\
\\
f\left(  x,0\right)  =f_{0}(x)\in{\large X}_{\infty}.
\end{array}
\right.  \label{EQ_Cauchy_Heat_eq}%
\end{equation}
The semigroup attached to (\ref{EQ_Cauchy_Heat_eq}), $\left\{  e^{\varepsilon
t\boldsymbol{L}}\right\}  _{t\geq0}$, is a Feller semigroup, and consequently,
there is a Markov process attached to (\ref{EQ_Cauchy_Heat_eq}), see Theorem
\ref{Theorem1}. This implies that in the mean-field model, the chemical
species $u$, $v$ interact via a random walk like in the classical case
(\ref{EQ_0}): the chemical reactions involving species $u$, $v$ occur as a
consequence of a random walk of the particles forming them. This random walk
is produced by changes in the concentrations, which can be deterministically
modeled by Fick's law of diffusion. The whole picture is coded as a $p$-adic
heat equation. A similar result holds in $X_{M}$\ for the equations
(\ref{EQ_Cauchy_Heat_eq}), see Theorem \ref{Theorem1_A}, and for the equations
of type (\ref{System_X_N}), see Sections \ref{Sect_selfsimilarity_1}\ and
\ref{Section_selsimilarity_XM}. It is important to mention here that the
$p$-adic diffusion has strong differences with respect to the classical
diffusion in $\mathbb{R}$.

There are infinitely many choices for the kernel $J_{N}\left(  x,y\right)  $
such that the corresponding operator $\boldsymbol{L}$ is a diffusion operator,
i.e. such that (\ref{EQ_Cauchy_Heat_eq}) is a heat equation, see Theorem
\ref{Theorem1} and the comments\ following it. Consequently in $X_{\infty}$
there are infinitely many reaction-diffusion systems `similar' to (\ref{EQ_1}).

The $p$-adic heat equations, which include equations of types
(\ref{System_X_N}) and (\ref{EQ_Cauchy_Heat_eq}), and their associated Markov
processes has been studied intensively in the last thirty years due to
physical motivations, see e.g. \cite{Av-4}, \cite{Av-5}, \cite{Bendiko et al},
\cite{Bendikov}, \cite{KKZuniga}, \cite{Koch}, \cite{Kozyrev SV},
\cite{To-Zuniga-1}, \cite{To-Zuniga-2}, \cite{V-V-Z}, \cite{Zuniga-LNM-2016}%
\ and the references therein. A central paradigm in physics of complex systems
(for instance proteins) asserts that the dynamics of such systems can be
modeled as a random walk in the energy landscape of the system, see e.g.
\cite{Fraunfelder et al}, \cite{Kozyrev SV}, and the references therein.
Typically the energy landscape is approximated by an ultrametric space (a tree
called the disconnectivity graph of the energy landscape) and a function on
this space describing the distribution of the activation barriers, see e.g.
\cite{Becker et al}. Then the dynamics is described by a discrete system of
ODEs on a tree, see e.g. \cite{Becker et al}, \cite{KKZuniga}, \cite{Kozyrev
SV}, and the references therein.

Avetisov, Kozyrev et al. discovered that under suitable physical and
mathematical hypotheses, the above-mentioned discrete system of ODEs on a tree
has a `continuous $p$-adic limit' (a mean-field model) as a $p$-adic diffusion
equation, see \cite{Av-4}, \cite{Av-5}, \cite{KKZuniga}. From a mathematical
perspective, the original system of ODEs on a tree is a good discretization of
the corresponding $p$-adic diffusion equation, see \cite{Zuniga-Nonlinear}.
The diffusion in the systems of type (\ref{EQ_System_2}) is $p$-adic
diffusion. Our results on $p$-adic diffusion equations are new. Indeed, they
correspond to the so-called `degenerate energy landscapes' in the terminology
of Avetisov, Kozyrev et al. In the case in which the kernel $J_{N}$\ has the
form $J(\left\vert x-y\right\vert _{p})$, the corresponding diffusion
equations were studied in \cite{To-Zuniga-1}. Operators similar to
$\boldsymbol{L}$ have been also studied by Bendikov \cite{Bendikov},
\cite{Bendikov-Pote} and Kozyrev \cite{Kozyrev-SB}.

Recently, $p$-adic nonlinear reaction-diffusion equations connected with
physical models involving tree-like graphs have been studied intensively,\ see
e.g. \cite{Antoniuk et al. 1}, \cite{Antoniuk et al.}, \cite{Khrennikov et al
1}, \cite{Khrennikov et al}, \cite{Kh-Kochubei}, \cite{OL-Kh-1},
\cite{OL-Kh-2}, \cite{Zuniga-JPA}, \cite{Zuniga-Nonlinear}. An important
novelty is that the $p$-adic reaction-diffusion equations presented here are
connected with arbitrary (non-oriented, simple) graphs, which \ means that we
are not restricted to tree-like graphs. Since this is a central idea, we
discuss it briefly. The dynamics of complex energy landscapes can be
approximated by a system of kinetic equations which describe the transitions
between local minima energy. The system of kinetic equations has the form%
\begin{equation}
\frac{d}{dt}f_{i}(t)=-\sum_{j}Q_{i,j}\left\{  f_{i}(t)-f_{j}(t)\right\}  ,
\label{Eq_System_1}%
\end{equation}
where $t\in\mathbb{R}$ is time, $f_{i}(t)$ is the population of the $i$-th
local minima, $Q=\left[  Q_{i,j}\right]  $ is the matrix of transition rates.
The matrix $Q$ has a special bock structure that can be understood using
ultrametrics. The matrix $Q$ admits a $p$-adic parametrization, i.e.
$Q_{i,j}=q\left(  \left\vert l(i)-l(j)\right\vert _{p}\right)  $ with
$l\left(  i\right)  \in p^{m}\mathbb{Z}_{p}/p^{n}\mathbb{Z}_{p}$, $n\geq m$.
By using this $p$-adic parametrization the system (\ref{Eq_System_1}) takes
the form%
\[
\frac{d}{dt}f_{I}(t)=-\sum_{J\in p^{m}\mathbb{Z}_{p}/p^{n}\mathbb{Z}_{p}%
}q\left(  \left\vert I-J\right\vert _{p}\right)  \left\{  f_{I}(t)-f_{J}%
(t)\right\}  p^{-n},
\]
where $I$, $J\in p^{m}\mathbb{Z}_{p}/p^{n}\mathbb{Z}_{p}$. In the limit
$m\rightarrow-\infty$, $n\rightarrow\infty$, the above system becomes%
\begin{equation}
\frac{\partial}{\partial t}f\left(  x,t\right)  =\int\limits_{\mathbb{Q}_{p}%
}q\left(  \left\vert x-y\right\vert _{p}\right)  \left\{
f(x,t)-f(y,t)\right\}  dy, \label{Eq_System_1A}%
\end{equation}
where $x$, $y\in\mathbb{Q}_{p}$, $t\in\mathbb{R}$, which is a $p$-adic
diffusion equation. In the above-mentioned limit process the size of the
matrix $Q$ changes. In the limit, the kernel $q$ in (\ref{Eq_System_1A})
corresponds to an infinite version of the matrix $Q$.

In this article the adjacency matrix $\left[  A_{JI}\right]  _{J,I\in
G_{N}^{0}}$ of a (non-oriented, simple) graph plays the role of the matrix
$Q$. The adjacency matrix is fixed, in particular the parameter $N$ is fixed.
With the matrix $\left[  A_{JI}\right]  _{J,I\in G_{N}^{0}}$ we construct a
kernel $J_{N}(x,y)$, and an operator $\boldsymbol{L}_{N}$, see
(\ref{Operator_L}). The kernel $J_{N}$ is the analog of $q$, but this
kernel\ is not obtained by a limit process starting with the adjacency matrix.
Initially, the operator $\boldsymbol{L}_{N}$ is defined in a\ finite
dimensional $\mathbb{R}$-space $X_{N}$ spanned by the characteristic functions
$\left\{  \Omega\left(  p^{N}\left\vert x-I\right\vert _{p}\right)  \right\}
_{I\in G_{N}^{0}}$. By `refining' the space $X_{N}$, we obtain new function
spaces and extensions of the operator $\boldsymbol{L}_{N}$, but all these
extensions have the same kernel, $J_{N}(x,y)$. Then the $p$-adic heat equation
introduced here are obtained by a limit process on the spaces $X_{N}$, $N\geq
M$, using a fixed kernel $J_{N}(x,y)$. In the previous paragraph, a new
$p$-adic diffusion operator, see (\ref{Eq_System_1A}), is constructed by a
limit process starting with a matrix $Q$. Here, we show that the Laplacian of
a graph is a $p$-adic integral operator, which admits an extension to the
space of continuous functions supported in a suitable open compact subset of
$\mathbb{Q}_{p}$.

Another goal of this work is to study the formation of Turing patterns in the
reaction-diffusion systems of type (\ref{EQ_System_2}).\ The novelty here is
the possibility of studying Turing patterns of reaction-diffusion systems on
networks by using $p$-adic continuous reaction-diffusion equations. To achieve
this goal, it is necessary to understand the spectra of operators $\left[
L_{IJ}\right]  _{I,J\in G_{N}^{0}}$, $\boldsymbol{L}_{N}$, $\boldsymbol{L}%
_{M}$ with $M>N$, $\boldsymbol{L}$. The spectrum of the graph Laplacian matrix
$\left[  L_{JI}\right]  _{J,I\in G_{N}^{0}}$ is well-understood, see e.g.
\cite{Van Mieghem}. The operator $\boldsymbol{L}$ has unique compact extension
$\boldsymbol{L}:L^{2}(\mathcal{K}_{N},\mathbb{C})\rightarrow L^{2}%
(\mathcal{K}_{N},\mathbb{C})$. Furthermore, the space $L^{2}(\mathcal{K}%
_{N},\mathbb{C})$ has an orthonormal basis formed by eigenfunctions of
operator $\boldsymbol{L}$, see Theorem \ref{Theorem4}.

In $X_{\bullet}$ the Turing instability criteria can be established using the
classical argument, see Theorem \ref{Theorem5}, Remark \ref{Nota_Theorem5},
and \cite[Chapter 2]{Murray}. In $X_{\infty}$, the Turing pattern has the form%
\begin{gather}
\sum\limits_{\kappa_{1}<\kappa<\kappa_{2}}\sum\limits_{I}A_{I\kappa}e^{\lambda
t}\Omega\left(  p^{N}\left\vert x-I\right\vert _{p}\right) \nonumber\\
+\sum\limits_{\kappa_{1}<\kappa<\kappa_{2}}\text{ }\sum\limits_{I}%
%TCIMACRO{\dsum \limits_{rnj}}%
%BeginExpansion
{\displaystyle\sum\limits_{rnj}}
%EndExpansion
B_{rnj}e^{\lambda t}p^{\frac{r}{2}}\cos\left(  \left\{  p^{-r-1}jx\right\}
_{p}\right)  \Omega\left(  \left\vert p^{r}x-n\right\vert _{p}\right)
\nonumber\\
+\sum\limits_{\kappa_{1}<\kappa<\kappa_{2}}\sum\limits_{I}%
%TCIMACRO{\dsum \limits_{rnj}}%
%BeginExpansion
{\displaystyle\sum\limits_{rnj}}
%EndExpansion
C_{rnj}e^{\lambda t}p^{\frac{r}{2}}\sin\left(  \left\{  p^{-r-1}jx\right\}
_{p}\right)  \Omega\left(  \left\vert p^{r}x-n\right\vert _{p}\right)  \text{
} \label{Turing_Pattern}%
\end{gather}
for $t\rightarrow+\infty$, where $\kappa$ runs through unstable modes, and
$\lambda=\lambda\left(  \kappa\right)  $, $r=r(I,\kappa)$, $n=n(I,\kappa)$,
and the ball $p^{-r}n+$ $p^{-r}\mathbb{Z}_{p}\subset\mathcal{K}_{N}$. In the
case $X_{M}$, with $M\geq N$, the Turing pattern does not contain the terms
involving sine and cosine functions. On the other hand, in the results
reported in the literature for the Turing patterns on networks the pattern is
described as $\sum A_{I}e^{\lambda t}\varphi_{I}$, where $\varphi_{I}$ is the
eigenfunction corresponding to $\mu_{I}$, see e.g. \cite{Nakao-Mikhailov}. Our
results, see Theorems \ref{Theorem2} and \ref{Theorem5}, show that Turing
criteria remain essentially the same as in the classical case. If the Turing
instability criteria holds true in $X_{M}$ for some $M\geq N$, then it is true
in $X_{\infty}$, but the converse is not true, see \ref{Theorem5} and Section
\ref{Section_complete_graphs}.

By identifying the ball $I+p^{N}\mathbb{Z}_{p}$ with a cluster, we have that
the Turing pattern (\ref{Turing_Pattern}) is organized in a finite number of
\ disjoint clusters, each of them supporting a stationary pattern. All these
patterns are controlled by the same kinetic parameters. Notice that the
occurrence of clusters in the Turing patterns is a direct consequence of the
hierarchical structure of $\mathbb{Q}_{p}$: every ball is a finite disjoint
union of balls of smaller radii. Notice that this corresponds exactly with the
qualitative description of the Turing patterns on networks given by Nakao and
Mikhailov \cite{Nakao-Mikhailov}\ in terms of clustering and multistability.
We have not found in the current literature results explaining (in\ a rigorous
mathematical way) the qualitative description of the Turing patterns on
networks given by Nakao and Mikhailov.

Reaction-diffusion systems on networks result to be ordinary differential
equations, under the assumption of Fickian diffusion, and thus standard
(local) existence and unicity conditions do apply. This is enough to prove the
onset of Turing instability, see e.g. \cite[Chapters 6, 7]{Perthame},
\cite{Chung}. On the other hand, the study of the $p$-adic continuous
reaction-diffusion systems on networks and its approximations requires the use
of abstract semilinear evolution equations. The study of traveling waves in
the $p$-adic continuous reaction-diffusion systems introduced here is an open
problem. The classical techniques cannot be applied because the time is a real
variable and the position is a $p$-adic number. Nowadays the study of
differential equations on graphs is a relevant mathematical matter, see e.g.
\cite{Berkolaiko} \cite{Mugnolo}.

In conclusion, we can say that from the perspective of the proposal of Nakao
and Mikhailov \cite{Nakao-Mikhailov} of studying reaction-diffusion systems on
networks\ and the corresponding Turing patterns using mean-field models, the
$p$-adic analysis results to be the natural tool for such studies. The key
point is that the clustering and multistability observed in the computer
simulations are naturally explained in the $p$-adic continuous model, and to
the best of our understanding, these properties cannot be mathematically
explained in the original discrete model.\ Of course, a good discretization of
the continuous model drives us to a classical reaction-diffusion system on a network.

\section{\label{Fourier Analysis}Basic facts on $p$-adic analysis}

In this section we collect some basic results about $p$-adic analysis that
will be used in the article. For an in-depth review of the $p$-adic analysis
the reader may consult \cite{Alberio et al}, \cite{Taibleson}, \cite{V-V-Z}.

\subsection{The field of $p$-adic numbers}

Along this article $p$ will denote a prime number. The field of $p-$adic
numbers $%
%TCIMACRO{\U{211a} }%
%BeginExpansion
\mathbb{Q}
%EndExpansion
_{p}$ is defined as the completion of the field of rational numbers
$\mathbb{Q}$ with respect to the $p-$adic norm $|\cdot|_{p}$, which is defined
as
\[
\left\vert x\right\vert _{p}=\left\{
\begin{array}
[c]{lll}%
0 & \text{if} & x=0\\
&  & \\
p^{-\gamma} & \text{if} & x=p^{\gamma}\frac{a}{b}\text{,}%
\end{array}
\right.
\]
where $a$ and $b$ are integers coprime with $p$. The integer $\gamma:=ord(x)
$, with $ord(0):=+\infty$, is called the\textit{\ }$p-$\textit{adic order of}
$x$.

Any $p-$adic number $x\neq0$ has a unique expansion of the form
\[
x=p^{ord(x)}\sum_{j=0}^{\infty}x_{j}p^{j},
\]
where $x_{j}\in\{0,\dots,p-1\}$ and $x_{0}\neq0$. By using this expansion, we
define \textit{the fractional part of }$x\in\mathbb{Q}_{p}$, denoted as
$\{x\}_{p}$, to be the rational number
\[
\left\{  x\right\}  _{p}=\left\{
\begin{array}
[c]{lll}%
0 & \text{if} & x=0\text{ or }ord(x)\geq0\\
&  & \\
p^{ord(x)}\sum_{j=0}^{-ord_{p}(x)-1}x_{j}p^{j} & \text{if} & ord(x)<0.
\end{array}
\right.
\]
In addition, any non-zero $p-$adic number can be represented uniquely as
$x=p^{ord(x)}ac\left(  x\right)  $ where $ac\left(  x\right)  =\sum
_{j=0}^{\infty}x_{j}p^{j}$, $x_{0}\neq0$, is called the \textit{angular
component} of $x$. Notice that $\left\vert ac\left(  x\right)  \right\vert
_{p}=1$.

For $r\in\mathbb{Z}$, denote by $B_{r}(a)=\{x\in%
%TCIMACRO{\U{211a} }%
%BeginExpansion
\mathbb{Q}
%EndExpansion
_{p};\left\vert x-a\right\vert _{p}\leq p^{r}\}$ \textit{the ball of radius
}$p^{r}$ \textit{with center at} $a\in%
%TCIMACRO{\U{211a} }%
%BeginExpansion
\mathbb{Q}
%EndExpansion
_{p}$, and take $B_{r}(0):=B_{r}$. The ball $B_{0}$ equals $\mathbb{Z}_{p}$,
\textit{the ring of }$p-$\textit{adic integers of }$%
%TCIMACRO{\U{211a} }%
%BeginExpansion
\mathbb{Q}
%EndExpansion
_{p}$. We also denote by $S_{r}(a)=\{x\in\mathbb{Q}_{p};|x-a|_{p}=p^{r}\}$
\textit{the sphere of radius }$p^{r}$ \textit{with center at} $a\in%
%TCIMACRO{\U{211a} }%
%BeginExpansion
\mathbb{Q}
%EndExpansion
_{p}$, and take $S_{r}(0):=S_{r}$. We notice that $S_{0}=\mathbb{Z}%
_{p}^{\times}$ (the group of units of $\mathbb{Z}_{p}$). The balls and spheres
are both open and closed subsets in $%
%TCIMACRO{\U{211a} }%
%BeginExpansion
\mathbb{Q}
%EndExpansion
_{p}$. In addition, two balls in $%
%TCIMACRO{\U{211a} }%
%BeginExpansion
\mathbb{Q}
%EndExpansion
_{p}$ are either disjoint or one is contained in the other.

The metric space $\left(
%TCIMACRO{\U{211a} }%
%BeginExpansion
\mathbb{Q}
%EndExpansion
_{p},\left\vert \cdot\right\vert _{p}\right)  $ is a complete ultrametric
space. As a topological space $\left(
%TCIMACRO{\U{211a} }%
%BeginExpansion
\mathbb{Q}
%EndExpansion
_{p},|\cdot|_{p}\right)  $ is totally disconnected, i.e. the only connected
subsets of $%
%TCIMACRO{\U{211a} }%
%BeginExpansion
\mathbb{Q}
%EndExpansion
_{p}$ are the empty set and the points. In addition, $\mathbb{Q}_{p}$\ is
homeomorphic to a Cantor-like subset of the real line, see e.g. \cite{Alberio
et al}, \cite{V-V-Z}. A subset of $\mathbb{Q}_{p}$ is compact if and only if
it is closed and bounded in $\mathbb{Q}_{p}$, see e.g. \cite[Section
1.3]{V-V-Z}, or \cite[Section 1.8]{Alberio et al}. The balls and spheres are
compact subsets. Thus $\left(
%TCIMACRO{\U{211a} }%
%BeginExpansion
\mathbb{Q}
%EndExpansion
_{p},|\cdot|_{p}\right)  $ is a locally compact topological space.

\begin{notation}
We will use $\Omega\left(  p^{-r}|x-a|_{p}\right)  $ to denote the
characteristic function of the ball $B_{r}(a)$.
\end{notation}

\subsection{Some function spaces}

A complex-valued function $\varphi$ defined on $%
%TCIMACRO{\U{211a} }%
%BeginExpansion
\mathbb{Q}
%EndExpansion
_{p}$ is \textit{called locally constant} if for any $x\in%
%TCIMACRO{\U{211a} }%
%BeginExpansion
\mathbb{Q}
%EndExpansion
_{p}$ there exist an integer $l(x)\in\mathbb{Z}$ such that
\begin{equation}
\varphi(x+x^{\prime})=\varphi(x)\text{ for any }x^{\prime}\in B_{l(x)}.
\label{local_constancy_parameter}%
\end{equation}
\ A function $\varphi:%
%TCIMACRO{\U{211a} }%
%BeginExpansion
\mathbb{Q}
%EndExpansion
_{p}\rightarrow\mathbb{C}$ is called a \textit{Bruhat-Schwartz function (or a
test function)} if it is locally constant with compact support. In this case,
we can take $l=l(\varphi)$ in (\ref{local_constancy_parameter}) independent of
$x$, the largest of such integers is called \textit{the parameter of local
constancy} of $\varphi$. The $\mathbb{C}$-vector space of Bruhat-Schwartz
functions is denoted by $\mathcal{D}(%
%TCIMACRO{\U{211a} }%
%BeginExpansion
\mathbb{Q}
%EndExpansion
_{p},\mathbb{C})$. We will denote by $\mathcal{D}(%
%TCIMACRO{\U{211a} }%
%BeginExpansion
\mathbb{Q}
%EndExpansion
_{p},\mathbb{R})$, the $\mathbb{R}$-vector space of test functions.

Since $(\mathbb{Q}_{p},+)$ is a locally compact topological group, there
exists a Borel measure $dx$, called the Haar measure of $(\mathbb{Q}_{p},+)$,
unique up to multiplication by a positive constant, such that $\int_{U}dx>0 $
for every non-empty Borel open set $U\subset\mathbb{Q}_{p}$, and satisfying
$\int_{E+z}dx=\int_{E}dx$ for every Borel set $E\subset\mathbb{Q}_{p}$, see
e.g. \cite[Chapter XI]{Halmos}. If we normalize this measure by the condition
$\int_{\mathbb{Z}_{p}}dx=1$, then $dx$ is unique. From now on we denote by
$dx$ the normalized Haar measure of $(\mathbb{Q}_{p},+)$.

Given $\rho\in\lbrack0,\infty)$ and an open subset $U\subset%
%TCIMACRO{\U{211a} }%
%BeginExpansion
\mathbb{Q}
%EndExpansion
_{p}$, we denote by $L^{\rho}\left(  U,%
%TCIMACRO{\U{2102} }%
%BeginExpansion
\mathbb{C}
%EndExpansion
\right)  $ the $%
%TCIMACRO{\U{2102} }%
%BeginExpansion
\mathbb{C}
%EndExpansion
-$vector space of all the complex valued functions $g$ defined on
$U$\ satisfying
\[
\left\Vert g\right\Vert _{\rho}=\left\{  \int_{U}\left\vert g\left(  x\right)
\right\vert ^{\rho}dx\right\}  ^{\frac{1}{\rho}}<\infty,
\]
and $L^{\infty}\left(  U,%
%TCIMACRO{\U{2102} }%
%BeginExpansion
\mathbb{C}
%EndExpansion
\right)  $ denotes the $\mathbb{C}-$vector space of all the complex valued
functions $g$ defined in $U$ such that the essential supremum of $|g|$ is
bounded. The corresponding $\mathbb{R}$-vector spaces are denoted as $L^{\rho
}\left(  U,\mathbb{R}\right)  $, $1\leq\rho\leq\infty$.

Let $U$ be an open subset of $%
%TCIMACRO{\U{211a} }%
%BeginExpansion
\mathbb{Q}
%EndExpansion
_{p}$, we denote by $\mathcal{D}(U,\mathbb{C})$ the $\mathbb{C}$-vector space
of all test functions with support in $U$. For each $\rho\in\lbrack1,\infty)$,
$\mathcal{D}(U,\mathbb{C})$ is dense in $L^{\rho}\left(  U,\mathbb{C}\right)
$, see e.g. \cite[Proposition 4.3.3]{Alberio et al}.

\subsection{Fourier transform}

Set $\chi_{p}(y)=\exp(2\pi i\{y\}_{p})$ for $y\in%
%TCIMACRO{\U{211a} }%
%BeginExpansion
\mathbb{Q}
%EndExpansion
_{p}$. The map $\chi_{p}$ is an additive character on $%
%TCIMACRO{\U{211a} }%
%BeginExpansion
\mathbb{Q}
%EndExpansion
_{p}$, i.e. a continuous map from $\left(
%TCIMACRO{\U{211a} }%
%BeginExpansion
\mathbb{Q}
%EndExpansion
_{p},+\right)  $ into $S$ (the unit circle considered as multiplicative group)
satisfying $\chi_{p}(x_{0}+x_{1})=\chi_{p}(x_{0})\chi_{p}(x_{1})$,
$x_{0},x_{1}\in%
%TCIMACRO{\U{211a} }%
%BeginExpansion
\mathbb{Q}
%EndExpansion
_{p}$. The additive characters of $%
%TCIMACRO{\U{211a} }%
%BeginExpansion
\mathbb{Q}
%EndExpansion
_{p}$ form an Abelian group which is isomorphic to $\left(
%TCIMACRO{\U{211a} }%
%BeginExpansion
\mathbb{Q}
%EndExpansion
_{p},+\right)  $, the isomorphism is given by $\xi\rightarrow\chi_{p}(\xi x)$,
see e.g. \cite[Section 2.3]{Alberio et al}.

If $f\in L^{1}\left(
%TCIMACRO{\U{211a} }%
%BeginExpansion
\mathbb{Q}
%EndExpansion
_{p},%
%TCIMACRO{\U{2102} }%
%BeginExpansion
\mathbb{C}
%EndExpansion
\right)  $ its Fourier transform is defined by
\[
(\mathcal{F}f)(\xi)=\int_{%
%TCIMACRO{\U{211a} }%
%BeginExpansion
\mathbb{Q}
%EndExpansion
_{p}}\chi_{p}(\xi x)f(x)dx,\quad\text{for }\xi\in%
%TCIMACRO{\U{211a} }%
%BeginExpansion
\mathbb{Q}
%EndExpansion
_{p}.
\]
We will also use the notation $\mathcal{F}_{x\rightarrow\xi}f$ and
$\widehat{f}$\ for the Fourier transform of $f$. The Fourier transform is a
linear isomorphism (algebraic and topological) from $\mathcal{D}(%
%TCIMACRO{\U{211a} }%
%BeginExpansion
\mathbb{Q}
%EndExpansion
_{p},\mathbb{C})$ onto itself satisfying
\begin{equation}
(\mathcal{F}(\mathcal{F}f))(\xi)=f(-\xi), \label{FF(f)}%
\end{equation}
for every $f\in\mathcal{D}(%
%TCIMACRO{\U{211a} }%
%BeginExpansion
\mathbb{Q}
%EndExpansion
_{p},\mathbb{C}),$ see e.g. \cite[Section 4.8]{Alberio et al}. If $f\in
L^{2},$ its Fourier transform is defined as
\[
(\mathcal{F}f)(\xi)=\lim_{k\rightarrow\infty}\int_{|x|_{p}\leq p^{k}}\chi
_{p}(\xi x)f(x)dx,\quad\text{for }\xi\in%
%TCIMACRO{\U{211a} }%
%BeginExpansion
\mathbb{Q}
%EndExpansion
_{p},
\]
where the limit is taken in $L^{2}(%
%TCIMACRO{\U{211a} }%
%BeginExpansion
\mathbb{Q}
%EndExpansion
_{p},\mathbb{C})$. We recall that the Fourier transform is unitary on $L^{2}(%
%TCIMACRO{\U{211a} }%
%BeginExpansion
\mathbb{Q}
%EndExpansion
_{p},\mathbb{C}),$ i.e. $||f||_{2}=||\mathcal{F}f||_{2}$ for $f\in L^{2}(%
%TCIMACRO{\U{211a} }%
%BeginExpansion
\mathbb{Q}
%EndExpansion
_{p},\mathbb{C})$ and that (\ref{FF(f)}) is also valid in $L^{2}(%
%TCIMACRO{\U{211a} }%
%BeginExpansion
\mathbb{Q}
%EndExpansion
_{p},\mathbb{C})$, see e.g. \cite[Chapter $III$, Section 2]{Taibleson}.

\section{$p$-Adic analogues of Reaction-diffusion systems on networks}

We consider an arbitrary graph $\mathcal{G}$ with vertices $I\in G_{N}^{0}$,
where $G_{N}^{0}$ is a finite set. Later on we will consider $G_{N}^{0}$ as a
subset of $\mathbb{Q}_{p}$, in this way we will obtain an embedding \ of
$\mathcal{G}$ in this space. When there is no connection between the vertices,
the dynamics on each vertex is controlled by local interactions described as%
\begin{equation}
\left\{
\begin{array}
[c]{l}%
\frac{\partial u_{J}}{\partial t}=f(u_{J},v_{J})\\
\\
\frac{\partial v_{J}}{\partial t}=g(u_{J},v_{J}),
\end{array}
\right.  \label{Eq_2}%
\end{equation}
for $J\in G_{N}^{0}$, where a pair $\left(  u_{J},v_{J}\right)  =(u_{J}\left(
t\right)  ,v_{J}\left(  t\right)  )$ represents some quantities in the vertex
$J$, such as population densities of biological species or concentrations of
chemical substances. When connection between vertices is taken into account,
we assume the existence of a flux of quantities between these vertices. If two
\ vertices are not connected, there is no flux between them. The flux is
assumed to be given by Fick's law of diffusion, which means that the flux is
proportional to the difference of quantities on the two vertices. Therefore
the dynamics of $u_{J}$\ and $v_{J}$ on vertex $J$ is described as%
\begin{equation}
\left\{
\begin{array}
[c]{l}%
\frac{\partial u_{J}}{\partial t}=f(u_{J},v_{J})+\varepsilon%
%TCIMACRO{\tsum \limits_{I\in G_{N}^{0}}}%
%BeginExpansion
{\textstyle\sum\limits_{I\in G_{N}^{0}}}
%EndExpansion
A_{JI}\left\{  u_{I}-u_{J}\right\} \\
\\
\frac{\partial v_{J}}{\partial t}=g(u_{J},v_{J})+\varepsilon d%
%TCIMACRO{\tsum \limits_{I\in G_{N}^{0}}}%
%BeginExpansion
{\textstyle\sum\limits_{I\in G_{N}^{0}}}
%EndExpansion
A_{JI}\left\{  v_{I}-v_{J}\right\}  ,
\end{array}
\right.  \label{Eq_3}%
\end{equation}
for $J\in G_{N}^{0}$, where%
\[
A_{JI}:=\left\{
\begin{array}
[c]{ll}%
1 & \text{if the vertices }J\text{ and }I\text{ are connected}\\
& \\
0 & \text{otherwise.}%
\end{array}
\right.
\]
The matrix $\left[  A_{JI}\right]  _{J,I\in G_{N}^{0}}$ is called the
\ \textit{adjacency matrix }of $\mathcal{G}$. We admit the possibility that
$A_{II}\neq0$, which means that graph $\mathcal{G}$ may have self-loops, we
also admit the possibility that $A_{IJ}\neq A_{JI}$, which means that we do
not require that matrix $\left[  A_{JI}\right]  _{J,I\in G_{N}^{0}}$ be
symmetric. However, in Sections \ref{Section_spectrum_L}\ and
\ref{Section_Turing_criteria}, we will require that $\mathcal{G}$ be an
undirected graph with a symmetric adjacency matrix.

The positive constants $\varepsilon$ and $\varepsilon d$ denote the
diffusivities of $u$ and $v$. The number of edges connecting to vertex $I$ is
\[
\gamma_{I}:=%
%TCIMACRO{\tsum \limits_{J\in G_{N}^{0}}}%
%BeginExpansion
{\textstyle\sum\limits_{J\in G_{N}^{0}}}
%EndExpansion
A_{IJ}\text{.}%
\]
We set
\begin{equation}
\gamma_{_{\mathcal{G}}}:=\max_{I\in G_{N}^{0}}\gamma_{I}. \label{Eq_Gamma_T}%
\end{equation}
For each $J\in G_{N}^{0}$, we can rewrite\ the flux term as%
\[%
%TCIMACRO{\tsum \limits_{I\in G_{N}^{0}}}%
%BeginExpansion
{\textstyle\sum\limits_{I\in G_{N}^{0}}}
%EndExpansion
A_{JI}\left\{  u_{I}-u_{J}\right\}  =%
%TCIMACRO{\tsum \limits_{I\in G_{N}^{0}}}%
%BeginExpansion
{\textstyle\sum\limits_{I\in G_{N}^{0}}}
%EndExpansion
L_{JI}u_{I},
\]
where $L_{JI}=A_{JI}-\gamma_{I}\delta_{JI}$, and $\delta_{JI}$ denotes the
Kronecker delta. The matrix $\left[  L_{JI}\right]  _{J,I\in G_{N}^{0}}$ is
called the \textit{Laplacian matrix of the graph} $\mathcal{G}$. Then, the
system (\ref{Eq_3}) can be rewritten as%
\begin{equation}
\left\{
\begin{array}
[c]{l}%
\frac{\partial u_{J}}{\partial t}=f(u_{J},v_{J})+\varepsilon%
%TCIMACRO{\tsum \limits_{I\in G_{N}^{0}}}%
%BeginExpansion
{\textstyle\sum\limits_{I\in G_{N}^{0}}}
%EndExpansion
L_{JI}u_{I}\\
\\
\frac{\partial v_{J}}{\partial t}=g(u_{J},v_{J})+\varepsilon d%
%TCIMACRO{\tsum \limits_{I\in G_{N}^{0}}}%
%BeginExpansion
{\textstyle\sum\limits_{I\in G_{N}^{0}}}
%EndExpansion
L_{JI}v_{I},
\end{array}
\right.  \label{Eq_4}%
\end{equation}
for $J\in G_{N}^{0}$.

\subsection{A $p$-Adic analogue of system (\ref{Eq_4})}

We now construct a $p$-adic analogue of system (\ref{Eq_4}). We set%
\[
G_{N}:=\mathbb{Z}_{p}/p^{N}\mathbb{Z}_{p}\text{ for }N\geq1\text{.}%
\]
We identify $G_{N}$\ with the set of representatives of the form%
\begin{equation}
I=I_{0}+I_{1}p+\ldots+I_{N-1}p^{N-1}, \label{Eq_I}%
\end{equation}
where the $I_{j}$s are $p$-adic digits. We assume that $G_{N}^{0}\subset
G_{N}$. This implies that the number of vertices $\#G_{N}^{0}$ of
$\mathcal{G}$ must satisfy $\#G_{N}^{0}\leq p^{N}$. From now on, we fix and
$N$ and $p$ such that this inequality holds. There is no a canonical way of
choosing $N$ and $p$. On the other hand, since the elements of the form
(\ref{Eq_I}) belong to $\mathbb{Z}_{p}\smallsetminus p^{N}\mathbb{Z}_{p}$, the
assumption $G_{N}^{0}\subset G_{N}$ gives rise an embedding of $\mathcal{G}$
into $\mathbb{Z}_{p}\smallsetminus p^{N}\mathbb{Z}_{p}$.

We define
\[
\mathcal{K}_{N}=%
%TCIMACRO{\tbigsqcup \limits_{I\in G_{N}^{0}}}%
%BeginExpansion
{\textstyle\bigsqcup\limits_{I\in G_{N}^{0}}}
%EndExpansion
I+p^{N}\mathbb{Z}_{p}\text{.}%
\]
Then $\mathcal{K}_{N}$ is an open compact subset of $\mathbb{Z}_{p}$. We also
define%
\begin{equation}
J_{N}(x,y)=p^{N}%
%TCIMACRO{\tsum \limits_{J\in G_{N}^{0}}}%
%BeginExpansion
{\textstyle\sum\limits_{J\in G_{N}^{0}}}
%EndExpansion%
%TCIMACRO{\tsum \limits_{K\in G_{N}^{0}}}%
%BeginExpansion
{\textstyle\sum\limits_{K\in G_{N}^{0}}}
%EndExpansion
A_{JK}\Omega\left(  p^{N}\left\vert x-J\right\vert _{p}\right)  \Omega\left(
p^{N}\left\vert y-K\right\vert _{p}\right)  \text{,} \label{Eq_Kernel_J_N}%
\end{equation}
$x$, $y\in\mathbb{Q}_{p}$, where $\left[  A_{JI}\right]  _{J,I\in G_{N}^{0}}$
is the adjacency matrix of graph $\mathcal{G}$. Notice that $J_{N}(x,y)$ is a
test function from $\mathcal{D}(\mathcal{K}_{N}\times\mathcal{K}%
_{N},\mathbb{R})$.

We denote by ${\large C}\left(  \mathcal{K}_{N},\mathbb{R}\right)  $ the
vector space of all the continuous real-valued functions on $\mathcal{K}_{N}$
endowed with supremum norm, denoted as $\left\Vert \cdot\right\Vert _{\infty}%
$. Notice that ${\large C}\left(  \mathcal{K}_{N},\mathbb{R}\right)  $ is a
Banach space. We denote by ${\large X}_{N}$, the $\mathbb{R}$-vector space
consisting of all the test functions supported in $\mathcal{K}_{N}$\ having
the form%
\[
\varphi\left(  x\right)  =%
%TCIMACRO{\tsum \limits_{J\in G_{N}^{0}}}%
%BeginExpansion
{\textstyle\sum\limits_{J\in G_{N}^{0}}}
%EndExpansion
\varphi\left(  J\right)  \Omega\left(  p^{N}\left\vert x-J\right\vert
_{p}\right)  \text{,}%
\]
where $\varphi\left(  J\right)  \in\mathbb{R}$. We endow ${\large X}_{N}$ with
\ the $\left\Vert \cdot\right\Vert _{\infty}$-norm. Then $X_{N}$ is a closed
subspace of ${\large C}\left(  \mathcal{K}_{N},\mathbb{R}\right)  $, in
addition,
\[
X_{N}\simeq\left(  \mathbb{R}^{\#G_{N}^{0}},\left\Vert \cdot\right\Vert
_{\infty}\right)  ,\text{ as \ Banach spaces,}%
\]
where $\left\Vert \left(  x_{1},\ldots,x_{^{\#G_{N}^{0}}}\right)  \right\Vert
_{\infty}:=\max\left\{  \left\vert x_{1}\right\vert ,\ldots,\left\vert
x_{^{\#G_{N}^{0}}}\right\vert \right\}  $ for $\left(  x_{1},\ldots
,x_{^{\#G_{N}^{0}}}\right)  \in\mathbb{R}^{\#G_{N}^{0}}$. \ Notice that
\[
\left\{  \Omega\left(  p^{N}\left\vert x-J\right\vert _{p}\right)  \right\}
_{J\in G_{N}^{0}}%
\]
is a basis of $X_{N}$.

We now define the operator%
\[
\boldsymbol{L}_{N}\varphi\left(  x\right)  =%
%TCIMACRO{\tint \limits_{\mathcal{K}_{N}}}%
%BeginExpansion
{\textstyle\int\limits_{\mathcal{K}_{N}}}
%EndExpansion
\left\{  \varphi\left(  y\right)  -\varphi\left(  x\right)  \right\}
J_{N}(x,y)dy\text{, for }\varphi\in X_{N}\text{.}%
\]
Then $\boldsymbol{L}_{N}:X_{N}\rightarrow X_{N}$ is a linear bounded operator
satisfying%
\[
\left\Vert \boldsymbol{L}_{N}\right\Vert \leq2\gamma_{_{\mathcal{G}}},
\]
see (\ref{Eq_Gamma_T}). In addition,%
\begin{gather*}
\boldsymbol{L}_{N}\Omega\left(  p^{N}\left\vert x-I\right\vert _{p}\right)  =%
%TCIMACRO{\tsum \limits_{J\in G_{N}^{0}}}%
%BeginExpansion
{\textstyle\sum\limits_{J\in G_{N}^{0}}}
%EndExpansion
A_{JI}\Omega\left(  p^{N}\left\vert x-J\right\vert _{p}\right)  -\left(
%TCIMACRO{\tsum \limits_{K\in G_{N}^{0}}}%
%BeginExpansion
{\textstyle\sum\limits_{K\in G_{N}^{0}}}
%EndExpansion
A_{IK}\right)  \Omega\left(  p^{N}\left\vert x-I\right\vert _{p}\right) \\
=%
%TCIMACRO{\tsum \limits_{J\in G_{N}^{0}}}%
%BeginExpansion
{\textstyle\sum\limits_{J\in G_{N}^{0}}}
%EndExpansion
A_{JI}\Omega\left(  p^{N}\left\vert x-J\right\vert _{p}\right)  -\gamma
_{I}\Omega\left(  p^{N}\left\vert x-I\right\vert _{p}\right) \\
=%
%TCIMACRO{\tsum \limits_{J\in G_{N}^{0}}}%
%BeginExpansion
{\textstyle\sum\limits_{J\in G_{N}^{0}}}
%EndExpansion
\left\{  A_{JI}-\gamma_{I}\delta_{JI}\right\}  \Omega\left(  p^{N}\left\vert
x-J\right\vert _{p}\right)  .
\end{gather*}
Consequently, operator $\boldsymbol{L}_{N}:X_{N}\rightarrow X_{N}$ is
represented by the matrix
\begin{equation}
\left[  A_{JI}-\gamma_{I}\delta_{JI}\right]  _{J,I\in G_{N}^{0}}.
\label{Eq_Matrix}%
\end{equation}

\begin{notation}
We set $\mathbb{R}_{+}:=\left[  0,+\infty\right)  $.
\end{notation}

Consider the following system:%
\begin{equation}
\left\{
\begin{array}
[c]{l}%
u^{\left(  N\right)  }\left(  \cdot,t\right)  ,v^{\left(  N\right)  }\left(
\cdot,t\right)  \in{\large C}^{1}(\mathbb{R}_{+},X_{N});\\
\\
\frac{\partial u^{\left(  N\right)  }\left(  x,t\right)  }{\partial
t}=f(u^{\left(  N\right)  }\left(  x,t\right)  ,v^{\left(  N\right)  }\left(
x,t\right)  )+\varepsilon\boldsymbol{L}_{N}u^{\left(  N\right)  }\left(
x,t\right) \\
\\
\frac{\partial v^{\left(  N\right)  }\left(  x,t\right)  }{\partial
t}=g(u^{\left(  N\right)  }\left(  x,t\right)  ,v^{\left(  N\right)  }\left(
x,t\right)  )+\varepsilon d\boldsymbol{L}_{N}v^{\left(  N\right)  }\left(
x,t\right)  .
\end{array}
\right.  \label{Eq_5}%
\end{equation}
By using that%
\[%
\begin{array}
[c]{l}%
u^{\left(  N\right)  }\left(  x,t\right)  =%
%TCIMACRO{\tsum \limits_{L\in G_{N}^{0}}}%
%BeginExpansion
{\textstyle\sum\limits_{L\in G_{N}^{0}}}
%EndExpansion
u^{\left(  N\right)  }\left(  L,t\right)  \Omega\left(  p^{N}\left\vert
x-L\right\vert _{p}\right)  ,\\
\\
v^{\left(  N\right)  }\left(  x,t\right)  =%
%TCIMACRO{\tsum \limits_{L\in G_{N}^{0}}}%
%BeginExpansion
{\textstyle\sum\limits_{L\in G_{N}^{0}}}
%EndExpansion
v^{\left(  N\right)  }\left(  L,t\right)  \Omega\left(  p^{N}\left\vert
x-L\right\vert _{p}\right)  ,
\end{array}
\]
where $u^{\left(  N\right)  }\left(  L,t\right)  $, $v^{\left(  N\right)
}\left(  L,t\right)  $ belong to ${\large C}^{1}(\mathbb{R}_{+})$, and
(\ref{Eq_Matrix}), and the fact that for any $h:\mathbb{R}^{2}\rightarrow
\mathbb{R}$,%
\[
h(u^{\left(  N\right)  }\left(  x,t\right)  ,v^{\left(  N\right)  }\left(
x,t\right)  )=%
%TCIMACRO{\tsum \limits_{L\in G_{N}^{0}}}%
%BeginExpansion
{\textstyle\sum\limits_{L\in G_{N}^{0}}}
%EndExpansion
h\left(  u^{\left(  N\right)  }\left(  L,t\right)  ,v^{\left(  N\right)
}\left(  L,t\right)  \right)  \Omega\left(  p^{N}\left\vert x-L\right\vert
_{p}\right)  ,
\]
we\ obtain that the system (\ref{Eq_5}) can be rewritten as%
\begin{equation}
\left\{
\begin{array}
[c]{l}%
\frac{\partial u^{\left(  N\right)  }\left(  I,t\right)  }{\partial t}=\\
\\
f(u^{\left(  N\right)  }\left(  I,t\right)  ,v^{\left(  N\right)  }\left(
I,t\right)  )+\varepsilon%
%TCIMACRO{\tsum \limits_{L\in G_{N}^{0}}}%
%BeginExpansion
{\textstyle\sum\limits_{L\in G_{N}^{0}}}
%EndExpansion
\left\{  A_{IL}-\gamma_{I}\delta_{IL}\right\}  u^{\left(  N\right)  }\left(
L,t\right) \\
\\
\frac{\partial v^{\left(  N\right)  }\left(  I,t\right)  }{\partial t}=\\
\\
g(u^{\left(  N\right)  }\left(  I,t\right)  ,v^{\left(  N\right)  }\left(
I,t\right)  )+\varepsilon d%
%TCIMACRO{\tsum \limits_{L\in G_{N}^{0}}}%
%BeginExpansion
{\textstyle\sum\limits_{L\in G_{N}^{0}}}
%EndExpansion
\left\{  A_{IL}-\gamma_{I}\delta_{IL}\right\}  u^{\left(  N\right)  }\left(
L,t\right)  ,
\end{array}
\right.  \label{Eq_6}%
\end{equation}
for $I\in G_{N}^{0}$, which is exactly the system (\ref{Eq_4}).

In order to obtain a $p$-adic continuous version of the system (\ref{Eq_5}),
we first notice that \ for $\varphi\in{\large C}(\mathcal{K}_{N},\mathbb{R})$,
the function%
\begin{equation}
\boldsymbol{L}\varphi\left(  x\right)  =%
%TCIMACRO{\tint \limits_{\mathcal{K}_{N}}}%
%BeginExpansion
{\textstyle\int\limits_{\mathcal{K}_{N}}}
%EndExpansion
\left\{  \varphi\left(  y\right)  -\varphi\left(  x\right)  \right\}
J_{N}\left(  x,y\right)  dy \label{Eq_Op_T}%
\end{equation}
belongs to ${\large C}(\mathcal{K}_{N},\mathbb{R})$, and that operator
$\boldsymbol{L}$ is a linear continuous operator satisfying
\[
\left\Vert \boldsymbol{L}\right\Vert \leq2\gamma_{_{\mathcal{G}}}\text{ and
}\boldsymbol{L}_{N}=\boldsymbol{L}\mid_{X_{N}}.
\]
By using the fact that operator $\boldsymbol{L}$ is an extension of
$\boldsymbol{L}_{N}$, it is natural to postulate\ that the system
\begin{equation}
\left\{
\begin{array}
[c]{l}%
u\left(  \cdot,t\right)  ,v\left(  \cdot,t\right)  \in{\large C}%
^{1}(\mathbb{R}_{+},{\large C}\left(  \mathcal{K}_{N},\mathbb{R}\right)  );\\
\\
\frac{\partial u\left(  x,t\right)  }{\partial t}=f(u,v)+\varepsilon
\boldsymbol{L}u\left(  x,t\right) \\
\\
\frac{\partial v\left(  x,t\right)  }{\partial t}=g(u,v)+\varepsilon
d\boldsymbol{L}v\left(  x,t\right)  ,
\end{array}
\right.  \label{Eq_7}%
\end{equation}
is a `$p$-adic analog' of the system (\ref{Eq_5}). In the following sections,
we will establish in a mathematically rigorous way this assertion.

\section{$p$-Adic diffusion and self-similarity in $X_{\infty}$}

In this section we study the following Cauchy problem:%
\begin{equation}
\left\{
\begin{array}
[c]{l}%
h\left(  x,t\right)  \in{\large C}^{1}(\left(  0,\infty\right)  ,{\large C}%
(\mathcal{K}_{N},\mathbb{R}));\\
\\
\frac{\partial h\left(  x,t\right)  }{\partial t}=\varepsilon\boldsymbol{L}%
h\left(  x,t\right)  \text{, \ }x\in\mathcal{K}_{N}\text{, }t>0;\\
\\
h\left(  x,0\right)  =h_{0}(x)\in{\large C}(\mathcal{K}_{N},\mathbb{R}),
\end{array}
\right.  \label{Eq_10}%
\end{equation}
where $\boldsymbol{L}:{\large C}(\mathcal{K}_{N},\mathbb{R})\rightarrow
{\large C}(\mathcal{K}_{N},\mathbb{R})$ is the operator defined in
(\ref{Eq_Op_T}). In this section we show that the equation (\ref{Eq_10}) is a
`$p$-adic heat equation,' which means that the corresponding semigroup is
Feller, and consequently there is a $p$-adic diffusion process in
$\mathcal{K}_{N}$ attached to the differential equation (\ref{Eq_10}).

\subsection{Yosida-Hille-Ray theorem and Feller semigroups}

We formulate Yosida-Hille-Ray Theorem in the setting of $(\mathbb{Q}%
_{p},\left\vert \cdot\right\vert _{p})$. For a general discussion the reader
may consult \cite[Chapter 4, Theorem 2.2]{E-K}.

A semigroup $\{\boldsymbol{Q}(t)\}_{t\geq0}$ on ${\large C}\left(
\mathcal{K}_{N},\mathbb{R}\right)  $ is said to be \textit{positive} if
$\boldsymbol{Q}(t)$ is a positive operator for each $t\geq0$, i.e. it maps
non-negative functions to non-negative functions. An operator $(\boldsymbol{A}%
,Dom(\boldsymbol{A}))$ on ${\large C}\left(  \mathcal{K}_{N},\mathbb{R}%
\right)  $ is said to satisfy the \textit{positive maximum principle} if
whenever $h\in Dom(\boldsymbol{A})\subseteq{\large C}\left(  \mathcal{K}%
_{N},\mathbb{R}\right)  $, $x_{0}\in\mathbb{Q}_{p}$, and $\sup_{x\in\mathbf{%
%TCIMACRO{\U{211a} }%
%BeginExpansion
\mathbb{Q}
%EndExpansion
}_{p}}h(x)=h(x_{0})\geq0$ we have $\boldsymbol{A}h(x_{0})\leq0$.

We recall that every linear operator on ${\large C}\left(  \mathcal{K}%
_{N},\mathbb{R}\right)  $ satisfying the positive maximum principle is
dissipative, see e.g. \cite[Chapter 4, Lemma 2.1]{E-K}.

\begin{theorem}
[Hille-Yosida-Ray Theorem]\label{Hille-Yosida-Ray-Theorem}Let $(\boldsymbol{A}%
,Dom(\boldsymbol{A}))$ be a linear operator on ${\large C}\left(
\mathcal{K}_{N},\mathbb{R}\right)  $. The closure $\overline{\boldsymbol{A}}$
of $\boldsymbol{A}$ on ${\large C}\left(  \mathcal{K}_{N},\mathbb{R}\right)  $
is single-valued and generates a strongly continuous, positive, contraction
semigroup $\{\boldsymbol{Q}_{t}\}_{t\geq0}$ on ${\large C}\left(
\mathcal{K}_{N},\mathbb{R}\right)  $ if and only if:

\noindent(i) $Dom(\boldsymbol{A})$ is dense in ${\large C}\left(
\mathcal{K}_{N},\mathbb{R}\right)  $;

\noindent(ii) $\boldsymbol{A}$ satisfies the positive maximum principle;

\noindent(iii) Rank$(\eta\boldsymbol{I}-\boldsymbol{A})$ is dense in
${\large C}\left(  \mathcal{K}_{N},\mathbb{R}\right)  $ for some $\eta>0$.
\end{theorem}

\begin{definition}
A family of bounded linear operators $\boldsymbol{P}_{t}:{\large C}\left(
\mathcal{K}_{N},\mathbb{R}\right)  \rightarrow{\large C}\left(  \mathcal{K}%
_{N},\mathbb{R}\right)  $ is called a Feller semigroup if

\noindent(i) $\boldsymbol{P}_{s+t}=\boldsymbol{P}_{s}\boldsymbol{P}_{t}$ and
$\boldsymbol{P}_{0}=I$;

\noindent(ii) $\lim_{t\rightarrow0}||\boldsymbol{P}_{t}h-h||_{\infty}=0$ for
any $h\in{\large C}\left(  \mathcal{K}_{N},\mathbb{R}\right)  $;

\noindent(iii) $0\leq\boldsymbol{P}_{t}h\leq1$ if $0\leq h\leq1$, with
$h\in{\large C}\left(  \mathcal{K}_{N},\mathbb{R}\right)  $ and for any
$t\geq0$.
\end{definition}

Therefore, Theorem \ref{Hille-Yosida-Ray-Theorem} characterizes the Feller
semigroups. More precisely, if $(\boldsymbol{A},Dom(\boldsymbol{A}))$
satisfies Theorem \ref{Hille-Yosida-Ray-Theorem}, then $\boldsymbol{A}$ has a
closed extension which is the generator of a Feller semigroup.

\begin{lemma}
\label{Lemma1}The operator $\varepsilon\boldsymbol{L}$ generates a strongly
continuous, positive, contraction semigroup $\left\{  e^{t\varepsilon
\boldsymbol{L}}\right\}  _{t\geq0}$ on ${\large C}\left(  \mathcal{K}%
_{N},\mathbb{R}\right)  $.
\end{lemma}

\begin{proof}
We first verify the conditions given in Theorem \ref{Hille-Yosida-Ray-Theorem}%
. The verification of conditions (i)-(ii) is straightforward. The third
condition in Theorem \ref{Hille-Yosida-Ray-Theorem} is equivalent to the
existence of a $\eta>0$ such that for any $h\in{\large C}\left(
\mathcal{K}_{N},\mathbb{R}\right)  $ the equation
\begin{equation}
\eta u-\varepsilon\boldsymbol{L}u=h \label{EQ_10AAA}%
\end{equation}
has a solution $u\in{\large C}\left(  \mathcal{K}_{N},\mathbb{R}\right)  $.
Set
\[
g(x):=%
%TCIMACRO{\tint \limits_{\mathcal{K}_{N}}}%
%BeginExpansion
{\textstyle\int\limits_{\mathcal{K}_{N}}}
%EndExpansion
J_{N}\left(  x,y\right)  dy=%
%TCIMACRO{\tsum \limits_{I\in G_{N}^{0}}}%
%BeginExpansion
{\textstyle\sum\limits_{I\in G_{N}^{0}}}
%EndExpansion
\gamma_{I}\Omega\left(  p^{N}\left\vert x-I\right\vert _{p}\right)  .
\]
Then $g(x)>0$, $\left\Vert g\right\Vert _{\infty}\leq\gamma_{\mathcal{G}}$ and
$g\in{\large C}\left(  \mathcal{K}_{N},\mathbb{R}\right)  $. We now rewrite
(\ref{EQ_10AAA}) as follows:%
\[
u(x)-\varepsilon%
%TCIMACRO{\tint \limits_{\mathcal{K}_{N}}}%
%BeginExpansion
{\textstyle\int\limits_{\mathcal{K}_{N}}}
%EndExpansion
u\left(  y\right)  \left\{  \frac{J_{N}\left(  x,y\right)  }{\eta+\varepsilon
g(x)}\right\}  dy=\frac{h(x)}{\eta+\varepsilon g(x)},
\]
where $\frac{h(x)}{\eta+\varepsilon g(x)}\in{\large C}\left(  \mathcal{K}%
_{N},\mathbb{R}\right)  $. Now, the operator $\boldsymbol{T}:{\large C}\left(
\mathcal{K}_{N},\mathbb{R}\right)  \rightarrow{\large C}\left(  \mathcal{K}%
_{N},\mathbb{R}\right)  $ defined as
\[
\boldsymbol{T}u\left(  x\right)  =\varepsilon%
%TCIMACRO{\tint \limits_{\mathcal{K}_{N}}}%
%BeginExpansion
{\textstyle\int\limits_{\mathcal{K}_{N}}}
%EndExpansion
u\left(  y\right)  \left\{  \frac{J_{N}\left(  x,y\right)  }{\eta+\varepsilon
g(x)}\right\}  dy
\]
satisfies $\left\Vert \boldsymbol{T}\right\Vert \leq\frac{\varepsilon
\gamma_{\mathcal{G}}}{\eta}$. By taking $\eta>\varepsilon\gamma_{\mathcal{G}}%
$, operator $\boldsymbol{I}-\boldsymbol{T}$ has an inverse on ${\large C}%
\left(  \mathcal{K}_{N},\mathbb{R}\right)  $. Notice that $\eta$ is
independent of $h$.

Therefore $\boldsymbol{L}=\overline{\boldsymbol{L}}$ generates a semigroup
$\{\boldsymbol{Q}_{t}\}_{t\geq0}$ having the properties announced in Theorem
\ref{Hille-Yosida-Ray-Theorem}. On the other hand, since $\varepsilon
\boldsymbol{L}$ is a linear bounded operator on a Banach space, $\left\{
e^{t\varepsilon\boldsymbol{L}}\right\}  _{t\geq0}$ is a uniformly continuous
semigroup, and by using that the infinitesimal generators of $\left\{
e^{t\varepsilon\boldsymbol{L}}\right\}  _{t\geq0}$\ and $\{\boldsymbol{Q}%
_{t}\}_{t\geq0}$ agree, we have $e^{t\varepsilon\boldsymbol{L}}=\boldsymbol{Q}%
_{t}$ for $t\geq0$, see e.g. \cite[Theorems 1.2 and 1.3]{Pazy}.
\end{proof}

\begin{theorem}
\label{Theorem1} There exists a probability measure $p_{t}\left(
x,\cdot\right)  $, $t\geq0$, $x\in\mathcal{K}_{N}$, on the Borel $\sigma
$-algebra of $\mathcal{K}_{N}$, such that the Cauchy problem (\ref{Eq_10}) has
a unique solution of the form%
\[
h(x,t)=\int\limits_{\mathcal{K}_{N}}h_{0}(y)p_{t}\left(  x,dy\right)  .
\]
In addition, $p_{t}\left(  x,\cdot\right)  $ is the transition function of a
Markov process $\mathfrak{X}$ whose paths are right continuous and have no
discontinuities other than jumps.
\end{theorem}

\begin{proof}
By Lemma \ref{Lemma1}, $\left\{  e^{t\varepsilon\boldsymbol{L}}\right\}
_{t\geq0}$ is a Feller semigroup on ${\large C}\left(  \mathcal{K}%
_{N},\mathbb{R}\right)  $. By using the correspondence between Feller
semigroups and transition functions, there exists a uniformly stochastically
continuous $C_{0}$-transition function $p_{t}\left(  x,dy\right)  $ satisfying
condition $(L)$, see \cite[Theorem 2.10]{Taira}, such that
\[
e^{t\boldsymbol{L}}h_{0}\left(  x\right)  =\int\limits_{\mathcal{K}_{N}}%
h_{0}(y)p_{t}\left(  x,dy\right)  \text{ for }h_{0}\in{\large C}\left(
\mathcal{K}_{N},\mathbb{R}\right)  \text{,}%
\]
see e.g. \cite[Theorem 2.15]{Taira}. Now, by using the correspondence between
transition functions and Markov processes, there exists a strong Markov
process $\mathfrak{X}$ whose paths are right continuous and have no
discontinuities other than jumps, see e.g. \cite[Theorem 2.12]{Taira}.
\end{proof}

\subsection{\label{Sect_selfsimilarity_1} Self-similarity}

Theorem \ref{Theorem1} can be easily extended to a larger class of operators.
For instance, take $J(x,y)\in L^{\infty}\left(  \mathcal{K}_{N}\times
\mathcal{K}_{N},\mathbb{R}\right)  $, $J(x,y)\geq0$ and $\lambda\geq1$, and
set
\begin{equation}
\varepsilon L_{\lambda}\varphi\left(  x\right)  =\varepsilon\int
\limits_{\mathcal{K}_{N}}\left\{  \varphi(y)-\lambda\varphi\left(  x\right)
\right\}  J\left(  x,y\right)  dy. \label{Eq_ope_L_lambda}%
\end{equation}
Then $L_{\lambda}:{\large C}\left(  \mathcal{K}_{N},\mathbb{R}\right)
\rightarrow{\large C}\left(  \mathcal{K}_{N},\mathbb{R}\right)  $ is a linear
bounded operator, with $\left\Vert \varepsilon L_{\lambda}\right\Vert
\leq\left(  1+\lambda\right)  \varepsilon\left\Vert J\right\Vert _{\infty}$.
Notice that condition $\lambda\geq1$ is essential to assure that operator
$\varepsilon L_{\lambda}$ satisfies the positive maximum principle. Theorem
\ref{Theorem1} holds for the following Cauchy problem:%
\begin{equation}
\left\{
\begin{array}
[c]{l}%
h\left(  x,t\right)  \in{\large C}^{1}(\left(  0,\infty\right)  ,{\large C}%
(\mathcal{K}_{N},\mathbb{R}));\\
\\
\frac{\partial h\left(  x,t\right)  }{\partial t}=\varepsilon\boldsymbol{L}%
_{\lambda}h\left(  x,t\right)  \text{, \ }x\in\mathcal{K}_{N}\text{, }t>0;\\
\\
h\left(  x,0\right)  =h_{0}(x)\in{\large C}(\mathcal{K}_{N},\mathbb{R}).
\end{array}
\right.  \label{Eq_Family_p_adic_eq}%
\end{equation}
Then, for a fixed $J(x,y)\in L^{\infty}\left(  \mathcal{K}_{N}\times
\mathcal{K}_{N},\mathbb{R}\right)  $, $J(x,y)\geq0$,
(\ref{Eq_Family_p_adic_eq}) is a family of $p$-adic diffusion equations
parametrized by the set%
\[
\mathcal{P}:=\left\{  \left(  \varepsilon,\lambda\right)  \in\mathbb{R}%
_{+}^{2};\text{ }\varepsilon>0\text{, }\lambda\geq1\right\}  .
\]
We identify the family (\ref{Eq_Family_p_adic_eq}) with the set $\mathcal{P}$.
Now, for $\sigma\in\left(  0,1\right]  $, we define the mapping:%
\[%
\begin{array}
[c]{llll}%
S_{\sigma}: & \mathcal{P} & \rightarrow & \mathcal{P}\\
&  &  & \\
& \left(  \varepsilon,\lambda\right)  & \rightarrow & \left(  \sigma
\varepsilon,\sigma^{-1}\lambda\right)  .
\end{array}
\]
The set of all $S_{\sigma}$\ for $\sigma\in\left(  0,1\right]  $ is naturally
a monoid (denoted as $\mathcal{S}_{\mathcal{P}}$), under the composition of
functions. Therefore, we have established the following result:

\begin{theorem}
\label{Theorem1_fractal}The family $\mathcal{P}$ is invariant under the action
of the monoid $\mathcal{S}_{\mathcal{P}}$.
\end{theorem}

\section{Discretizations}

\subsection{Some additional function spaces and operators}

Let $M$ be a positive integer satisfying $M\geq N$. We fix a system of
representatives $I_{j}$s for the quotient%
\[
G_{I}^{M}:=\left(  I+p^{N}\mathbb{Z}_{p}\right)  /p^{M}\mathbb{Z}_{p}.
\]
This means that%
\[
B_{-N}(I)=%
%TCIMACRO{\tbigsqcup \limits_{I_{j}\in G_{I}^{M}}}%
%BeginExpansion
{\textstyle\bigsqcup\limits_{I_{j}\in G_{I}^{M}}}
%EndExpansion
B_{-M}\left(  I_{j}\right)  ,
\]
where $B_{-L}(J)=\left\{  x\in\mathbb{Q}_{p};\left\vert x-J\right\vert
_{p}\leq p^{-L}\right\}  $. Now, we set
\[
G_{N}^{M}:=%
%TCIMACRO{\tbigsqcup \limits_{I\in G_{N}^{0}}}%
%BeginExpansion
{\textstyle\bigsqcup\limits_{I\in G_{N}^{0}}}
%EndExpansion
G_{I}^{M}.
\]

\begin{remark}
Notice that $G_{N}^{N}=G_{N}^{0}$, and thus $G_{N}^{N}$ can be identified with
$\mathcal{G}$.
\end{remark}

Since $\mathcal{K}_{N}$ is the disjoint union of the $I+p^{N}\mathbb{Z}_{p}$,
for $I\in G_{N}^{0}$,
\[
\mathcal{K}_{N}=%
%TCIMACRO{\tbigsqcup \limits_{I\in G_{N}^{0}}}%
%BeginExpansion
{\textstyle\bigsqcup\limits_{I\in G_{N}^{0}}}
%EndExpansion
\text{ }%
%TCIMACRO{\tbigsqcup \limits_{I_{j}\in G_{I}^{M}}}%
%BeginExpansion
{\textstyle\bigsqcup\limits_{I_{j}\in G_{I}^{M}}}
%EndExpansion
I_{j}+p^{M}\mathbb{Z}_{p}=%
%TCIMACRO{\tbigsqcup \limits_{I_{j}\in G_{N}^{M}}}%
%BeginExpansion
{\textstyle\bigsqcup\limits_{I_{j}\in G_{N}^{M}}}
%EndExpansion
I_{j}+p^{M}\mathbb{Z}_{p}.
\]
We set $X_{M}$, $M\geq N$, to be the $\mathbb{R}$-vector space of all the test
functions supported in $\mathcal{K}_{N}$ of the form
\[
\varphi\left(  x\right)  =%
%TCIMACRO{\tsum \limits_{I_{j}\in G_{N}^{M}}}%
%BeginExpansion
{\textstyle\sum\limits_{I_{j}\in G_{N}^{M}}}
%EndExpansion
\varphi\left(  I_{j}\right)  \Omega\left(  p^{M}\left\vert x-I_{j}\right\vert
_{p}\right)  \text{, }\varphi\left(  I_{j}\right)  \in\mathbb{R}\text{,}%
\]
endowed with the $\left\Vert \cdot\right\Vert _{\infty}$-norm. This is a real
Banach space. For our convenience, from now on, we set $X_{\infty}%
:={\large C}\left(  \mathcal{K}_{N},\mathbb{R}\right)  $ endowed with the
$\left\Vert \cdot\right\Vert _{\infty}$-norm. This is also a real Banach space.

For $M\geq N$, we define $\boldsymbol{P}_{M}\in\mathfrak{B}(X_{\infty},X_{M}%
)$, the bounded linear operators from $X_{\infty}$

into $X_{M}$, as%
\begin{equation}
\boldsymbol{P}_{M}\varphi\left(  x\right)  =%
%TCIMACRO{\tsum \limits_{I_{j}\in G_{N}^{M}}}%
%BeginExpansion
{\textstyle\sum\limits_{I_{j}\in G_{N}^{M}}}
%EndExpansion
\varphi\left(  I_{j}\right)  \Omega\left(  p^{M}\left\vert x-I_{j}\right\vert
_{p}\right)  . \label{definition_P_M}%
\end{equation}
We denote by $\boldsymbol{E}_{M}:$ $X_{M}\hookrightarrow X_{\infty}$, $M\geq
N$, the natural continuous embedding. Notice that $\left\Vert \boldsymbol{E}%
_{M}\right\Vert \leq1$, and that $\boldsymbol{P}_{M}\boldsymbol{E}_{M}%
\varphi=\varphi$ for $\varphi\in X_{M}$, $M\geq N$.

\begin{notation}
Whenever this is possible, we will omit in our formulas the operator
$\boldsymbol{E}_{M}$, instead we will use the fact that $X_{M}\hookrightarrow
X_{\infty}$, $M\geq N$.
\end{notation}

\begin{lemma}
\label{Lemma0}With the above notation, the following assertions hold true: (i)
$\left\Vert \boldsymbol{P}_{M}\right\Vert \leq1$; (ii) $\lim_{M\rightarrow
\infty}\left\Vert \boldsymbol{P}_{M}\varphi-\varphi\right\Vert _{\infty}=0$
for $\varphi\in X_{\infty}$.
\end{lemma}

\begin{proof}
The first part is straightforward. The second part is a consequence of the
following two facts: first, $\mathcal{D}(\mathcal{K}_{N},\mathbb{R})$ is dense
in $X_{\infty}$ in the $\left\Vert \cdot\right\Vert _{\infty}$-norm; second,
given a function $\varphi\in\mathcal{D}(\mathcal{K}_{N},\mathbb{R})$, there
exists an $M$ sufficiently large such that $\varphi\in X_{M}$. Alternatively,
the reader may see \cite[Lemma 1]{Zuniga-Nonlinear}.
\end{proof}

We now consider the real Banach spaces $X_{\infty}\oplus X_{\infty}$,
$X_{M}\oplus X_{M}$ for $M\geq N$, endowed with the norm $\left\Vert u\oplus
v\right\Vert :=\max\left\{  \left\Vert u\right\Vert _{\infty},\left\Vert
v\right\Vert _{\infty}\right\}  $. We will identify $u\oplus v$ with the
column vector $\left[
\begin{array}
[c]{l}%
u\\
v
\end{array}
\right]  $.

\subsection{Conditions on the nonlinearity}

With respect to the nonlinearity we assume the following. We fix $a$,
$b\in\mathbb{R}$, with $a<b$, and assume that%
\begin{equation}
\left\{
\begin{array}
[c]{cc}%
\text{(i)} & f,g:\left(  a,b\right)  \times\left(  a,b\right)  \rightarrow
\mathbb{R}\text{;}\\
& \\
\text{(ii)} & f,g\in C^{1}\left(  \left(  a,b\right)  \times\left(
a,b\right)  \right)  \text{;}\\
& \\
\text{(iii)} & \nabla f\left(  x,y\right)  \neq0\text{ and }\nabla g\left(
x,y\right)  \neq0\text{ for any }\left(  x,y\right)  \in\left(  a,b\right)
\times\left(  a,b\right)  \text{.}%
\end{array}
\right.  \tag{Hypothesis 1}%
\end{equation}
Now we define%
\begin{equation}
U=\left\{  v\in X_{\infty};a<v\left(  x\right)  <b\text{ for any }%
x\in\mathcal{K}_{N}\right\}  . \label{Definition_set_U}%
\end{equation}
Notice that $U$ is an open set in $X_{\infty}$. Indeed, take $\delta>0$
sufficiently small and $v\in U$, if%
\[
h\in B\left(  v,\delta\right)  =\left\{  h\in X_{\infty};\left\Vert
v-h\right\Vert _{\infty}<\delta\right\}  ,
\]
then
\[
a<-\delta+\min_{x\in\mathcal{K}_{N}}v(x)<h(x)<\delta+\max_{x\in\mathcal{K}%
_{N}}v(x)<b,
\]
for $\delta$ sufficiently small.

By $\left[
\begin{array}
[c]{l}%
f(u,v)\\
g(u,v)
\end{array}
\right]  $, with $u\oplus v\in U\oplus U$, we mean the mapping%
\begin{equation}%
\begin{array}
[c]{llll}%
\left[
\begin{array}
[c]{l}%
f\\
g
\end{array}
\right]  : & U\oplus U & \rightarrow & \mathbb{R\oplus R}\\
&  &  & \\
& u\oplus v & \rightarrow & f(u,v)\mathbb{\oplus}g(u,v).
\end{array}
\label{Definition_map_f_g}%
\end{equation}

\subsection{Two Cauchy problems}

We denote by $\varepsilon\boldsymbol{L}\left[
\begin{array}
[c]{ll}%
1 & 0\\
0 & d
\end{array}
\right]  $ the operator \ acting on $X_{\infty}\oplus X_{\infty}$ as%
\[
\varepsilon\boldsymbol{L}\left[
\begin{array}
[c]{ll}%
1 & 0\\
0 & d
\end{array}
\right]  \left[
\begin{array}
[c]{l}%
u\\
v
\end{array}
\right]  =\left[
\begin{array}
[c]{l}%
\varepsilon\boldsymbol{L}u\\
\varepsilon d\boldsymbol{L}v
\end{array}
\right]  .
\]
In the next sections, we will study the following Cauchy problem, for some
$\tau>0$ fixed:%
\begin{equation}
\left\{
\begin{array}
[c]{l}%
\frac{\partial}{\partial t}\left[
\begin{array}
[c]{l}%
u\left(  t\right) \\
\\
v\left(  t\right)
\end{array}
\right]  =\left[
\begin{array}
[c]{l}%
f(u\left(  t\right)  ,v\left(  t\right)  )\\
\\
g(u\left(  t\right)  ,v\left(  t\right)  )
\end{array}
\right]  +\left[
\begin{array}
[c]{l}%
\varepsilon\boldsymbol{L}u\left(  t\right) \\
\\
\varepsilon d\boldsymbol{L}v\left(  t\right)
\end{array}
\right]  \text{,}\\
\\
t\in\left[  0,\tau\right)  \text{, }x\in\mathcal{K}_{N};\\
\\
u\left(  0\right)  \oplus v\left(  0\right)  =u_{0}\oplus v_{0}\in U\oplus U.
\end{array}
\right.  \label{Eq_14}%
\end{equation}
Also, we will study the Cauchy problem for the following discretization of
(\ref{Eq_14}), with $\boldsymbol{L}_{M}=\boldsymbol{L}\mid_{X_{M}}$:%
\begin{equation}
\left\{
\begin{array}
[c]{l}%
\frac{\partial}{\partial t}\left[
\begin{array}
[c]{l}%
u^{\left(  M\right)  }\left(  t\right) \\
\\
v^{\left(  M\right)  }\left(  t\right)
\end{array}
\right]  =\left[
\begin{array}
[c]{l}%
P_{M}f(E_{M}u^{\left(  M\right)  }\left(  t\right)  ,E_{M}v^{\left(  M\right)
}\left(  t\right)  )\\
\\
P_{M}g(E_{M}u^{\left(  M\right)  }\left(  t\right)  ,E_{M}v^{\left(  M\right)
}\left(  t\right)  )
\end{array}
\right]  +\left[
\begin{array}
[c]{l}%
\varepsilon\boldsymbol{L}_{M}u^{\left(  M\right)  }\left(  t\right) \\
\\
\varepsilon d\boldsymbol{L}_{M}v^{\left(  M\right)  }\left(  t\right)
\end{array}
\right]  ,\\
\\
t\in\left[  0,\tau\right)  \text{, }x\in\mathcal{K}_{N};\\
\\
u^{\left(  M\right)  }\left(  0\right)  \oplus v^{\left(  M\right)  }\left(
0\right)  \in U\cap X_{M}\oplus U\cap X_{M}.
\end{array}
\right.  \label{Eq_15}%
\end{equation}

\begin{remark}
Notice that
\begin{equation}
\left[
\begin{array}
[c]{l}%
P_{M}f(E_{M}u^{\left(  M\right)  }\left(  t\right)  ,E_{M}v^{\left(  M\right)
}\left(  t\right)  )\\
\\
P_{M}g(E_{M}u^{\left(  M\right)  }\left(  t\right)  ,E_{M}v^{\left(  M\right)
}\left(  t\right)  )
\end{array}
\right]  =\left[
\begin{array}
[c]{l}%
f(u^{\left(  M\right)  }\left(  t\right)  ,v^{\left(  M\right)  }\left(
t\right)  )\\
\\
g(u^{\left(  M\right)  }\left(  t\right)  ,v^{\left(  M\right)  }\left(
t\right)  )
\end{array}
\right]  . \label{Eq_nota}%
\end{equation}

\end{remark}

\section{$p$-Adic diffusion in $X_{M}$ and self-similarity}

We first observe that any function $h:G_{N}^{M}\rightarrow\mathbb{R}$,
$I_{j}\rightarrow h\left(  I_{j}\right)  $ can be uniquely identified with an
element of $X_{M}$:%
\[
h(x)=\sum_{I_{j}\in G_{N}^{M}}h\left(  I_{j}\right)  \Omega\left(
p^{M}\left\vert x-I_{j}\right\vert _{p}\right)  \text{, }x\in\mathcal{K}_{N}.
\]
In this way, $X_{M}$ is naturally the space of continuous real-valued
functions on$\ G_{N}^{M}$.

\begin{remark}
\label{Nota_Lemma1}Lemma \ref{Lemma1} holds true in $X_{M}$, $M\geq N$, i.e.
the operator $\varepsilon\boldsymbol{L}_{M}$ generates\ a strongly continuous,
positive, contraction\ semigroup $\left\{  e^{t\varepsilon\boldsymbol{L}_{M}%
}\right\}  _{t\geq0}$ on ${\large X}_{M}$.
\end{remark}

Now, by Theorem \ref{Theorem1}, $\boldsymbol{L}:{\large X}_{\infty}%
\rightarrow{\large X}_{\infty}$ generates a Feller semigroup, since
${\large X}_{M}\hookrightarrow{\large X}_{\infty}$, $\boldsymbol{L}%
_{M}=\boldsymbol{L}\mid_{X_{M}}$ and $\boldsymbol{L}_{M}:{\large X}%
_{M}\rightarrow{\large X}_{M}$, we have $\boldsymbol{L}_{M}$ generates a
Feller semigroup on $X_{M}$ and Theorem \ \ref{Theorem1} holds in $X_{M}$:

\begin{theorem}
\label{Theorem1_A}There exists a probability measure $p_{t}^{(M)}\left(
x,\cdot\right)  $, $t\geq0$, $x\in\mathcal{K}_{N}$, on the Borel $\sigma
$-algebra of $\mathcal{K}_{N}$, such that the Cauchy problem:%
\begin{equation}
\left\{
\begin{array}
[c]{l}%
h^{\left(  M\right)  }\left(  x,t\right)  \in{\large C}^{1}(\left(
0,\infty\right)  ,{\large X}_{M})\\
\\
\frac{\partial h^{\left(  M\right)  }\left(  x,t\right)  }{\partial
t}=\varepsilon\boldsymbol{L}_{M}h^{\left(  M\right)  }\left(  x,t\right)
\text{, \ }x\in\mathcal{K}_{N}\text{, }t>0\\
\\
h^{\left(  M\right)  }\left(  x,0\right)  =h_{0}^{\left(  M\right)  }%
(x)\in{\large X}_{M}%
\end{array}
\right.  \label{Eq_p_adic_dif_XM}%
\end{equation}
has a unique solution of the form%
\[
h^{\left(  M\right)  }(x,t)=\int\limits_{\mathcal{K}_{N}}h_{0}^{\left(
M\right)  }(y)p_{t}^{\left(  M\right)  }\left(  x,dy\right)  .
\]
In addition, $p_{t}^{\left(  M\right)  }\left(  x,\cdot\right)  $ is the
transition function of a Markov process $\mathfrak{X}^{\left(  M\right)  }$
whose paths are right continuous and have no discontinuities other than jumps.
\end{theorem}

If $M=N$, then $G_{N}^{M}=G_{N}^{0}$ and Theorem \ref{Theorem1_A} describes
the diffusion mechanism in $X_{N}$, which is exactly the diffusion mechanism
in the original discrete system (\ref{EQ_1}). In the construction of the
Markov process attached to (\ref{Eq_p_adic_dif_XM}) the nature of the points
of $G_{N}^{M}$ is not relevant (consequently in the case $M=N$, the embedding
of $\mathcal{G}$ into $\mathbb{Q}_{p}$ is not relevant). What matters is the
existence of a Feller semigroup acting on a suitable space of continuous
functions on $X_{M}$.

\subsection{\label{Section_selsimilarity_XM} Self-similarity}

We take $J(x,y)=J_{N}(x,y)$ and consider operators $L_{\lambda}$, $\lambda
\geq1$, introduced in Section \ref{Sect_selfsimilarity_1}, see
(\ref{Eq_ope_L_lambda}). We define, for $M\geq N$,
\begin{equation}
L_{M,\lambda}=L_{\lambda}\mid_{X_{M}}. \label{Eq_ope_L_lambda_XM}%
\end{equation}
This definition makes sense since $X_{M}\hookrightarrow X_{\infty}.$
Furthermore, $L_{M,\lambda}:X_{M}\rightarrow X_{M}$ is a bounded linear
operator. Theorem \ref{Theorem1_A} holds true for the following Cauchy
problem:
\begin{equation}
\left\{
\begin{array}
[c]{l}%
h^{\left(  M\right)  }\left(  x,t\right)  \in{\large C}^{1}(\left(
0,\infty\right)  ,{\large X}_{M})\\
\\
\frac{\partial h^{\left(  M\right)  }\left(  x,t\right)  }{\partial
t}=\varepsilon\boldsymbol{L}_{M,\lambda}h^{\left(  M\right)  }\left(
x,t\right)  \text{, \ }x\in\mathcal{K}_{N}\text{, }t>0\\
\\
h^{\left(  M\right)  }\left(  x,0\right)  =h_{0}^{\left(  M\right)  }%
(x)\in{\large X}_{M}.
\end{array}
\right.  \label{Eq_Family_p_adic_eq_2}%
\end{equation}
Thus, (\ref{Eq_Family_p_adic_eq_2}) is a family of $p$-adic diffusion
equations on $X_{M}$ parametrized by the set%
\[
\mathcal{P}_{M}:=\left\{  \left(  \varepsilon,\lambda\right)  \in
\mathbb{R}_{+}^{2};\text{ }\varepsilon>0\text{, }\lambda\geq1\right\}  .
\]
We identify the family (\ref{Eq_Family_p_adic_eq_2}) with the set
$\mathcal{P}_{M}$. By defining $\mathcal{S}_{\mathcal{P}_{M}}$ in an analogous
way to $\mathcal{S}_{\mathcal{P}}$, we obtain the following result:

\begin{theorem}
\label{Theorem1_fractal_A}The family $\mathcal{P}_{M}$ is invariant under the
action of the monoid $\mathcal{S}_{\mathcal{P}_{M}}$.
\end{theorem}

\begin{remark}
\label{Nota_replica}Assume that $\left(  \varepsilon,1\right)  \in
\mathcal{P}_{N}$, i.e. the equation%
\[
\frac{\partial h_{I}^{\left(  N\right)  }}{\partial t}=\varepsilon%
%TCIMACRO{\tsum \limits_{J\in G_{N}^{0}}}%
%BeginExpansion
{\textstyle\sum\limits_{J\in G_{N}^{0}}}
%EndExpansion
\left\{  A_{IJ}-\gamma_{I}\delta_{IJ}\right\}  h_{J}^{\left(  N\right)
}\text{, }I\in G_{N}^{0}\text{,}%
\]
is a $p$-adic diffusion equation, then for any $0<\sigma\leq1$, $\left(
\varepsilon\sigma,\sigma^{-1}\right)  \in\mathcal{P}_{N}$, i.e. the equation%
\[
\frac{\partial h_{I}^{\left(  N\right)  }}{\partial t}=\varepsilon%
%TCIMACRO{\tsum \limits_{J\in G_{N}^{0}}}%
%BeginExpansion
{\textstyle\sum\limits_{J\in G_{N}^{0}}}
%EndExpansion
\left\{  \sigma A_{IJ}-\gamma_{I}\delta_{IJ}\right\}  h_{J}^{\left(  N\right)
}\text{, }I\in G_{N}^{0}\text{,}%
\]
is a $p$-adic diffusion equation.
\end{remark}

\section{The Cauchy problem in $X_{\bullet}$}

In this section we study in unified form the Cauchy problems (\ref{Eq_14}%
)-(\ref{Eq_15}). We use the following notation:%
\[
X_{\bullet}:=\left\{
\begin{array}
[c]{lll}%
X_{\infty} & \text{if} & \bullet=\infty\\
&  & \\
X_{M} & \text{if} & \bullet=M
\end{array}
\right.  \text{, \ \ \ \ \ }\boldsymbol{L}_{\bullet}:=\left\{
\begin{array}
[c]{lll}%
\boldsymbol{L} & \text{if} & \bullet=\infty\\
&  & \\
\boldsymbol{L}_{M} & \text{if} & \bullet=M
\end{array}
\right.  \text{,}%
\]
and $u^{\left(  \bullet\right)  }\left(  t\right)  \oplus v^{\left(
\bullet\right)  }\left(  t\right)  $ means $u\left(  t\right)  \oplus v\left(
t\right)  $ if $\bullet=\infty$.\ By using this notation and (\ref{Eq_nota}),
the Cauchy problems (\ref{Eq_14})-(\ref{Eq_15}), with initial data in $U\cap
X_{N}\oplus U\cap X_{N}$, can be written as%
\begin{equation}
\left\{
\begin{array}
[c]{l}%
\frac{\partial}{\partial t}\left[
\begin{array}
[c]{l}%
u^{\left(  \bullet\right)  }\left(  t\right) \\
\\
v^{\left(  \bullet\right)  }\left(  t\right)
\end{array}
\right]  =\left[
\begin{array}
[c]{l}%
f(u^{\left(  \bullet\right)  }\left(  t\right)  ,v^{\left(  \bullet\right)
}\left(  t\right)  )\\
\\
g(u^{\left(  \bullet\right)  }\left(  t\right)  ,v^{\left(  \bullet\right)
}\left(  t\right)  )
\end{array}
\right]  +\left[
\begin{array}
[c]{l}%
\varepsilon\boldsymbol{L}_{\bullet}u^{\left(  \bullet\right)  }\left(
t\right) \\
\\
\varepsilon d\boldsymbol{L}_{\bullet}v^{\left(  \bullet\right)  }\left(
t\right)
\end{array}
\right]  ,\\
\\
t\in\left[  0,\tau\right)  \text{, }x\in\mathcal{K}_{N};\\
\\
u^{\left(  \bullet\right)  }\left(  0\right)  \oplus v^{\left(  \bullet
\right)  }\left(  0\right)  \in U\cap X_{N}\oplus U\cap X_{N}.
\end{array}
\right.  \label{Eq_18}%
\end{equation}

\subsection{Mild solutions}

We review the notion of mild solutions for the Cauchy problem (\ref{Eq_18}).
We will follow \cite[Chapter 5]{Milan}. We use the following conditions:

\paragraph{\noindent\textbf{Condition AS1}}

$X_{\bullet}\oplus X_{\bullet}$ is a real Banach space.

\paragraph{\noindent\textbf{Condition AS2}}

The operator $\left[
\begin{array}
[c]{l}%
\varepsilon\boldsymbol{L}_{\bullet}\\
\varepsilon d\boldsymbol{L}_{\bullet}%
\end{array}
\right]  $ is the generator of a strongly continuous semigroup $\left\{
e^{\varepsilon t\boldsymbol{L}_{\bullet}}\right\}  _{t\geq0}\oplus\left\{
e^{\varepsilon dt\boldsymbol{L}_{\bullet}}\right\}  _{t\geq0}$ satisfying
\[
\left\Vert e^{\varepsilon t\boldsymbol{L}_{\bullet}}\oplus e^{\varepsilon
dt\boldsymbol{L}_{\bullet}}\right\Vert \leq1\text{ for }t\geq0.
\]
This condition follows from Lemma \ref{Lemma1} and Remark \ref{Nota_Lemma1}.

\paragraph{\noindent\textbf{Condition AS3}}

Let $U\subset X_{\infty}$ be the open set defined in (\ref{Definition_set_U}),
and let
\[
\left[
\begin{array}
[c]{l}%
f\\
g
\end{array}
\right]  :\left(  U\oplus U\right)  \rightarrow X_{\infty}\oplus X_{\infty}%
\]
be the continuous mapping defined in (\ref{Definition_map_f_g}). Then for each
$u_{0}\oplus v_{0}\in U\oplus U$, there exist $\delta>0$ and $L<\infty$ such
that%
\begin{equation}
\left\Vert \left[
\begin{array}
[c]{l}%
f(u_{1},v_{1})\\
\\
g(u_{1},v_{1})
\end{array}
\right]  -\left[
\begin{array}
[c]{l}%
f(u_{2},v_{2})\\
\\
g(u_{2},v_{2})
\end{array}
\right]  \right\Vert \leq L\left\Vert \left(  u_{1}-u_{2}\right)
\oplus\left(  v_{1}-v_{2}\right)  \right\Vert , \label{Lipschitz_condition}%
\end{equation}
for $u_{1}\oplus v_{1}$, $u_{2}\oplus v_{2}$ in the ball $B\left(  u_{0}\oplus
v_{0},\delta\right)  $. This result follows from the Taylor formula applied to
$f$ and $g$ and Hypothesis 1. Indeed, take $\delta>0$ such that $u_{1}\oplus
v_{1}$, $u_{2}\oplus v_{2}\in B(u_{0}\oplus v_{0},\delta)\subset U\oplus U$.
This implies
\[
\left\Vert u_{i}-u_{0}\right\Vert _{\infty}<\delta\text{ and }\left\Vert
v_{i}-v_{0}\right\Vert _{\infty}<\delta\text{, }i=1,2\text{,}%
\]
and by using the definition of $U$,%
\[
a<u_{i}<b\text{ and }a<v_{i}<b\text{, }i=1,2\text{.}%
\]
Now by using the Taylor formula, for $h\in C^{1}\left(  \left(  a,b\right)
\times\left(  a,b\right)  \right)  $, with $\left(  u_{1},v_{1}\right)  $,
$\left(  u_{1}+h_{1},v_{1}+h_{2}\right)  \in\left(  a,b\right)  \times\left(
a,b\right)  $,
\begin{multline*}
h(u_{1}+h_{1},v_{1}+h_{2})=h\left(  u_{1},v_{1}\right)  +\frac{\partial
h\left(  u_{1},v_{1}\right)  }{\partial u}h_{1}+\frac{\partial h\left(
u_{1},v_{1}\right)  }{\partial v}h_{2}\\
+\sqrt{h_{1}^{2}+h_{2}^{2}}\text{ }E(\left(  u_{1}+h_{1},v_{1}+h_{2}\right)  ,
\end{multline*}
where $E$ is a continuous function near to $\left(  u_{1},v_{1}\right)  $, we
obtain%
\begin{multline*}
\left\Vert f\left(  u_{1},v_{1}\right)  -f\left(  u_{2},v_{2}\right)
\right\Vert _{\infty}<\\
2\delta\left\{  \sup_{\substack{a<u_{1}<b\\<a<v_{1}<b}}\left\vert
\frac{\partial f\left(  u_{1},v_{1}\right)  }{\partial u}\right\vert
+\sup_{\substack{a<u_{1}<b\\<a<v_{1}<b}}\left\vert \frac{\partial f\left(
u_{1},v_{1}\right)  }{\partial v}\right\vert \right\}  +\sqrt{2}\delta
\sup_{\substack{a<u_{2}<b\\<a<v_{2}<b}}\left\vert E\left(  u_{2},v_{2}\right)
\right\vert .
\end{multline*}
A similar estimation holds for $g\left(  u_{1},v_{1}\right)  -g\left(
u_{2},v_{2}\right)  $.

Notice that our nonlinearity does not depend on the temporal variable. For
this reason, we are using \ a simpler condition than the one used in
\cite[Section 5.2]{Milan}.

\begin{remark}
Take
\begin{equation}
\left[
\begin{array}
[c]{l}%
f\\
g
\end{array}
\right]  :\left(  U\cap X_{M}\oplus U\cap X_{M}\right)  \rightarrow
X_{M}\oplus X_{M}\text{,} \label{map_f_g_X_M}%
\end{equation}
since $X_{M}\hookrightarrow X_{\infty}$, condition (\ref{Lipschitz_condition})
holds true for map (\ref{map_f_g_X_M}).
\end{remark}

\begin{definition}
For $\tau_{0}\in\left(  0,\tau\right]  $, let $\mathcal{S}_{\text{Mild}%
}\left(  \tau_{0},X_{\bullet}\oplus X_{\bullet}\right)  $ be the collection of
all $u^{\left(  \bullet\right)  }\oplus v^{\left(  \bullet\right)  }\in
C\left(  \left[  0,\tau_{0}\right)  ,U\cap X_{\bullet}\oplus U\cap X_{\bullet
}\right)  $ which satisfy%
\[
\int\nolimits_{0}^{t}u^{\left(  \bullet\right)  }\left(  s\right)  ds\in
Dom\left(  \varepsilon\boldsymbol{L}_{\bullet}\right)  =X_{\bullet}\text{
\ and \ }\int\nolimits_{0}^{t}v^{\left(  \bullet\right)  }\left(  s\right)
ds\in Dom\left(  \varepsilon d\boldsymbol{L}_{\bullet}\right)  =X_{\bullet
}\text{,}%
\]
and%
\[
\left\{
\begin{array}
[c]{ccc}%
u^{\left(  \bullet\right)  }\left(  t\right)  -u^{\left(  \bullet\right)
}\left(  0\right)  +\varepsilon\boldsymbol{L}_{\bullet}\int\nolimits_{0}%
^{t}u^{\left(  \bullet\right)  }\left(  s\right)  ds & = & \int\nolimits_{0}%
^{t}f\left(  u^{\left(  \bullet\right)  }\left(  s\right)  ,v^{\left(
\bullet\right)  }\left(  s\right)  \right)  ds\\
&  & \\
v^{\left(  \bullet\right)  }\left(  t\right)  -v^{\left(  \bullet\right)
}\left(  0\right)  +\varepsilon d\boldsymbol{L}_{\bullet}\int\nolimits_{0}%
^{t}v^{\left(  \bullet\right)  }\left(  s\right)  ds & = & \int\nolimits_{0}%
^{t}g\left(  u^{\left(  \bullet\right)  }\left(  s\right)  ,v^{\left(
\bullet\right)  }\left(  s\right)  \right)  ds,
\end{array}
\right.
\]
for $t\in\left[  0,\tau_{0}\right)  $. The elements of $\mathcal{S}%
_{\text{Mild}}\left(  \tau_{0},X_{\bullet}\oplus X_{\bullet}\right)  $ are the
called mild solutions of (\ref{Eq_18}).
\end{definition}

By using well-known results from semigroup theory, see e.g. \cite[Theorem
4.7.3]{Milan}, we have $u^{\left(  \bullet\right)  }\oplus v^{\left(
\bullet\right)  }\in\mathcal{S}_{\text{Mild}}\left(  \tau_{0},X_{\bullet
}\oplus X_{\bullet}\right)  $ if and only if
\begin{equation}
u^{\left(  \bullet\right)  }\oplus v^{\left(  \bullet\right)  }\in C\left(
\left[  0,\tau_{0}\right)  ,U\cap X_{\bullet}\oplus U\cap X_{\bullet}\right)
\label{Mild_sol_1}%
\end{equation}
and
\begin{equation}
\left\{
\begin{array}
[c]{ccc}%
u^{\left(  \bullet\right)  }\left(  t\right)  & = & e^{\varepsilon
t\boldsymbol{L}_{\bullet}}u^{\left(  \bullet\right)  }\left(  0\right)
+\int\nolimits_{0}^{t}e^{\varepsilon\left(  t-s\right)  \boldsymbol{L}%
_{\bullet}}f\left(  u^{\left(  \bullet\right)  }\left(  s\right)  ,v^{\left(
\bullet\right)  }\left(  s\right)  \right)  ds\\
&  & \\
v^{\left(  \bullet\right)  }\left(  t\right)  & = & e^{\varepsilon
dt\boldsymbol{L}_{\bullet}}v^{\left(  \bullet\right)  }\left(  0\right)
+\int\nolimits_{0}^{t}e^{\varepsilon d\left(  t-s\right)  \boldsymbol{L}%
_{\bullet}}g\left(  u^{\left(  \bullet\right)  }\left(  s\right)  ,v^{\left(
\bullet\right)  }\left(  s\right)  \right)  ds,
\end{array}
\right.  \label{Mild_sol_2}%
\end{equation}
for $t\in\left[  0,\tau_{0}\right)  $.

The following result shows that Hypothesis 1, which also implies Condition
AS3, implies that any mild solution is a classical solution.

\begin{lemma}
\label{Lemma_mild_sol}$\mathcal{S}_{\text{Mild}}\left(  \tau_{0},X_{\bullet
}\oplus X_{\bullet}\right)  \subset C^{1}\left(  \left[  0,\tau_{0}\right)
,U\cap X_{\bullet}\oplus U\cap X_{\bullet}\right)  $.
\end{lemma}

\begin{proof}
Set
\[%
\begin{array}
[c]{ccc}%
r^{\left(  \bullet\right)  }\left(  t\right)  & := & \int\nolimits_{0}%
^{t}e^{\varepsilon\left(  t-s\right)  \boldsymbol{L}_{\bullet}}f\left(
u^{\left(  \bullet\right)  }\left(  s\right)  ,v^{\left(  \bullet\right)
}\left(  s\right)  \right)  ds\\
&  & \\
s^{\left(  \bullet\right)  }\left(  t\right)  & := & \int\nolimits_{0}%
^{t}e^{\varepsilon d\left(  t-s\right)  \boldsymbol{L}_{\bullet}}g\left(
u^{\left(  \bullet\right)  }\left(  s\right)  ,v^{\left(  \bullet\right)
}\left(  s\right)  \right)  ds,
\end{array}
\]
for $t\in\left[  0,\tau_{0}\right)  $. Then by (\ref{Mild_sol_1}%
)-(\ref{Mild_sol_2}), it is sufficient to show that
\[
r^{\left(  \bullet\right)  }\left(  t\right)  \text{, }s^{\left(
\bullet\right)  }\left(  t\right)  \in C^{1}\left(  \left[  0,\tau_{0}\right)
,U\cap X_{\bullet}\oplus U\cap X_{\bullet}\right)  \text{.}%
\]
We show that
\[
\frac{d}{dt}r^{\left(  \bullet\right)  }\left(  t\right)  =\varepsilon
\int\nolimits_{0}^{t}e^{\varepsilon\left(  t-s\right)  \boldsymbol{L}%
_{\bullet}}\boldsymbol{L}_{\bullet}f\left(  u^{\left(  \bullet\right)
}\left(  s\right)  ,v^{\left(  \bullet\right)  }\left(  s\right)  \right)
ds+f\left(  u^{\left(  \bullet\right)  }\left(  t\right)  ,v^{\left(
\bullet\right)  }\left(  t\right)  \right)  ,
\]
for $t\in\left[  0,\tau_{0}\right)  $. A similar formula holds for the
derivative of $s^{\left(  \bullet\right)  }\left(  t\right)  $. Take
$t\in\left[  0,\tau_{0}\right)  $ and $h\in\left[  0,t-s\right]  $, then
\begin{align*}
\frac{r^{\left(  \bullet\right)  }\left(  t+h\right)  -r^{\left(
\bullet\right)  }\left(  t\right)  }{h}  &  =\int\nolimits_{0}^{t}%
e^{\varepsilon\left(  t-s\right)  \boldsymbol{L}_{\bullet}}\left\{
\frac{e^{\varepsilon h\boldsymbol{L}_{\bullet}}-\boldsymbol{I}}{h}\right\}
f\left(  u^{\left(  \bullet\right)  }\left(  s\right)  ,v^{\left(
\bullet\right)  }\left(  s\right)  \right)  ds\\
&  +\frac{1}{h}\int\nolimits_{t}^{t+h}e^{\varepsilon\left(  t+h-s\right)
\boldsymbol{L}_{\bullet}}f\left(  u^{\left(  \bullet\right)  }\left(
s\right)  ,v^{\left(  \bullet\right)  }\left(  s\right)  \right)  ds.
\end{align*}
Then
\begin{gather*}
\left\Vert \frac{r^{\left(  \bullet\right)  }\left(  t+h\right)  -r^{\left(
\bullet\right)  }\left(  t\right)  }{h}-\frac{d}{dt}r^{\left(  \bullet\right)
}\left(  t\right)  \right\Vert _{\infty}\leq\\
\left\Vert \int\nolimits_{0}^{t}e^{\varepsilon\left(  t-s\right)
\boldsymbol{L}_{\bullet}}\left\{  \frac{e^{\varepsilon h\boldsymbol{L}%
_{\bullet}}-\boldsymbol{I}}{h}-\varepsilon\boldsymbol{L}_{\bullet}\right\}
f\left(  u^{\left(  \bullet\right)  }\left(  s\right)  ,v^{\left(
\bullet\right)  }\left(  s\right)  \right)  ds\right\Vert _{\infty}+\\
\left\Vert \frac{1}{h}\int\nolimits_{t}^{t+h}e^{\varepsilon\left(
t-s+h\right)  \boldsymbol{L}_{\bullet}}\left\{  f\left(  u^{\left(
\bullet\right)  }\left(  s\right)  ,v^{\left(  \bullet\right)  }\left(
s\right)  \right)  -f\left(  u^{\left(  \bullet\right)  }\left(  t\right)
,v^{\left(  \bullet\right)  }\left(  t\right)  \right)  \right\}
ds\right\Vert _{\infty}\\
=:I_{1}+I_{2}.
\end{gather*}
Now by using that $u^{\left(  \bullet\right)  }\oplus$ $v^{\left(
\bullet\right)  }\in C\left(  \left[  0,\tau_{0}\right)  ,U\cap X_{\bullet
}\oplus U\cap X_{\bullet}\right)  $ and Hypothesis 1:%
\[
I_{1}\leq\left\Vert \frac{e^{\varepsilon h\boldsymbol{L}_{\bullet}%
}-\boldsymbol{I}}{h}-\varepsilon\boldsymbol{L}_{\bullet}\right\Vert \left\Vert
f\left(  u^{\left(  \bullet\right)  }\left(  s\right)  ,v^{\left(
\bullet\right)  }\left(  s\right)  \right)  \right\Vert _{\infty}\leq
B\left\Vert \frac{e^{\varepsilon h\boldsymbol{L}_{\bullet}}-\boldsymbol{I}}%
{h}-\varepsilon\boldsymbol{L}_{\bullet}\right\Vert \rightarrow0\text{ }%
\]
as $h\rightarrow0$, where $B=\sup_{u,v\in U\cap X_{\bullet}}\left\vert
f\left(  u,v\right)  \right\vert <\infty$. And by using that
\[
u^{\left(  \bullet\right)  }\oplus v^{\left(  \bullet\right)  }\in C\left(
\left[  0,\tau_{0}\right)  ,U\cap X_{\bullet}\oplus U\cap X_{\bullet}\right)
\]
and Conditions AS2-AS3, we have%
\begin{gather*}
I_{2}\leq\frac{1}{h}\int\nolimits_{t}^{t+h}\left\Vert f\left(  u^{\left(
\bullet\right)  }\left(  s\right)  ,v^{\left(  \bullet\right)  }\left(
s\right)  \right)  -f\left(  u^{\left(  \bullet\right)  }\left(  t\right)
,v^{\left(  \bullet\right)  }\left(  t\right)  \right)  \right\Vert _{\infty
}\text{ }ds\\
\leq L\frac{1}{h}\int\nolimits_{t}^{t+h}\max\left\{  \left\Vert u^{\left(
\bullet\right)  }\left(  s\right)  -u^{\left(  \bullet\right)  }\left(
t\right)  \right\Vert _{\infty},\left\Vert v^{\left(  \bullet\right)  }\left(
s\right)  -v^{\left(  \bullet\right)  }\left(  t\right)  \right\Vert _{\infty
}\right\}  ds\\
\leq L\sup_{s\in\left[  t,t+h\right]  }\left\Vert u^{\left(  \bullet\right)
}\left(  s\right)  -u^{\left(  \bullet\right)  }\left(  t\right)  \right\Vert
_{\infty}+L\sup_{s\in\left[  t,t+h\right]  }\left\Vert v^{\left(
\bullet\right)  }\left(  s\right)  -v^{\left(  \bullet\right)  }\left(
t\right)  \right\Vert _{\infty}\rightarrow0\text{ as }h\rightarrow0.
\end{gather*}

\end{proof}

\subsection{Existence for the initial value problems}

\begin{theorem}
\label{Theorem2} For each $u_{0}^{\left(  \bullet\right)  }\oplus
v_{0}^{\left(  \bullet\right)  }\in U\cap X_{N}\oplus U\cap X_{N}$, there
exist $\tau_{u_{0}^{\left(  \bullet\right)  }\oplus v_{0}^{\left(
\bullet\right)  }}\in\left(  0,\tau\right)  $ and $u^{\left(  \bullet\right)
}\oplus v^{\left(  \bullet\right)  }\in\mathcal{S}_{\text{Mild}}\left(
\tau_{u_{0}^{\left(  \bullet\right)  }\oplus v_{0}^{\left(  \bullet\right)  }%
},X_{\bullet}\oplus X_{\bullet}\right)  $ such that $u^{\left(  \bullet
\right)  }\left(  0\right)  \oplus v^{\left(  \bullet\right)  }\left(
0\right)  =u_{0}^{\left(  \bullet\right)  }\oplus v_{0}^{\left(
\bullet\right)  }$. Furthermore,
\[
\lim_{k\rightarrow\infty}\text{ }\sup_{0\leq t\leq\tau_{u_{0}^{\left(
\bullet\right)  }\oplus v_{0}^{\left(  \bullet\right)  }}}\left\Vert
u_{k}^{\left(  \bullet\right)  }\left(  t\right)  \oplus v_{k}^{\left(
\bullet\right)  }\left(  t\right)  -u^{\left(  \bullet\right)  }\left(
t\right)  \oplus v^{\left(  \bullet\right)  }\left(  t\right)  \right\Vert
=0\text{,}%
\]
where $u_{k}^{\left(  \bullet\right)  }\oplus v_{k}^{\left(  \bullet\right)
}\in C(\left[  0,\tau_{u_{0}^{\left(  \bullet\right)  }\oplus v_{0}^{\left(
\bullet\right)  }}\right]  ,U\cap X_{\bullet}\oplus U\cap X_{\bullet})$ are
defined by $u_{1}^{\left(  \bullet\right)  }\left(  t\right)  \oplus
v_{1}^{\left(  \bullet\right)  }\left(  t\right)  =u_{0}^{\left(
\bullet\right)  }\oplus v_{0}^{\left(  \bullet\right)  }$ and%
\[
\left\{
\begin{array}
[c]{ccc}%
u_{k+1}^{\left(  \bullet\right)  }\left(  t\right)  & = & e^{\varepsilon
t\boldsymbol{L}_{\bullet}}u_{0}^{\left(  \bullet\right)  }+\int\nolimits_{0}%
^{t}e^{\varepsilon\left(  t-s\right)  \boldsymbol{L}_{\bullet}}f\left(
u_{k}^{\left(  \bullet\right)  }\left(  s\right)  ,v_{k}^{\left(
\bullet\right)  }\left(  s\right)  \right)  ds\\
&  & \\
v_{k+1}^{\left(  \bullet\right)  }\left(  t\right)  & = & e^{\varepsilon
dt\boldsymbol{L}_{\bullet}}v_{0}^{\left(  \bullet\right)  }+\int
\nolimits_{0}^{t}e^{\varepsilon d\left(  t-s\right)  \boldsymbol{L}_{\bullet}%
}g\left(  u_{k}^{\left(  \bullet\right)  }\left(  s\right)  ,v_{k}^{\left(
\bullet\right)  }\left(  s\right)  \right)  ds,
\end{array}
\right.
\]
for $t\in\left[  0,\tau_{u_{0}^{\left(  \bullet\right)  }\oplus v_{0}^{\left(
\bullet\right)  }}\right]  $ and $k\in\mathbb{N}\setminus\left\{  0\right\}  $.
\end{theorem}

\begin{proof}
The results follow from Theorems 5.2.2 and 5.1.2 in \cite{Milan}, by using
conditions AS1, AS2, AS3.
\end{proof}

\subsubsection{Some additional remarks}

The uniqueness of all the initial value problems (\ref{Eq_18}) follows \ from
standard results, see e.g. \cite[Theorem 5.2.3]{Milan}. Also, the continuous
dependence of the initial value for all the initial value problems
(\ref{Eq_18}) also holds, see e.g. \cite[Theorem 5.2.4]{Milan}.

Maximal interval existence: for $u_{0}^{\bullet}\oplus v_{0}^{\bullet}\in
U\cap X_{\bullet}\oplus U\cap X_{\bullet}$, there exists $\tau_{\max}$
depending on the initial datum, and $u^{\bullet}\oplus v^{\bullet}%
\in\mathcal{S}_{\text{Mild}}\left(  \tau_{\max},X_{\bullet}\oplus X_{\bullet
}\right)  $ such that $u^{\bullet}\left(  0\right)  \oplus v^{\bullet}\left(
0\right)  =u_{0}^{\bullet}\oplus v_{0}^{\bullet}$, and if $\tau_{\max}<\infty$
then%
\[
\int\nolimits_{0}^{\tau_{\max}}\left\Vert f\left(  u^{\bullet}\left(
t\right)  ,v^{\bullet}\left(  t\right)  \right)  \right\Vert _{\infty
}dt=\infty\text{ \ or \ }\int\nolimits_{0}^{\tau_{\max}}\left\Vert g\left(
u^{\bullet}\left(  t\right)  ,v^{\bullet}\left(  t\right)  \right)
\right\Vert _{\infty}dt=\infty,
\]
or there exists $\widetilde{u}^{\bullet}\oplus\widetilde{v}^{\bullet}%
\in\overline{U}\smallsetminus U$ such that $\lim_{t\rightarrow\tau_{\max}%
}u^{\bullet}\left(  t\right)  \oplus v^{\bullet}\left(  t\right)
=\widetilde{u}^{\bullet}\oplus\widetilde{v}^{\bullet}$. By taking the same
initial datum for all the initial value problems (\ref{Eq_18}), we have the
same maximal interval existence for all of them.

\subsection{\label{Section_examples}Some examples of reaction-diffusion
networks on $X_{\bullet}$}

When studying Turing instability, one is interested in the study of a system
of type (\ref{Eq_18}) around a state $u_{0}\oplus v_{0}$ satisfying
\begin{equation}
f(u_{0},v_{0})=g(u_{0},v_{0})=0\text{ \ with }u_{0}\text{, }v_{0}\text{
non-negative real numbers.} \label{Eq_system}%
\end{equation}
In order to achieve this goal, we pick an open neighborhood $\left(
a,b\right)  \times\left(  a,b\right)  \subset\mathbb{R}^{2}$ containing
$(u_{0},v_{0})$, but no other solution of (\ref{Eq_system}). Then, by Theorem
\ref{Theorem2}, if the initial datum is closed to $(u_{0},v_{0})$, then all
the initial value problems (\ref{Eq_18}) have mild solutions, and the same
maximal interval of existence.

\subsubsection{\label{Sec_Brusselator}The Brusselator}

Taking $A>0$ and $B>0$, the Brusselator on ${\large X}_{\bullet}$ is the
following reaction-diffusion system:%

\begin{equation}
\left\{
\begin{array}
[c]{l}%
u\left(  t\right)  ,v\left(  t\right)  \in{\large C}^{1}(\left[
0,\tau\right)  ,{\large X}_{\bullet});\\
\\
\frac{\partial u^{\left(  \bullet\right)  }\left(  x,t\right)  }{\partial
t}-\varepsilon\boldsymbol{L}_{\bullet}u\left(  x,t\right)  =A-\left(
B+1\right)  u+u^{2}v\\
\\
\frac{\partial v^{\left(  \bullet\right)  }\left(  x,t\right)  }{\partial
t}-\varepsilon d\boldsymbol{L}_{\bullet}v\left(  x,t\right)  =Bu-u^{2}v\text{,
}%
\end{array}
\right.  \label{Brusselator}%
\end{equation}
for $t\in\left[  0,\tau\right)  $, $x\in\mathcal{K}_{N}$. This system has only
a homogeneous steady state: $u=A$, $v=\frac{B}{A}$. We consider
$f(u,v)=A-\left(  B+1\right)  u+u^{2}v$, $g(u,v)=Bu-u^{2}v$ as functions
defined on
\[
\left(  -\delta+A,\delta+A\right)  \times\left(  -\delta+\frac{B}{A}%
,\delta+\frac{B}{A}\right)  \subset\left(  a,b\right)  \times\left(
a,b\right)  ,
\]
for $\delta>0$ sufficiently small. Since $\nabla f(u,v)\neq\left(  0,0\right)
$ and $\nabla g(u,v)\neq\left(  0,0\right)  $ for any $(u,v)\in\mathbb{R}%
\times\mathbb{R}$, there exist $a$, $b\in\mathbb{R}$ such that $\nabla
f\mid_{\left(  a,b\right)  \times\left(  a,b\right)  }\neq$ $\left(
0,0\right)  $ and $\nabla g\mid_{\left(  a,b\right)  \times\left(  a,b\right)
}\neq$ $\left(  0,0\right)  $, and consequently Hypothesis 1 is valid. Now, we
take the subset
\[
\mathcal{U}:=\left\{  r\in X_{\infty};\left\Vert r-A\right\Vert _{\infty
}<\delta\right\}  \oplus\left\{  s\in X_{\infty};\left\Vert h-\frac{B}%
{A}\right\Vert _{\infty}<\delta\right\}  \subset U\oplus U.
\]
Then for any initial datum in $\mathcal{U\cap}X_{\bullet}\oplus X_{\bullet}$,
system (\ref{Brusselator}) has a unique solution, cf. Theorem \ref{Theorem2}
and Lemma \ref{Lemma_mild_sol}.

In the case $X_{\infty}$, the system (\ref{Brusselator}) is a $p$-adic version
of the Brusselator. This system was first considered by Prigogine and Lefever
\cite{Prig-Lefe} and Nicolis and Prigogine \cite{Nicolis-Prig}, see also
\cite[Section 7.5.4]{Perthame}.

\subsubsection{The CIMA reaction}

The CIMA\ reaction (chlorite-iodide-malonic acid) provided experimental
evidence of Turing instability. It was modeled by Lengyel and Epstein
\cite{Lengyel et al}. The following is a version of this system in
$X_{\bullet}$:%
\begin{equation}
\left\{
\begin{array}
[c]{l}%
u\left(  t\right)  ,v\left(  t\right)  \in{\large C}^{1}(\left[
0,\tau\right)  ,{\large X}_{\bullet});\\
\\
\frac{\partial u^{\left(  \bullet\right)  }\left(  x,t\right)  }{\partial
t}-\varepsilon\boldsymbol{L}_{\bullet}u\left(  x,t\right)  =A-u-\frac
{4uv}{1+u^{2}}\\
\\
\frac{\partial v^{\left(  \bullet\right)  }\left(  x,t\right)  }{\partial
t}-\varepsilon d\boldsymbol{L}_{\bullet}v\left(  x,t\right)  =BCu-\frac
{Cuv}{1+u^{2}}\text{, }%
\end{array}
\right.  \label{CIMA}%
\end{equation}
for $t\in\left[  0,\tau\right)  $, $x\in\mathcal{K}_{N}$. Here $u$ (the
activator) denotes the iodide ($I^{-}$ ) concentration and $v$ (the inhibitor)
the chlorite ($ClO_{2}^{-}$ ) concentration. We follow the presentation in
Perthame's book \cite[Section 7.5.2]{Perthame}. Then, we consider this system
with $A>0$, $B>0$, $C>0$. There is a single homogeneous steady state%
\[
u_{0}=\frac{A}{4B+1}\text{, \ }v_{0}=B(1+\frac{A^{2}}{\left(  4B+1\right)
^{2}}).
\]
We consider $f(u,v)=A-u-\frac{4uv}{1+u^{2}}$, $g(u,v)=BCu-\frac{Cuv}{1+u^{2}}$
as functions defined in
\[
\left(  -\delta+u_{0},\delta+u_{0}\right)  \times\left(  -\delta+v_{0}%
,\delta+v_{0}\right)  \subset\left(  a,b\right)  \times\left(  a,b\right)  ,
\]
for $\delta>0$ sufficiently small so that $\left(  0,0\right)  \notin\left(
a,b\right)  \times\left(  a,b\right)  $. Notice that%
\[
\nabla f(u,v)=\left(  0,0\right)  \Leftrightarrow u=0,v=\frac{1}{4}\text{ and
}\nabla g(u,v)=\left(  0,0\right)  \text{ }\Leftrightarrow u=0,v=B.
\]
Then, there exist $a,b\in\mathbb{R}$ such that $\nabla f\mid_{\left(
a,b\right)  \times\left(  a,b\right)  }\neq$ $\left(  0,0\right)  $ and
$\nabla g\mid_{\left(  a,b\right)  \times\left(  a,b\right)  }\neq$ $\left(
0,0\right)  $, and consequently Hypothesis 1 holds. Now, we take the subset
\[
\mathcal{U}:=\left\{  r\in X_{\infty};\left\Vert r-u_{0}\right\Vert _{\infty
}<\delta\right\}  \oplus\left\{  s\in X_{\infty};\left\Vert h-v_{0}\right\Vert
_{\infty}<\delta\right\}  \subset U\oplus U.
\]
Then for any initial datum in $\mathcal{U\cap}X_{\bullet}\oplus X_{\bullet}$,
the system (\ref{CIMA}) has a unique solution, cf. Theorem \ref{Theorem2} and
Lemma \ref{Lemma_mild_sol}. In the case $X_{\infty}$, system (\ref{CIMA}) is a
$p$-adic version of the CIMA reaction-diffusion system.

\section{Approximations}

In this section we show that discretization (\ref{Eq_15}) provides a good
approximation of (\ref{Eq_14}). We use standard techniques for approximating
nonlinear evolution equations, see e.g. \cite{Milan}. The following conditions
will be needed in order to use the results in \cite{Milan}:

\subsection{Condition A}

\noindent(a) $X_{\infty}$, $X_{N}$, $X_{N+1}$, \ldots\ are real Banach spaces,
with the norm $\left\Vert u\oplus v\right\Vert =\max\left\{  \left\Vert
u\right\Vert _{\infty},\left\Vert v\right\Vert _{\infty}\right\}  $.

\noindent(b) $\boldsymbol{P}_{M}\oplus\boldsymbol{P}_{M}\in\mathfrak{B}%
(X_{\infty}\oplus X_{\infty},X_{M}\oplus X_{M})$, for $M\geq N$, where
$\boldsymbol{P}_{M}$ as in (\ref{definition_P_M}). In addition,
\[
\left\Vert \left(  \boldsymbol{P}_{M}\oplus\boldsymbol{P}_{M}\right)  \left(
u\oplus v\right)  \right\Vert \leq\left\Vert u\oplus v\right\Vert \text{ for
}M\geq N.
\]
\noindent(c) $\boldsymbol{E}_{M}\oplus\boldsymbol{E}_{M}\in\mathfrak{B}%
(X_{\infty}\oplus X_{\infty},X_{M}\oplus X_{M})$, for $M\geq N$, where
$\boldsymbol{E}_{M}$ is the canonical continuous embedding $X_{M}%
\hookrightarrow X_{\infty}$.

\noindent(d) $\left(  \boldsymbol{P}_{M}\oplus\boldsymbol{P}_{M}\right)
\left(  \boldsymbol{E}_{M}\oplus\boldsymbol{E}_{M}\right)  \left(  u\oplus
v\right)  =u\oplus v$ for $M\geq N$ and $u\oplus v\in X_{M}\oplus X_{M}$.

\subsection{Condition B}

$\varepsilon\boldsymbol{L}_{M}\oplus\varepsilon d\boldsymbol{L}_{M}%
\in\mathfrak{B}(X_{M}\oplus X_{M})$ for $M\geq N$, and
\[
\left\Vert e^{\varepsilon t\boldsymbol{L}_{M}}\oplus e^{\varepsilon
dt\boldsymbol{L}_{M}}\right\Vert \leq1\text{ for }t\geq0\text{, }M\geq
N\text{.}%
\]
This condition follows from Lemma \ref{Lemma1} and Remark \ref{Nota_Lemma1}.
Notice that Condition B, implies that $\left(  -\infty,0\right)  \subset
\rho\left(  \varepsilon\boldsymbol{L}_{M}\right)  \cap\rho\left(  \varepsilon
d\boldsymbol{L}_{M}\right)  $ for $M\geq N$, where $\rho$ denotes the
resolvent set, see e.g. \cite[Theorem 4.3.2]{Milan}.

\subsection{Condition C'}

$\varepsilon\boldsymbol{L}\oplus\varepsilon d\boldsymbol{L}$ is a densely
defined linear operator in $X_{\infty}\oplus X_{\infty}$, $\lambda_{0}%
\in\left(  -\infty,0\right)  \cap\rho\left(  \varepsilon\boldsymbol{L}\right)
\cap\rho\left(  \varepsilon d\boldsymbol{L}\right)  $ and
\begin{gather}
\lim_{M\rightarrow\infty}\left\Vert \left(  \varepsilon\boldsymbol{L}%
_{M}\oplus\varepsilon d\boldsymbol{L}_{M}\right)  \left(  \boldsymbol{P}%
_{M}\oplus\boldsymbol{P}_{M}\right)  \left(  u\oplus v\right)  -\right.
\nonumber\\
\left.  \left(  \boldsymbol{P}_{M}\oplus\boldsymbol{P}_{M}\right)  \left(
\varepsilon\boldsymbol{L}\oplus\varepsilon d\boldsymbol{L}\right)  \left(
u\oplus v\right)  \right\Vert =0, \label{Eq_20}%
\end{gather}
and
\begin{equation}
\lim_{M\rightarrow\infty}\left\Vert \left(  \boldsymbol{E}_{M}\oplus
\boldsymbol{E}_{M}\right)  \left(  \boldsymbol{P}_{M}\oplus\boldsymbol{P}%
_{M}\right)  \left(  u\oplus v\right)  -\left(  u\oplus v\right)  \right\Vert
=0, \label{Eq_21}%
\end{equation}
for all $u\oplus v\in Dom\left(  \varepsilon\boldsymbol{L}\oplus\varepsilon
d\boldsymbol{L}\right)  .$

Since $\varepsilon\boldsymbol{L}\oplus\varepsilon d\boldsymbol{L}$ is a
bounded linear operator $Dom\left(  \varepsilon\boldsymbol{L}\oplus\varepsilon
d\boldsymbol{L}\right)  =X_{\infty}\oplus X_{\infty}$ satisfying $\left\Vert
e^{\varepsilon t\boldsymbol{L}}\oplus e^{\varepsilon dt\boldsymbol{L}%
}\right\Vert \leq1$ for $t\geq0$, the existence of $\lambda_{0}$ is immediate.
Condition (\ref{Eq_21}) was already established in Lemma \ref{Lemma0}.
Condition (\ref{Eq_20}) is equivalent to%
\begin{equation}
\lim_{M\rightarrow\infty}\left\Vert \varepsilon\boldsymbol{L}_{M}%
\boldsymbol{P}_{M}u-\boldsymbol{P}_{M}\varepsilon\boldsymbol{L}u\right\Vert
_{\infty}=0\text{, and }\lim_{M\rightarrow\infty}\left\Vert d\varepsilon
\boldsymbol{L}_{M}\boldsymbol{P}_{M}u-\boldsymbol{P}_{M}\varepsilon
d\boldsymbol{L}u\right\Vert _{\infty}=0. \label{Eq_22}%
\end{equation}
Now, since $\boldsymbol{L}_{M}=\boldsymbol{L}\mid_{X_{M}}=\boldsymbol{L}%
\boldsymbol{P}_{M}$, and $\boldsymbol{P}_{M}=\boldsymbol{P}_{M}^{2}$, it is
sufficient to show:

\begin{lemma}
$\lim_{M\rightarrow\infty}\left\Vert \boldsymbol{L}_{M}\boldsymbol{P}%
_{M}u-\boldsymbol{P}_{M}\boldsymbol{L}u\right\Vert _{\infty}=0$ for $u\in
X_{\infty}$.
\end{lemma}

\begin{proof}
By using the fact that all the operators involved are linear, it is sufficient
to establish the lemma for the following class of operators:
\[
\boldsymbol{L}_{JK}u(x):=p^{N}A_{JK}\Omega\left(  p^{N}\left\vert
x-J\right\vert _{p}\right)
%TCIMACRO{\tint \limits_{\mathcal{K}_{N}}}%
%BeginExpansion
{\textstyle\int\limits_{\mathcal{K}_{N}}}
%EndExpansion
\left\{  u(y)-u(x)\right\}  \Omega\left(  p^{N}\left\vert y-K\right\vert
_{p}\right)  dy,
\]
for $u\in X_{\infty}$. We define%
\[
\text{Aver}(u;K+p^{N}\mathbb{Z}_{p})=p^{N}%
%TCIMACRO{\tint \limits_{K+p^{N}\mathbb{Z}_{p}}}%
%BeginExpansion
{\textstyle\int\limits_{K+p^{N}\mathbb{Z}_{p}}}
%EndExpansion
u(y)dy.
\]
Then%
\begin{gather}
\boldsymbol{L}_{JK}u(x)=\left(  A_{JK}\text{Aver}(u;K+p^{N}\mathbb{Z}%
_{p})\right)  \Omega\left(  p^{N}\left\vert x-J\right\vert _{p}\right)
\label{Eq_24}\\
-A_{JK}u(x)\Omega\left(  p^{N}\left\vert x-J\right\vert _{p}\right)
.\nonumber
\end{gather}
We show that
\[
\lim_{M\rightarrow\infty}\left\Vert \boldsymbol{L}_{JK}\boldsymbol{P}%
_{M}u-\boldsymbol{P}_{M}\boldsymbol{L}_{JK}u\right\Vert _{\infty}=0\text{ for
}u\in X_{\infty}.
\]
The identity%
\[
J+p^{N}\mathbb{Z}_{p}=%
%TCIMACRO{\tbigsqcup \limits_{J_{j}\in G_{J}^{M}}}%
%BeginExpansion
{\textstyle\bigsqcup\limits_{J_{j}\in G_{J}^{M}}}
%EndExpansion
J_{j}+p^{N}\mathbb{Z}_{p}\text{, for }M\geq N\text{,}%
\]
can be rewritten as%
\begin{equation}
\Omega\left(  p^{N}\left\vert x-J\right\vert _{p}\right)  =%
%TCIMACRO{\tsum \limits_{J_{j}\in G_{J}^{M}}}%
%BeginExpansion
{\textstyle\sum\limits_{J_{j}\in G_{J}^{M}}}
%EndExpansion
\Omega\left(  p^{N}\left\vert x-J_{j}\right\vert _{p}\right)  \text{, for
}M\geq N. \label{Eq_25}%
\end{equation}
Now by using that%
\[
\boldsymbol{P}_{M}\left(  \boldsymbol{L}_{JK}u(x)\right)  =%
%TCIMACRO{\tsum \limits_{R_{j}\in G_{N}^{M}}}%
%BeginExpansion
{\textstyle\sum\limits_{R_{j}\in G_{N}^{M}}}
%EndExpansion
\boldsymbol{L}_{JK}u(R_{j})\Omega\left(  p^{M}\left\vert x-R_{j}\right\vert
_{p}\right)
\]
and (\ref{Eq_24})-(\ref{Eq_25}), we obtain%
\begin{gather}
\boldsymbol{P}_{M}\left(  \boldsymbol{L}_{JK}u(x)\right)  =\left(
A_{JK}\text{Aver}(u;K+p^{N}\mathbb{Z}_{p})\right)
%TCIMACRO{\tsum \limits_{R_{j}\in G_{J}^{M}}}%
%BeginExpansion
{\textstyle\sum\limits_{R_{j}\in G_{J}^{M}}}
%EndExpansion
\Omega\left(  p^{M}\left\vert x-R_{j}\right\vert _{p}\right) \label{Eq_26}\\
-A_{JK}%
%TCIMACRO{\tsum \limits_{R_{j}\in G_{J}^{M}}}%
%BeginExpansion
{\textstyle\sum\limits_{R_{j}\in G_{J}^{M}}}
%EndExpansion
u\left(  R_{j}\right)  \Omega\left(  p^{M}\left\vert x-R_{j}\right\vert
_{p}\right)  ,\nonumber
\end{gather}
and by using again (\ref{Eq_25}) and (\ref{Eq_26}), and the definition of
$P_{M}u(x)$,%
\begin{gather}
\boldsymbol{P}_{M}\left(  \boldsymbol{L}_{JK}u(x)\right)  =\left(
A_{JK}\text{Aver}(u;K+p^{N}\mathbb{Z}_{p})\right)  \Omega\left(
p^{N}\left\vert x-J\right\vert _{p}\right) \label{Eq_27A}\\
-A_{JK}\boldsymbol{P}_{M}u(x)\Omega\left(  p^{N}\left\vert x-J\right\vert
_{p}\right)  .\nonumber
\end{gather}
To compute $\boldsymbol{L}_{JK}\left(  \boldsymbol{P}_{M}u(x)\right)  $, we
use first (\ref{Eq_24}) to get%
\begin{gather}
\boldsymbol{L}_{JK}\left(  \boldsymbol{P}_{M}u(x)\right)  =\left(
A_{JK}\text{Aver}(\boldsymbol{P}_{M}u;K+p^{N}\mathbb{Z}_{p})\right)
\Omega\left(  p^{N}\left\vert x-J\right\vert _{p}\right) \label{Eq_29}\\
-A_{JK}\left(  \boldsymbol{P}_{M}u\left(  x\right)  \right)  \Omega\left(
p^{N}\left\vert x-J\right\vert _{p}\right)  .\nonumber
\end{gather}
Now from (\ref{Eq_27A}) and (\ref{Eq_29}), we have%
\begin{equation}
\left\Vert \boldsymbol{P}_{M}\left(  \boldsymbol{L}_{JK}u(x)\right)
-\boldsymbol{L}_{JK}\left(  P_{M}u(x)\right)  \right\Vert _{\infty}\leq
A_{JK}\left\vert \text{Aver}(u-\boldsymbol{P}_{M}u;K+p^{N}\mathbb{Z}%
_{p})\right\vert . \label{Eq_30}%
\end{equation}
On the other hand,%
\begin{multline*}
\left\vert \text{Aver}(u-\boldsymbol{P}_{M}u;K+p^{N}\mathbb{Z}_{p})\right\vert
\leq p^{N}%
%TCIMACRO{\tint \limits_{K+p^{N}\mathbb{Z}_{p}}}%
%BeginExpansion
{\textstyle\int\limits_{K+p^{N}\mathbb{Z}_{p}}}
%EndExpansion
\left\vert u(y)-\boldsymbol{P}_{M}u\left(  y\right)  \right\vert dy\\
\leq\sup_{y\in K+p^{N}\mathbb{Z}_{p}}\left\vert u(y)-\boldsymbol{P}%
_{M}u\left(  y\right)  \right\vert \leq\left\Vert u-\boldsymbol{P}%
_{M}u\right\Vert _{\infty}\rightarrow0\text{ as }M\rightarrow\infty,
\end{multline*}
cf. Lemma \ref{Lemma0}.
\end{proof}

\subsection{Existence of good approximations}

\begin{theorem}
\label{Theorem3}Take $u_{0}\oplus v_{0}\in U\oplus U$. Let $u\oplus v$ be the
mild solution of (\ref{Eq_14}), and let $u^{\left(  M\right)  }\oplus
v^{\left(  M\right)  }$ be the mild solution of (\ref{Eq_15}) with initial
datum $u^{\left(  M\right)  }\left(  0\right)  \oplus v^{\left(  M\right)
}\left(  0\right)  =\left(  P_{M}\oplus P_{M}\right)  \left(  u_{0}\oplus
v_{0}\right)  $. Then%
\[
\lim_{M\rightarrow\infty}\sup_{0\leq t\leq\tau}\left\Vert u^{\left(  M\right)
}\left(  t\right)  \oplus v^{\left(  M\right)  }\left(  t\right)  -u\left(
t\right)  \oplus v\left(  t\right)  \right\Vert =0,
\]
where $\tau<\tau_{\max}$, and $\tau_{\max}$ is the maximal interval of
existence for the solution $u\left(  t\right)  \oplus v\left(  t\right)  $
with initial datum $u_{0}\oplus v_{0}$.
\end{theorem}

\begin{proof}
The result follows from Conditions A, B, C', by using the argument given in
\cite{Milan} for Theorem 5.4.7. For this reason, we just provide some details
of the proof. \ Notice that for $t\in\left[  0,\tau\right]  $, $M\geq N$, we
have
\[
\left\{
\begin{array}
[c]{l}%
u^{\left(  M\right)  }\left(  t\right)  =e^{t\varepsilon\boldsymbol{L}_{M}%
}P_{M}u_{0}+%
%TCIMACRO{\tint \nolimits_{0}^{t}}%
%BeginExpansion
{\textstyle\int\nolimits_{0}^{t}}
%EndExpansion
e^{\left(  t-s\right)  \varepsilon\boldsymbol{L}_{M}}P_{M}f\left(  u^{\left(
M\right)  }\left(  s\right)  ,v^{\left(  M\right)  }\left(  s\right)  \right)
ds\\
\\
v^{\left(  M\right)  }\left(  t\right)  =e^{t\varepsilon d\boldsymbol{L}_{M}%
}P_{M}v_{0}+%
%TCIMACRO{\tint \nolimits_{0}^{t}}%
%BeginExpansion
{\textstyle\int\nolimits_{0}^{t}}
%EndExpansion
e^{\left(  t-s\right)  \varepsilon d\boldsymbol{L}_{M}}P_{M}g\left(
u^{\left(  M\right)  }\left(  s\right)  ,v^{\left(  M\right)  }\left(
s\right)  \right)  ds,
\end{array}
\right.
\]%
\[
\left\{
\begin{array}
[c]{l}%
u\left(  t\right)  =e^{t\varepsilon\boldsymbol{L}}u_{0}+%
%TCIMACRO{\tint \nolimits_{0}^{t}}%
%BeginExpansion
{\textstyle\int\nolimits_{0}^{t}}
%EndExpansion
e^{\left(  t-s\right)  \varepsilon\boldsymbol{L}}f\left(  u\left(  s\right)
,v\left(  s\right)  \right)  ds\\
\\
v\left(  t\right)  =e^{t\varepsilon d\boldsymbol{L}}v_{0}+%
%TCIMACRO{\tint \nolimits_{0}^{t}}%
%BeginExpansion
{\textstyle\int\nolimits_{0}^{t}}
%EndExpansion
e^{\left(  t-s\right)  \varepsilon d\boldsymbol{L}}g\left(  u\left(  s\right)
,v\left(  s\right)  \right)  ds.
\end{array}
\right.
\]
The estimations for $\sup_{0\leq t\leq\tau}\left\Vert u^{\left(  M\right)
}\left(  t\right)  -u\left(  t\right)  \right\Vert _{\infty}$ and $\sup_{0\leq
t\leq\tau}\left\Vert v^{\left(  M\right)  }\left(  t\right)  -v\left(
t\right)  \right\Vert _{\infty}$ follow from Theorem 5.4.2 \ in \cite{Milan},
which is a consequence of Conditions A, B, C', by using the reasoning given in
the proof Theorem 5.4.7 in \cite{Milan}.
\end{proof}

\section{Replicas of reaction-diffusion systems on networks}

In the above section we establish that the $p$-adic continuous model
(\ref{Eq_14}), which we call the \textit{mean-field approximation} of the
original system (\ref{Eq_6}), can be very well approximated by discretization
(\ref{Eq_15}) for $M$ sufficiently large. The purpose of this section is to
study the relations between the solutions of systems (\ref{Eq_6}) and
(\ref{Eq_15}).

We recall that the matrix of operator $\boldsymbol{L}_{N}$ acting on $X_{N}$
is $\left[  A_{JI}-\gamma_{I}\delta_{JI}\right]  _{J,I\in G_{N}^{0}}$. We set
$\mathbb{A}_{N;M}:=\left[  p^{N-M}A_{JI}-\gamma_{I}\delta_{JI}\right]
_{J,I\in G_{N}^{0}}$.

\begin{lemma}
\label{Lemma3}The matrix $\mathbb{A}^{\left(  M\right)  }$ of the operator
$\boldsymbol{L}_{M}$ acting on $X_{M}$ is a diagonal-type matrix of the form%
\begin{equation}
\mathbb{A}^{\left(  M\right)  }=\left[
\begin{array}
[c]{ccccc}%
\mathbb{A}_{N;M} &  &  &  & \\
& \ddots &  &  & \\
&  & \mathbb{A}_{N;M} &  & \\
&  &  & \ddots & \\
&  &  &  & \mathbb{A}_{N;M}%
\end{array}
\right]  _{p^{M-N}\times p^{M-N}}, \label{Eq_Matrix_A_M}%
\end{equation}
after renaming the elements of the basis of $X_{M}$.
\end{lemma}

\begin{proof}
We first note that kernel $J_{N}(x,y)$, see (\ref{Eq_Kernel_J_N}), can be
rewritten as%
\begin{gather*}
J_{N}(x,y)=\\
p^{N}%
%TCIMACRO{\tsum \limits_{J\in G_{N}^{0}}}%
%BeginExpansion
{\textstyle\sum\limits_{J\in G_{N}^{0}}}
%EndExpansion%
%TCIMACRO{\tsum \limits_{K\in G_{N}^{0}}}%
%BeginExpansion
{\textstyle\sum\limits_{K\in G_{N}^{0}}}
%EndExpansion%
%TCIMACRO{\tsum \limits_{J_{j}\in G_{J}^{M}}}%
%BeginExpansion
{\textstyle\sum\limits_{J_{j}\in G_{J}^{M}}}
%EndExpansion%
%TCIMACRO{\tsum \limits_{K_{j}\in G_{K}^{M}}}%
%BeginExpansion
{\textstyle\sum\limits_{K_{j}\in G_{K}^{M}}}
%EndExpansion
A_{JK}\Omega\left(  p^{M}\left\vert x-J_{j}\right\vert _{p}\right)
\Omega\left(  p^{M}\left\vert y-K_{j}\right\vert _{p}\right)  \text{, }%
\end{gather*}
for $\left(  x,y\right)  \in\mathbb{Q}_{p}$. In order to compute the matrix
$\mathbb{A}^{\left(  M\right)  }$, it is necessary to compute
\[
\boldsymbol{L}_{M}\Omega\left(  p^{M}\left\vert x-I_{j}\right\vert
_{p}\right)  \text{ \ for \ }I_{j}\in G_{N}^{M}.
\]
We first compute%
\begin{gather}
\int\limits_{\mathcal{K}_{N}}\Omega\left(  p^{M}\left\vert y-I_{j}\right\vert
_{p}\right)  J_{N}(x,y)dy=\label{Eq_34}\\
p^{N}%
%TCIMACRO{\tsum \limits_{J\in G_{N}^{0}}}%
%BeginExpansion
{\textstyle\sum\limits_{J\in G_{N}^{0}}}
%EndExpansion%
%TCIMACRO{\tsum \limits_{J_{j}\in G_{J}^{M}}}%
%BeginExpansion
{\textstyle\sum\limits_{J_{j}\in G_{J}^{M}}}
%EndExpansion
A_{JI}\Omega\left(  p^{M}\left\vert x-J_{j}\right\vert _{p}\right)
\int\limits_{\mathcal{K}_{N}}\Omega\left(  p^{M}\left\vert y-I_{j}\right\vert
_{p}\right)  dy\nonumber\\
=p^{N-M}%
%TCIMACRO{\tsum \limits_{J_{j}\in G_{N}^{M}}}%
%BeginExpansion
{\textstyle\sum\limits_{J_{j}\in G_{N}^{M}}}
%EndExpansion
A_{JI}\Omega\left(  p^{M}\left\vert x-J_{j}\right\vert _{p}\right)  .\nonumber
\end{gather}
We now compute%
\begin{gather}
\Omega\left(  p^{M}\left\vert x-I_{j}\right\vert _{p}\right)  \int
\limits_{\mathcal{K}_{N}}J_{N}(x,y)dy=\label{Eq_35}\\
p^{N-M}\left\{
%TCIMACRO{\tsum \limits_{K\in G_{N}^{0}}}%
%BeginExpansion
{\textstyle\sum\limits_{K\in G_{N}^{0}}}
%EndExpansion%
%TCIMACRO{\tsum \limits_{K_{j}\in G_{K}^{M}}}%
%BeginExpansion
{\textstyle\sum\limits_{K_{j}\in G_{K}^{M}}}
%EndExpansion
A_{IK}\right\}  \Omega\left(  p^{M}\left\vert x-I_{j}\right\vert _{p}\right)
=\nonumber\\
p^{N-M}\left\{
%TCIMACRO{\tsum \limits_{K\in G_{N}^{0}}}%
%BeginExpansion
{\textstyle\sum\limits_{K\in G_{N}^{0}}}
%EndExpansion
p^{M-N}A_{IK}\right\}  \Omega\left(  p^{M}\left\vert x-I_{j}\right\vert
_{p}\right)  =\gamma_{I}\Omega\left(  p^{M}\left\vert x-I_{j}\right\vert
_{p}\right)  .\nonumber
\end{gather}
From (\ref{Eq_34})-(\ref{Eq_35}) we have%
\[
\boldsymbol{L}_{M}\Omega\left(  p^{M}\left\vert x-I_{j}\right\vert
_{p}\right)  =%
%TCIMACRO{\tsum \limits_{J_{j}\in G_{N}^{M}}}%
%BeginExpansion
{\textstyle\sum\limits_{J_{j}\in G_{N}^{M}}}
%EndExpansion
\left[  p^{N-M}A_{JI}-\gamma_{I}\delta_{J_{j}I_{j}}\right]  \Omega\left(
p^{M}\left\vert x-J_{j}\right\vert _{p}\right)  .
\]
Consequently, the matrix of $\boldsymbol{L}_{M}$ is%
\begin{equation}
\left[  p^{N-M}A_{JI}-\gamma_{I}\delta_{J_{j}I_{j}}\right]  _{J_{j},I_{j}\in
G_{N}^{M}}. \label{Eq_matrix_T_M}%
\end{equation}
Given $I\in G_{N}^{0}$, which corresponds to the center of a ball of type
$I+p^{N}\mathbb{Z}_{p}$, there are $p^{M-N}$, $I_{j}$s in $G_{I}^{M}$, which
correspond to the centers of the balls $I_{j}+p^{M}\mathbb{Z}_{p}\subset
I+p^{N}\mathbb{Z}_{p}$, then given any pair $\left(  J_{j},I_{j}\right)  \in
G_{J}^{M}\times G_{I}^{M}$, the expression $p^{N-M}A_{JI}-\gamma_{I}%
\delta_{J_{j}I_{j}}$ is constant. More precisely, this value occurs
$p^{2\left(  M-N\right)  }$-times when $\left(  I_{j},J_{j}\right)  $ runs
through $G_{N}^{M}\times G_{N}^{M}$. This implies, that after renaming the
elements of the basis of $X_{M}$, the matrix of $\boldsymbol{L}_{M}$ can be
rewritten as (\ref{Eq_Matrix_A_M}).
\end{proof}

The matrix $\mathbb{A}^{\left(  M\right)  }$, see (\ref{Eq_Matrix_A_M}),
corresponds to a network constructed by using $p^{M-N}$ \textit{replicas of
the original network}, each of these replicas correspond to a network having a
diffusion operator of type
\[
\left[  p^{N-M}A_{JI}-\gamma_{I}\delta_{JI}\right]  _{J,I\in G_{N}^{0}%
}=p^{N-M}\left[  A_{JI}-p^{M-N}\gamma_{I}\delta_{JI}\right]  _{J,I\in
G_{N}^{0}},
\]
and the corresponding reaction-diffusion equation is
\begin{equation}
\left\{
\begin{array}
[c]{l}%
\frac{\partial u^{\left(  N\right)  }\left(  x,t\right)  }{\partial
t}=f(u^{\left(  N\right)  }\left(  x,t\right)  ,v^{\left(  N\right)  }\left(
x,t\right)  )+\varepsilon^{\prime}\boldsymbol{L}_{N,\lambda}u^{\left(
N\right)  }\left(  x,t\right) \\
\\
\frac{\partial v^{\left(  N\right)  }\left(  x,t\right)  }{\partial
t}=g(u^{\left(  N\right)  }\left(  x,t\right)  ,v^{\left(  N\right)  }\left(
x,t\right)  )+\varepsilon^{\prime}d\boldsymbol{L}_{N,\lambda}v^{\left(
N\right)  }\left(  x,t\right)  ,
\end{array}
\right.  \label{Eq_36}%
\end{equation}
where $\varepsilon^{\prime}=p^{N-M}\varepsilon$, $\lambda=p^{M-N}$, where
$\boldsymbol{L}_{N,\lambda}:X_{N}\rightarrow X_{N}$ is defined in
(\ref{Eq_ope_L_lambda_XM}). By using the results of Section
\ref{Section_selsimilarity_XM} and Remark \ref{Nota_replica}, we say that
system (\ref{Eq_36}) is a scaled replica of the original system (\ref{Eq_6}).

In conclusion, system (\ref{Eq_14}) is the `limit' of system (\ref{Eq_15})
when $M$ tends to infinity. This means that both systems have solutions in the
time interval $\left[  0,\tau_{0}\right)  $, and these two solutions are
numerically very close for $M$ sufficiently large (Theorem \ref{Theorem3}). In
turn, any solution of system (\ref{Eq_15}) corresponds to $p^{M-N}$ solutions
of $p^{M-N}$ systems of type (\ref{Eq_35}), each of them is a scaled version
(a replica) of the original system (\ref{Eq_6}).

\section{\label{Section_spectrum_L}The spectrum of operator $\boldsymbol{L}$}

\begin{remark}
From now on, we assume that $\mathcal{G}$ is an unoriented graph, with a
symmetric adjacency matrix $\left[  A_{JI}\right]  _{J,I\in G_{N}^{0}}$ such
\ that its diagonal contains only zeros. In this case, the spectrum of the
discrete Laplacian $\left[  L_{JI}\right]  _{J,I\in G_{N}^{0}}$ is
well-understood, see e.g. \cite[Section 4.1]{Van Mieghem}. Since the adjacency
matrix $\left[  A_{JI}\right]  _{J,I\in G_{N}^{0}}$ is symmetric, the
eigenvalues, $\mu_{I}$, $I\in G_{N}^{0}$, of $\left[  L_{JI}\right]  _{J,I\in
G_{N}^{0}}$ are non-positive and $\max_{I\in G_{N}^{0}}\left\{  \mu
_{I}\right\}  =0$. We denote by $mult(\mu_{I})$ the multiplicity of the
eigenvalue $\mu_{I}$. We set $X_{\bullet}\otimes\mathbb{C}$ for the
complexification of $X_{\bullet}$. In particular, $X_{\infty}\otimes
\mathbb{C=}C\left(  \mathcal{K}_{N},\mathbb{C}\right)  $, with the $L^{\infty
}$-norm. Then $\boldsymbol{L}:X_{\infty}\otimes\mathbb{C\rightarrow}X_{\infty
}\otimes\mathbb{C}$ is a linear bounded operator. We set $\boldsymbol{L}%
_{M}:=\boldsymbol{L}\mid_{X_{M}\otimes\mathbb{C}}$.
\end{remark}

\begin{lemma}
\label{Lemma5}The operator $\boldsymbol{L}$ has a unique extension to
$L^{2}(\mathcal{K}_{N},\mathbb{C})$ as a bounded linear operator.
\end{lemma}

\begin{proof}
We first establish the following assertion:

\textbf{Claim 1. }Let $f\in C\left(  \mathcal{K}_{N},\mathbb{C}\right)  $,
then $\left\Vert \boldsymbol{L}f\right\Vert _{2}\leq2\sqrt{\sum_{J\in
G_{N}^{0}}\gamma_{J}^{2}}\left\Vert f\right\Vert _{2}$.

\noindent By using that
\[
\boldsymbol{L}f(x)=%
%TCIMACRO{\tint \limits_{\mathcal{K}_{N}}}%
%BeginExpansion
{\textstyle\int\limits_{\mathcal{K}_{N}}}
%EndExpansion
f\left(  y\right)  J_{N}(x,y)dy-f\left(  x\right)
%TCIMACRO{\tint \limits_{\mathcal{K}_{N}}}%
%BeginExpansion
{\textstyle\int\limits_{\mathcal{K}_{N}}}
%EndExpansion
J_{N}(x,y)dy,
\]
we get%
\[
\left\Vert \boldsymbol{L}f\right\Vert _{2}\leq\left\Vert
%TCIMACRO{\tint \limits_{\mathcal{K}_{N}}}%
%BeginExpansion
{\textstyle\int\limits_{\mathcal{K}_{N}}}
%EndExpansion
f\left(  y\right)  J_{N}(x,y)dy\right\Vert _{2}+\left\Vert f\left(  x\right)
%TCIMACRO{\tint \limits_{\mathcal{K}_{N}}}%
%BeginExpansion
{\textstyle\int\limits_{\mathcal{K}_{N}}}
%EndExpansion
J_{N}(x,y)dy\right\Vert _{2}=:\mathcal{I}_{1}+\mathcal{I}_{2}.
\]
We consider first the term $\mathcal{I}_{1}$. By using Cauchy-Schwarz
inequality,%
\begin{gather*}
\mathcal{I}_{1}^{2}=%
%TCIMACRO{\tint \limits_{\mathcal{K}_{N}}}%
%BeginExpansion
{\textstyle\int\limits_{\mathcal{K}_{N}}}
%EndExpansion
\left\vert
%TCIMACRO{\tint \limits_{\mathcal{K}_{N}}}%
%BeginExpansion
{\textstyle\int\limits_{\mathcal{K}_{N}}}
%EndExpansion
f\left(  y\right)  J_{N}(x,y)dy\right\vert ^{2}dx\leq%
%TCIMACRO{\tint \limits_{\mathcal{K}_{N}}}%
%BeginExpansion
{\textstyle\int\limits_{\mathcal{K}_{N}}}
%EndExpansion
\left\{
%TCIMACRO{\tint \limits_{\mathcal{K}_{N}}}%
%BeginExpansion
{\textstyle\int\limits_{\mathcal{K}_{N}}}
%EndExpansion
\left\vert f\left(  y\right)  \right\vert J_{N}(x,y)dy\right\}  ^{2}dx\\
=%
%TCIMACRO{\tint \limits_{\mathcal{K}_{N}}}%
%BeginExpansion
{\textstyle\int\limits_{\mathcal{K}_{N}}}
%EndExpansion
\left\{  p^{N}%
%TCIMACRO{\tsum \limits_{J,I\in G_{N}^{0}}}%
%BeginExpansion
{\textstyle\sum\limits_{J,I\in G_{N}^{0}}}
%EndExpansion
A_{JI}\Omega\left(  p^{N}\left\vert x-J\right\vert _{p}\right)
%TCIMACRO{\tint \limits_{\mathcal{K}_{N}}}%
%BeginExpansion
{\textstyle\int\limits_{\mathcal{K}_{N}}}
%EndExpansion
\left\vert f\left(  y\right)  \right\vert \Omega\left(  p^{N}\left\vert
y-I\right\vert _{p}\right)  dy\right\}  ^{2}dx\\
\leq%
%TCIMACRO{\tint \limits_{\mathcal{K}_{N}}}%
%BeginExpansion
{\textstyle\int\limits_{\mathcal{K}_{N}}}
%EndExpansion
\left\{  \left\Vert f\right\Vert _{2}p^{\frac{N}{2}}%
%TCIMACRO{\tsum \limits_{J,I\in G_{N}^{0}}}%
%BeginExpansion
{\textstyle\sum\limits_{J,I\in G_{N}^{0}}}
%EndExpansion
A_{JI}\Omega\left(  p^{N}\left\vert x-J\right\vert _{p}\right)  \right\}
^{2}dx\\
=\left\Vert f\right\Vert _{2}^{2}p^{N}%
%TCIMACRO{\tsum \limits_{J}}%
%BeginExpansion
{\textstyle\sum\limits_{J}}
%EndExpansion
\gamma_{J}^{2}%
%TCIMACRO{\tint \limits_{\mathcal{K}_{N}}}%
%BeginExpansion
{\textstyle\int\limits_{\mathcal{K}_{N}}}
%EndExpansion
\Omega\left(  p^{N}\left\vert x-J\right\vert _{p}\right)  dx=\left\Vert
f\right\Vert _{2}^{2}\sum_{J\in G_{N}^{0}}\gamma_{J}^{2}.
\end{gather*}
We now consider $\mathcal{I}_{2}$:%
\begin{gather*}
\mathcal{I}_{2}^{2}=%
%TCIMACRO{\tint \limits_{\mathcal{K}_{N}}}%
%BeginExpansion
{\textstyle\int\limits_{\mathcal{K}_{N}}}
%EndExpansion
\left\vert f\left(  x\right)  \right\vert ^{2}\left\{
%TCIMACRO{\tint \limits_{\mathcal{K}_{N}}}%
%BeginExpansion
{\textstyle\int\limits_{\mathcal{K}_{N}}}
%EndExpansion
J_{N}(x,y)dy\right\}  ^{2}dx\\
=%
%TCIMACRO{\tint \limits_{\mathcal{K}_{N}}}%
%BeginExpansion
{\textstyle\int\limits_{\mathcal{K}_{N}}}
%EndExpansion
\left\vert f\left(  x\right)  \right\vert ^{2}\left\{
%TCIMACRO{\tsum \limits_{J\in G_{N}^{0}}}%
%BeginExpansion
{\textstyle\sum\limits_{J\in G_{N}^{0}}}
%EndExpansion
\gamma_{J}\text{ }\Omega\left(  p^{N}\left\vert x-J\right\vert _{p}\right)
\right\}  ^{2}dx\\
=%
%TCIMACRO{\tsum \limits_{J\in G_{N}^{0}}}%
%BeginExpansion
{\textstyle\sum\limits_{J\in G_{N}^{0}}}
%EndExpansion
\gamma_{J}^{2}%
%TCIMACRO{\tint \limits_{\mathcal{K}_{N}}}%
%BeginExpansion
{\textstyle\int\limits_{\mathcal{K}_{N}}}
%EndExpansion
\left\vert f\left(  x\right)  \right\vert ^{2}\Omega\left(  p^{N}\left\vert
x-J\right\vert _{p}\right)  dx\leq\left\Vert f\right\Vert _{2}^{2}%
%TCIMACRO{\tsum \limits_{J\in G_{N}^{0}}}%
%BeginExpansion
{\textstyle\sum\limits_{J\in G_{N}^{0}}}
%EndExpansion
\gamma_{J}^{2}.
\end{gather*}
The announced result follows from Claim 1, by using the fact that
$\mathcal{D}\left(  \mathcal{K}_{N},\mathbb{C}\right)  \subset C\left(
\mathcal{K}_{N},\mathbb{C}\right)  $ is dense in $L^{2}(\mathcal{K}%
_{N},\mathbb{C})$, see e.g. \cite[Proposition 4.3.3]{Alberio et al}.
\end{proof}

\begin{remark}
\label{Nota_basi_1} The eigenvalues of $\boldsymbol{L}\mid_{X_{N}%
\otimes\mathbb{C}}=\boldsymbol{L}_{N}$ are exactly the eigenvalues of the
matrix $\left[  L_{JI}\right]  _{J,I\in G_{N}^{0}}$, which are $\mu_{I}\leq0$,
$I\in G_{N}^{0}$ with $\max_{I\in G_{N}^{0}}\left\{  \mu_{I}\right\}  =0$. We
denote the eigenfunctions of $\left[  L_{JI}\right]  _{J,I\in G_{N}^{0}}$ as
$\varphi_{I}$, $I\in G_{N}^{0}$.

Let $\left[  c_{J}^{I}\right]  _{J\in G_{N}^{0}}$ be an eigenvector
corresponding to $\mu_{I}$, by identifying it with the function
\[
\varphi_{I}\left(  x\right)  :=\sum\limits_{J\in G_{N}^{0}}c_{J}^{I}%
\Omega\left(  p^{N}\left\vert x-J\right\vert _{p}\right)  \in X_{N}%
\otimes\mathbb{C}\text{, }c_{J}^{I}\in\mathbb{C}\text{,}%
\]
and by using that $X_{N}\otimes\mathbb{C\hookrightarrow}X_{\infty}%
\otimes\mathbb{C}$ and that $\boldsymbol{L}:X_{N}\otimes\mathbb{C\rightarrow
}X_{N}\otimes\mathbb{C}$, we have%
\[
\left\{
\begin{array}
[c]{c}%
\varphi_{I}\in X_{\infty}\otimes\mathbb{C};\\
\\
\boldsymbol{L}\varphi_{I}=\mu_{I}\varphi_{I}.
\end{array}
\right.
\]
The $\varphi_{I}$s form a $\mathbb{C}$-vector space of dimension mult$\left(
\mu_{I}\right)  $.
\end{remark}

\begin{remark}
\label{Nota_th_Kozyrev}We now recall that the set of functions $\left\{
\Psi_{rnj}\right\}  $ defined as%
\begin{equation}
\Psi_{rnj}\left(  x\right)  =p^{\frac{-r}{2}}\chi_{p}\left(  p^{r-1}jx\right)
\Omega\left(  \left\vert p^{r}x-n\right\vert _{p}\right)  , \label{Eq_40}%
\end{equation}
where $r\in\mathbb{Z}$, $j\in\left\{  1,\cdots,p-1\right\}  $, and $n$ runs
through a fixed set of representatives of $\mathbb{Q}_{p}/\mathbb{Z}_{p}$, is
an orthonormal basis of $L^{2}(\mathbb{Q}_{p})$. Furthermore,
\begin{equation}
\int_{\mathbb{Q}_{p}}\Psi_{rnj}\left(  x\right)  dx=0. \label{Eq_41}%
\end{equation}
This result is due to Kozyrev see e.g. \cite[Theorem 3.29]{KKZuniga} or
\cite[Theorem 9.4.2]{Alberio et al}.
\end{remark}

\begin{remark}
\label{Nota_basis_2}Any function of the the form $\Psi_{rnj}\left(  x\right)
$ supported in $\mathcal{K}_{N}=\bigsqcup\nolimits_{I\in G_{N}^{0}}\left(
I+p^{N}\mathbb{Z}_{p}\right)  $ satisfies
\begin{equation}
\boldsymbol{L}\Psi_{rnj}\left(  x\right)  =-\gamma_{I}\Psi_{rnj} \label{Eq_43}%
\end{equation}
for some $I\in G_{N}^{0}$, $j\in\left\{  1,\cdots,p-1\right\}  $.
\end{remark}

\begin{theorem}
\label{Theorem4}The elements of the set:%
\[
\left\{  \mu_{I};I\in G_{N}^{0}\setminus\left\{  I_{0}\right\}  \right\}
%TCIMACRO{\tbigsqcup }%
%BeginExpansion
{\textstyle\bigsqcup}
%EndExpansion
\left\{  -\gamma_{I};I\in G_{N}^{0}\right\}  \subset\left(  -\infty,0\right)
,
\]
where $\left\{  \mu_{I}\right\}  _{I\in G_{N}^{0}\setminus\left\{
I_{0}\right\}  }$ are the non-zero eigenvalues of matrix $\left[
L_{JI}\right]  _{J,I\in G_{N}^{0}}$, are the non-zero eigenvalues of
$\boldsymbol{L}$. The corresponding eigenfunctions are%
\begin{equation}
\left\{  \frac{\varphi_{I}}{\left\Vert \varphi_{I}\right\Vert _{2}};I\in
G_{N}^{0}\right\}
%TCIMACRO{\tbigsqcup }%
%BeginExpansion
{\textstyle\bigsqcup}
%EndExpansion
\left\{  \Psi_{rnj};\text{ supp}\Psi_{rnj}\subset\mathcal{K}_{N}\right\}  ,
\label{Eq_Basis}%
\end{equation}
where the functions $\varphi_{I}$, $\Psi rnj$ are defined in Remarks
\ref{Nota_basi_1}, \ref{Nota_basis_2}, respectively. Furthermore, the set
(\ref{Eq_Basis}) is an orthonormal basis of $L^{2}(\mathcal{K}_{N}%
,\mathbb{C})$, and
\begin{equation}
L^{2}(\mathcal{K}_{N},\mathbb{C})=X_{N}\otimes\mathbb{C}\text{ }%
\mathbb{\oplus}\text{ }\mathcal{L}_{0}^{2}(\mathcal{K}_{N},\mathbb{C}),
\label{Eq_50}%
\end{equation}
where
\[
\mathcal{L}_{0}^{2}(\mathcal{K}_{N},\mathbb{C}):=\left\{  f\in L^{2}%
(\mathcal{K}_{N},\mathbb{C});\int_{\mathcal{K}_{N}}fdx=0\right\}  .
\]

\end{theorem}

\begin{proof}
By Remarks (\ref{Nota_basi_1}), (\ref{Nota_basis_2}), it is sufficient to show
(\ref{Eq_50}). This assertion follows from Propositions 1, 2 in
\cite{Zuniga-Galindo-PNAS}.
\end{proof}

\section{\label{Section_Turing_criteria}Turing Instability}

\subsection{Turing Criteria}

We now consider a homogeneous steady state $\left(  u_{0},v_{0}\right)  $ of
(\ref{Eq_14}), which is a nonnegative solution of
\begin{equation}
f(u,v)=g(u,v)=0. \label{EQ_6}%
\end{equation}
Since $u$, $v$ are real-valued functions, to study the linear stability of
$\left(  u_{0},v_{0}\right)  $, we can use the classical results. Following
Turing, in the absence of any spatial variation, the homogeneous state must be
linearly stable. With no spatial variation $u$, $v$ satisfy%
\begin{equation}
\left\{
\begin{array}
[c]{c}%
\frac{\partial u}{\partial t}(x,t)=f\left(  u,v\right) \\
\\
\frac{\partial v}{\partial t}(x,t)=g\left(  u,v\right)  .
\end{array}
\right.  \label{EQ_7}%
\end{equation}
Notice that (\ref{EQ_7}) \ is an ordinary system of differential equations in
$\mathbb{R}^{2}$.

Now, for $\delta>0$ sufficiently small and $\left(  u_{0},v_{0}\right)  $ as
in (\ref{EQ_6}), we define%
\begin{multline*}
U_{\delta,u_{0}}\oplus U_{\delta,v_{0}}=\\
\left\{  u_{1}\oplus u_{2}\in C\left(  \mathcal{K}_{N},\mathbb{R}\right)
\oplus C\left(  \mathcal{K}_{N},\mathbb{R}\right)  ;\left\Vert u_{1}%
-u_{0}\right\Vert _{\infty}<\delta\text{, }\left\Vert v_{1}-v_{0}\right\Vert
_{\infty}<\delta\right\}  .
\end{multline*}
Then, the Cauchy \ problem:%
\begin{equation}
\left\{
\begin{array}
[c]{l}%
u\oplus v\in C^{1}\left(  \left[  0,\tau_{0}\right)  ,U_{\delta,u_{0}}\oplus
U_{\delta,v_{0}}\right)  ;\\
\\
\frac{\partial}{\partial t}\left[
\begin{array}
[c]{l}%
u\left(  t\right) \\
\\
v\left(  t\right)
\end{array}
\right]  =\left[
\begin{array}
[c]{l}%
f(u\left(  t\right)  ,v\left(  t\right)  )\\
\\
g(u\left(  t\right)  ,v\left(  t\right)  )
\end{array}
\right]  +\varepsilon\boldsymbol{L}\mathbb{D}\left[
\begin{array}
[c]{l}%
u\left(  t\right) \\
\\
v\left(  t\right)
\end{array}
\right]  \text{;}\\
\\
u\left(  0\right)  \oplus v\left(  0\right)  \in U_{\delta,u_{0}}\oplus
U_{\delta,v_{0}},
\end{array}
\right.  \label{EQ_2}%
\end{equation}
where%
\[
\mathbb{D}{\scriptsize =}\left[
\begin{array}
[c]{ccc}%
1 &  & 0\\
&  & \\
0 &  & d
\end{array}
\right]  ,
\]
has a classical solution, cf. Theorem \ref{Theorem2} and Lemma
\ref{Lemma_mild_sol}. Our goal is to give an asymptotic profile as $t$ tends
infinity of this mild solution (the Turing instability criteria). We linearize
system (\ref{EQ_2}) about the steady state $\left(  u_{0},v_{0}\right)  $, by
setting%
\begin{equation}
\boldsymbol{w}=\left[
\begin{array}
[c]{c}%
w_{1}\\
\\
w_{2}%
\end{array}
\right]  =\left[
\begin{array}
[c]{c}%
u-u_{0}\\
\\
v-v_{0}%
\end{array}
\right]  . \label{EQ_8}%
\end{equation}
By using the fact that $f$ and $g$ are differentiable, and assuming that
$\left\Vert \boldsymbol{w}\right\Vert =\left\Vert w_{1}\oplus w_{2}\right\Vert
$ is sufficiently small, then (\ref{EQ_7}) can be approximated as%
\begin{equation}
\frac{\partial\boldsymbol{w}}{\partial t}(x,t)=\mathbb{J}\boldsymbol{w}\text{,
} \label{EQ_9}%
\end{equation}
where \
\[
\mathbb{J}_{u_{0},v_{0}}=:\mathbb{J}=\left[
\begin{array}
[c]{ccc}%
\frac{\partial f}{\partial u} &  & \frac{\partial f}{\partial v}\\
&  & \\
\frac{\partial g}{\partial u} &  & \frac{\partial g}{\partial v}%
\end{array}
\right]  \left(  u_{0},v_{0}\right)  =:\left[
\begin{array}
[c]{ccc}%
f_{u_{0}} &  & f_{v_{0}}\\
&  & \\
g_{u_{0}} &  & g_{v_{0}}%
\end{array}
\right]  .
\]
We now look for solutions of (\ref{EQ_9}) of the form%
\begin{equation}
\boldsymbol{w}\left(  t;\lambda\right)  =e^{\lambda t}\boldsymbol{w}_{0}.
\label{EQ_10}%
\end{equation}
By substituting (\ref{EQ_10}) in (\ref{EQ_9}), the eigenvalues $\lambda$ are
the solutions of
\[
\det\left(  \mathbb{J}-\lambda\mathbb{I}\right)  =0,
\]
i.e.
\begin{equation}
\lambda^{2}-\left(  \mathit{Tr}\mathbb{J}\right)  \lambda+\det\mathbb{J}=0.
\label{EQ_10AB}%
\end{equation}
Consequently%
\begin{equation}
\lambda=\frac{1}{2}\left\{  \mathit{Tr}\mathbb{J}\text{ }\pm\sqrt{\left(
\mathit{Tr}\mathbb{J}\right)  ^{2}-4\det\mathbb{J}}\right\}  . \label{EQ_10AA}%
\end{equation}
The steady state $\boldsymbol{w}=\boldsymbol{0}$ is linearly stable if
$\operatorname{Re}\lambda<0$, this last condition is guaranteed if
\begin{equation}
\mathit{Tr}\mathbb{J}<0\text{ \ and \ }\det\mathbb{J}>0. \label{EQ_10A}%
\end{equation}
We now consider the full reaction-ultradiffusion system (\ref{EQ_2}). We
linearize it about the steady state, which with (\ref{EQ_8}) is
$\boldsymbol{w}=\boldsymbol{0}:\boldsymbol{=}\left[
\begin{array}
[c]{c}%
0\\
0
\end{array}
\right]  $, to get%
\begin{equation}
\left\{
\begin{array}
[c]{l}%
u\oplus v\in C^{1}\left(  \left[  0,\tau\right)  ,U_{\delta,u_{0}}\oplus
U_{\delta,v_{0}}\right)  \text{;}\\
\\
\frac{\partial}{\partial t}\boldsymbol{w}(x,t)=\left(  \mathbb{J}%
+\varepsilon\boldsymbol{L}\mathbb{D}\right)  \boldsymbol{w}(x,t)\text{, }%
t\in\left[  0,\tau\right)  \text{;}\\
\\
u\left(  0\right)  \oplus v\left(  0\right)  \in U_{\delta,u_{0}}\oplus
U_{\delta,v_{0}},
\end{array}
\right.  \label{EQ_12}%
\end{equation}
where $\mathbb{J}+\varepsilon\boldsymbol{L}\mathbb{D}$ is a strongly
continuous semigroup on $C\left(  \mathcal{K}_{N},\mathbb{R}\right)  \oplus
C\left(  \mathcal{K}_{N},\mathbb{R}\right)  $, see e.g. \cite[Corollary
5.1.3]{Milan}, then (\ref{EQ_12}) has a mild solution, which is differentiable
and unique, see e.g. \cite[Theorems 5.1.2, 4.3.1]{Milan}. Furthermore,
(\ref{EQ_12}), has also a unique solution, when $\boldsymbol{L}$ is considered
as an operator on $L^{2}\left(  \mathcal{K}_{N},\mathbb{C}\right)  $, cf.
Lemma \ref{Lemma5}, for this reason, we can use the orthonormal basis given in
Theorem \ref{Theorem4} to solve (\ref{EQ_12}) in $L^{2}\left(  \mathcal{K}%
_{N},\mathbb{C}\right)  $, by using the separation of variables method, then,
the solution of the original problem is exactly the real part of the solution
of (\ref{EQ_12}) in $L^{2}\left(  \mathcal{K}_{N},\mathbb{C}\right)  $.

To solve the system (\ref{EQ_12}) in $L^{2}\left(  \mathcal{K}_{N}%
,\mathbb{C}\right)  $, we first consider the following eigenvalue pro\-blem:%
\begin{equation}
\left\{
\begin{array}
[c]{l}%
\boldsymbol{L}\mathbb{D}\boldsymbol{w}_{\kappa}(x)=\kappa\boldsymbol{w}%
_{\kappa}(x)\\
\\
\boldsymbol{w}_{\kappa}\in L^{2}\left(  \mathcal{K}_{N},\mathbb{C}\right)
\oplus L^{2}\left(  \mathcal{K}_{N},\mathbb{C}\right)  ,
\end{array}
\right.  \label{EQ_14}%
\end{equation}
which has a solution $\boldsymbol{w}_{\kappa}=w_{\kappa,1}\oplus w_{\kappa,2}$
due to Theorem \ref{Theorem4}, where
\[
w_{\kappa,1},w_{\kappa,2}\in\left\{  \frac{\varphi_{I}}{\left\Vert \varphi
_{I}\right\Vert _{2}};I\in G_{N}^{0}\right\}
%TCIMACRO{\tbigsqcup }%
%BeginExpansion
{\textstyle\bigsqcup}
%EndExpansion
\left\{  \Psi_{rnj};\text{supp}\Psi_{rnj}\subset\mathcal{K}_{N}\right\}  .
\]
We look for an solution of type
\begin{equation}
\boldsymbol{w}(x,t)=\sum\limits_{rnj}\boldsymbol{a}_{rnj}e^{\lambda t}%
\Psi_{rnj}+\sum\limits_{I\in G_{N}^{0}}\boldsymbol{b}_{I}\varphi
_{I}\label{EQ_15A}%
\end{equation}
where the vectors $\boldsymbol{a}_{rnj}$, $\boldsymbol{b}_{I}$ are determined
by the Fourier expansion of the initial conditions. Substituting
(\ref{EQ_15A}) with (\ref{EQ_14}) in (\ref{EQ_12}), we obtain that the
existence of a non-trivial solution $\boldsymbol{w}(x,t)$ requires that the
$\lambda$s satisfy%
\begin{equation}
\det\left(  \lambda\mathbb{I}-\mathbb{J}-\varepsilon\kappa\mathbb{D}\right)
=0,\label{EQ_16A}%
\end{equation}
i.e.,
\begin{equation}
\lambda^{2}-\left\{  \left(  1+d\right)  \varepsilon\kappa+\mathit{Tr}%
\mathbb{J}\right\}  \lambda+h\left(  \kappa\right)  =0,\label{EQ_16}%
\end{equation}
where%
\begin{equation}
h\left(  \kappa\right)  :=\varepsilon^{2}d\kappa^{2}+\varepsilon\kappa\left(
df_{u_{0}}+g_{v_{0}}\right)  +\det\mathbb{J}.\label{EQ_17}%
\end{equation}
Notice that condition (\ref{EQ_16A}) becomes condition (\ref{EQ_10AB}) when
$\kappa=0$. By using the classical reasoning, see e.g. \cite[Chapter
2]{Murray}, one obtains the necessary conditions (T1)-(T4) given in Theorem
\ref{Theorem5}.

The analysis of the sufficient conditions is similar to the classical case,
see e.g. \cite[Chapter 2]{Murray}. One shows that the condition $h\left(
\kappa\right)  <0$, for some $\kappa\neq0$, requires that
\begin{equation}
\frac{\left(  df_{u_{0}}+g_{v_{0}}\right)  ^{2}}{4d}>\det\mathbb{J}.
\label{EQ_21}%
\end{equation}
Furthermore, there exists a critical diffusion $d_{c}$, which is given as an
appropriate root of
\begin{equation}
f_{u_{0}}^{2}d_{c}^{2}+2\left(  2f_{v_{0}}g_{u_{0}}-f_{u_{0}}g_{v_{0}}\right)
d_{c}+g_{v_{0}}^{2}=0. \label{EQ_26B}%
\end{equation}
For $d>d_{c}$ model ((\ref{EQ_2})) exhibits Turing instability, while for
$d<d_{c}$ no. When $d>d_{c}$, there exists a range of unstable of positive
wavenumbers \ $\kappa_{1}<\kappa<\kappa_{2}$, where $\kappa_{1}$, $\kappa_{2}$
are the zeros of $h\left(  \kappa\right)  =0$, see (\ref{EQ_17}) and
(\ref{EQ_21}):
\[
\kappa_{2}=\frac{-1}{2d\varepsilon}\left\{  \left(  df_{u_{0}}+g_{v_{0}%
}\right)  -\sqrt{\left(  df_{u_{0}}+g_{v_{0}}\right)  ^{2}-4d\det\mathbb{J}%
}\right\}  <0,
\]%
\[
\kappa_{1}=\frac{-1}{2d\varepsilon}\left\{  \left(  df_{u_{0}}+g_{v_{0}%
}\right)  +\sqrt{\left(  df_{u_{0}}+g_{v_{0}}\right)  ^{2}-4d\det\mathbb{J}%
}\right\}  <0.
\]
Notice that, within the unstable range, $\operatorname{Re}\lambda\left(
\kappa\right)  >0$ has a maximum for the wavenumber $\kappa_{\text{min}}%
=\frac{-\left(  df_{u_{0}}+g_{v_{0}}\right)  }{2\varepsilon d_{c}}$.

In the solution $\boldsymbol{w}\left(  x,t\right)  $ given by (\ref{EQ_15A}),
the dominant contributions as $t$ increases are the modes for which
$\operatorname{Re}\lambda\left(  \kappa\right)  >0$ since the other modes tend
to zero exponentially, thus, if
\[
\left\{  \kappa\in\sigma\left(  L\right)  \smallsetminus\left\{  0\right\}
;\kappa_{1}<\kappa<\kappa_{2}\right\}  \neq\emptyset,
\]
then
\begin{gather}
\boldsymbol{w}\left(  x,t\right)  \sim\sum\limits_{\kappa_{1}<\kappa
<\kappa_{2}}\sum\limits_{I}A_{I\kappa}e^{\lambda t}\Omega\left(
p^{N}\left\vert x-I\right\vert _{p}\right)  +\label{EQ_expansion}\\
\sum\limits_{\kappa_{1}<\kappa<\kappa_{2}}\text{ }\sum\limits_{I}%
%TCIMACRO{\dsum \limits_{rnj}}%
%BeginExpansion
{\displaystyle\sum\limits_{rnj}}
%EndExpansion
B_{rnj}e^{\lambda t}p^{\frac{r}{2}}\cos\left(  \left\{  p^{-r-1}jx\right\}
_{p}\right)  \Omega\left(  p^{r}\left\vert x-I\right\vert _{p}\right)
+\nonumber\\
\sum\limits_{\kappa_{1}<\kappa<\kappa_{2}}\sum\limits_{I}%
%TCIMACRO{\dsum \limits_{rnj}}%
%BeginExpansion
{\displaystyle\sum\limits_{rnj}}
%EndExpansion
C_{rnj}e^{\lambda t}p^{\frac{r}{2}}\sin\left(  \left\{  p^{-r-1}jx\right\}
_{p}\right)  \Omega\left(  p^{r}\left\vert x-I\right\vert _{p}\right)  \text{
}\nonumber
\end{gather}
for $t\rightarrow+\infty$, where $\lambda=\lambda\left(  \kappa\right)  $,
$r=r(I,\kappa)$, $n=(I,\kappa)$.

Expansion (\ref{EQ_expansion}) implies that%
\begin{equation}
\int\nolimits_{0}^{\boldsymbol{\tau}}\left\Vert f\left(  u\left(  t\right)
,v\left(  t\right)  \right)  \right\Vert _{\infty}dt<\infty\text{ \ and
\ }\int\nolimits_{0}^{\boldsymbol{\tau}}\left\Vert g\left(  u\left(  t\right)
,v\left(  t\right)  \right)  \right\Vert _{\infty}dt<\infty,
\label{EQ_integrals}%
\end{equation}
for any $\boldsymbol{\tau}<\infty$. Which implies that $\tau_{\max}=+\infty$,
for any initial data in $U_{\delta,u_{0}}\oplus U_{\delta,v_{0}}$. Notice that
in (\ref{EQ_integrals}), we can replace $\left\Vert \cdot\right\Vert _{\infty
}$ by $\left\Vert \cdot\right\Vert _{2}$. We now formulate the Turing
instability criteria for our $p$-adic continuous model.

\begin{remark}
In the case of reaction-diffusion systems on $X_{M}$, $M\geq N$, the Turing
pattern (\ref{EQ_expansion}) does not contain the sine and cosine functions.
\end{remark}

\begin{theorem}
\label{Theorem5}Consider the reaction-diffusion system (\ref{EQ_12}). The
steady state $(u_{0},v_{0})$ is linearly unstable (Turing unstable) if the
following conditions hold:

\noindent(T1) $\mathit{Tr}\mathbb{J=}f_{u_{0}}+g_{v_{0}}<0$;

\noindent(T2) $\det\mathbb{J}=f_{u_{0}}g_{v_{0}}-f_{v_{0}}g_{u_{0}}>0$;

\noindent(T3) $df_{u_{0}}+g_{v_{0}}>0$;

\noindent(T4) the derivatives $f_{u_{0}}$ and $g_{v_{0}}$ must have opposite signs;

\noindent(T5) $\left(  df_{u_{0}}+g_{v_{0}}\right)  ^{2}-4d\left(  f_{u_{0}%
}g_{v_{0}}-f_{v_{0}}g_{u_{0}}\right)  >0$;

\noindent(T6) $\left\{  \kappa\in\sigma\left(  \boldsymbol{L}\right)
\smallsetminus\left\{  0\right\}  ;\kappa_{1}<\kappa<\kappa_{2}\right\}
\neq\emptyset$;

Furthermore in (\ref{EQ_2}), we can take $\tau_{0}=+\infty$, for any initial
data in $U_{\delta,u_{0}}\oplus U_{\delta,v_{0}}$.
\end{theorem}

\begin{remark}
\label{Nota_Theorem5}Theorem \ref{Theorem5} is also valid for
reaction-diffusion systems on $X_{M}$, for $M\geq N$. In this case condition
\ (T6) should be replaced by $\left\{  \kappa\in\sigma\left(  \boldsymbol{L}%
_{M}\right)  \smallsetminus\left\{  0\right\}  ;\kappa_{1}<\kappa<\kappa
_{2}\right\}  \neq\emptyset$. Now since%
\[
\left\{  \kappa\in\sigma\left(  \boldsymbol{L}_{M}\right)  \smallsetminus
\left\{  0\right\}  ;\kappa_{1}<\kappa<\kappa_{2}\right\}  \subset\left\{
\kappa\in\sigma\left(  \boldsymbol{L}\right)  \smallsetminus\left\{
0\right\}  ;\kappa_{1}<\kappa<\kappa_{2}\right\}  ,
\]
the existence of Turing patterns in $X_{M}$ implies the existence of Turing
patterns in $X_{\infty}$,\ but the converse is not necessarily true such as
the examples presented in the next section show.
\end{remark}

\subsection{Turing patterns for the $p$-adic Brusselator}

In this section, we study the existence of Turing patterns for the $p$-adic
Brusselator in $X_{M}$, see Section \ref{Sec_Brusselator}. There is only a
homogeneous steady state $u_{0}=A$, $v_{0}=\frac{B}{A}$, the Jacobian matrix
of $\left(  f,g\right)  $ at $\left(  u_{0},v_{0}\right)  $ is
\[
\mathbb{J}=\left[
\begin{array}
[c]{ll}%
B-1 & A^{2}\\
-B & -A^{2}%
\end{array}
\right]  ,
\]
and thus $\mathit{Tr}\mathbb{J}=B-1-A^{2}$ and $\det\mathbb{J}=A^{2}$.
Condition (T1) in Theorem \ref{Theorem5} implies that
\begin{equation}
B<1+A^{2}. \label{Eq_54}%
\end{equation}
The conditions (T2) holds true immediately. By taking $B>1$, condition (T4)
holds true. Condition (T3) implies that%
\begin{equation}
B>\frac{A^{2}}{d}+1. \label{Eq_55}%
\end{equation}
Notice that $d>1$. Condition (T6) requires\ that
\begin{equation}
\left(  d(B-1)-A^{2}\right)  ^{2}-4dA^{2}>0\Leftrightarrow d(B-1)>A^{2}%
+2\sqrt{d}A. \label{Eq_55AA}%
\end{equation}
To verify condition (T6), we compute $\kappa_{1}$, $\kappa_{2}$ as
\begin{align*}
\kappa_{2,1}  &  =\frac{-1}{2d\varepsilon}\left\{  \left(  df_{u_{0}}%
+g_{v_{0}}\right)  \mp\sqrt{\left(  df_{u_{0}}+g_{v_{0}}\right)  ^{2}%
-4d\det\mathbb{J}}\right\} \\
&  =\frac{1}{2d\varepsilon}\left\{  \left(  A^{2}-d(B-1\right)  )\pm
\sqrt{\left(  A^{2}-d(B-1\right)  )^{2}-4dA^{2}}\right\} \\
&  =\frac{1}{2d\varepsilon}\left\{  \left(  A^{2}-d(B-1\right)  )\mp\left(
A^{2}-d(B-1\right)  )\sqrt{1-\frac{4dA^{2}}{\left(  A^{2}-d(B-1\right)  )^{2}%
}}\right\} \\
&  =\frac{\left(  A^{2}-d(B-1\right)  )}{2d\varepsilon}\left\{  1\mp
\sqrt{1-\frac{4dA^{2}}{\left(  A^{2}-d(B-1\right)  )^{2}}}\right\} \\
&  =\frac{\left(  \left(  \frac{A}{\sqrt{d}}\right)  ^{2}-(B-1)\right)
}{2\varepsilon}\left\{  1\mp\sqrt{1-\left(  \frac{2\left(  \frac{A}{\sqrt{d}%
}\right)  }{\left(  \frac{A}{\sqrt{d}}\right)  ^{2}-(B-1)}\right)  ^{2}%
}\right\}  ,
\end{align*}
where we used that $\left(  B-1\right)  d-A^{2}>0$, see (\ref{Eq_55}). Now, we
assume that $d$ is sufficiently large, so we can \ use the Taylor expansion:%
\[
\sqrt{1-\left(  \frac{2\left(  \frac{A}{\sqrt{d}}\right)  }{\left(  \frac
{A}{\sqrt{d}}\right)  ^{2}-(B-1)}\right)  ^{2}}\approx1-\frac{1}{2}\left(
\frac{2\left(  \frac{A}{\sqrt{d}}\right)  }{\left(  \frac{A}{\sqrt{d}}\right)
^{2}-(B-1)}\right)  ^{2},
\]
then%
\[
\kappa_{2}\approx\frac{-A^{2}}{\varepsilon d(B-1)}\text{ and \ }\kappa
_{1}\approx-\frac{(B-1)}{\varepsilon}.
\]
Now since $\kappa_{1}<\kappa_{2}$, i.e. $\frac{d\left(  B-1\right)  }{A^{2}%
}>\frac{1}{B-1}$ $\Leftrightarrow\left(  B-1\right)  >\frac{A^{2}}{d\left(
B-1\right)  }$and by (\ref{Eq_55}), we have $B-1>\max\left\{  \frac{A^{2}}%
{d},\frac{A^{2}}{d\left(  B-1\right)  }\right\}  $. We pick
\begin{equation}
\frac{A^{2}}{d}>\frac{A^{2}}{d\left(  B-1\right)  }\Leftrightarrow B>2.
\label{Eq_56}%
\end{equation}
Now (\ref{Eq_56}) implies that $\kappa_{2}=\frac{-A^{2}}{\varepsilon
d(B-1)}>\frac{-A^{2}}{\varepsilon d}$, \ and consequently, when verifying
condition (T6) we can check the condition:%
\begin{equation}
\left\{  \kappa\in\sigma\left(  \boldsymbol{L}_{\cdot}\right)  \smallsetminus
\left\{  0\right\}  ;-\frac{(B-1)}{\varepsilon}<\kappa<\frac{-A^{2}%
}{\varepsilon d}\right\}  \neq\emptyset. \label{Eq_58}%
\end{equation}
Finally, under the hypothesis $B_{>}2$, condition (\ref{Eq_55}) is satisfied
if $A^{2}<d$. In turn, since $A^{2}+2\sqrt{d}A<3d$, (\ref{Eq_55AA}) is
satisfied if $B>4$.

\subsubsection{\label{Section_complete_graphs}Complete graphs}

A complete graph $K_{n}$ is a connected graph on $n$ vertices \ where all
vertices are of degree $n-1$. A complete graph has $\frac{n\left(  n-1\right)
}{2}$\ edges. The eigenvalues of the Laplacian $\left[  L_{JI}\right]
_{J,I\in G_{N}^{0}}$ are $\mu_{I_{0}}=0$, $\mu_{I_{j}}$ $=-n$ for $1\leq j\leq
n-1$, see e.g. \cite[Section 5.1]{Van Mieghem}. In the case $\boldsymbol{L}%
_{M}$ condition (\ref{Eq_58}) becomes
\[
\frac{A^{2}}{d}<n\varepsilon<B-1.
\]
This last condition is satisfied if $1<n\varepsilon<3$, since we have not
imposed conditions on $\varepsilon$, we can pick $\varepsilon$ satisfying
$\frac{1}{n}<\varepsilon<\frac{3}{n}$.\ In this case, there exist Turing
patterns in $X_{N}$, $X_{M}$ with $M>N$ and in $X_{\infty}$.

In $X_{\infty}$, $\sigma\left(  \boldsymbol{L}\right)  \smallsetminus\left\{
0\right\}  =\left\{  -n\right\}
%TCIMACRO{\tbigsqcup }%
%BeginExpansion
{\textstyle\bigsqcup}
%EndExpansion
\left\{  -n+1\right\}  $. Then condition (\ref{Eq_58}) becomes%
\[
\frac{A^{2}}{d}<n\varepsilon<B-1\text{ and }\frac{A^{2}}{d}<\left(
n-1\right)  \varepsilon<B-1.
\]
These two conditions cannot be satisfied simultaneously \ for some
$\varepsilon>0$. We pick $\frac{3}{n}<\varepsilon<\frac{3}{n-1}$, then there
exists a Turing Pattern in $X_{\infty}$ but no in $X_{M}$ with $M\geq N$.


\bigskip

\begin{thebibliography}{99}                                                                                               %


\bibitem {Alberio et al}S. Albeverio, A. Yu. Khrennikov, V. M. Shelkovich,
\textit{Theory of }$p$\textit{-adic distributions: linear and nonlinear
models}. London Mathematical Society Lecture Note Series, 370. Cambridge
University Press, Cambridge, 2010.

\bibitem {Ambrosio et al}B. Ambrosio, M. A. Aziz-Alaoui, V. L. E. Phan, Global
attractor of complex networks of reaction-diffusion systems of Fitzhugh-Nagumo
type, \textit{Discrete Contin. Dyn. Syst., Ser. B} \textbf{23} (2018), No. 9, 3787-3797.

\bibitem {Antoniuk et al. 1}Alexandra V. Antoniouk , Klaudia Oleschko, Anatoly
N. Kochubei, Andrei Yu Khrennikov, A stochastic p-adic model of the capillary
flow in porous random medium, \textit{Phys. A} 505 (2018), 763--777.

\bibitem {Antoniuk et al.}Alexandra V. Antoniouk, Andrei Yu. Khrennikov,
Anatoly N. Kochubei, Multidimensional nonlinear pseudo-differential evolution
equation with p-adic spatial variables, \textit{J. Pseudo-Differ. Oper. Appl.}
11 (2020), no. 1, 311--343.

\bibitem {Av-4}V. A. Avetisov, A. Kh.\ Bikulov, V. A. Osipov, $p$-adic
description of characteristic relaxation in complex systems, \textit{J. Phys.
A} \textbf{36} (2003), no. 15, 4239--4246.

\bibitem {Av-5}V. A. Avetisov, A. H. Bikulov, S. V. Kozyrev, V. A. Osipov, $p
$-adic models of ultrametric diffusion constrained by hierarchical energy
landscapes, \textit{J. Phys. A} \textbf{35} (2002), no. 2, 177--189.

\bibitem {Becker  et al}O. M. Becker, M. Karplus, The topology of
multidimensional protein energy surfaces: theory and application to peptide
structure and kinetics, \textit{J. Chem.Phys.} \textbf{106 }(1997), 1495--1517.

\bibitem {Bendiko et al}A. D. Bendikov, A. A. Grigor'yan, K. Pitt\`{e}, V.
V\"{e}ss, Isotropic Markov semigroups on ultra-metric spaces, \textit{Russian
Math. Surveys} \textbf{69} (2014), no. 4, 589--680.

\bibitem {Bendikov}Alexander Bendikov, Heat kernels for isotropic-like Markov
generators on ultrametric spaces: a survey, $p$\textit{-Adic Numbers
Ultrametric Anal. Appl. }\textbf{10} (2018), no. 1, 1--11.

\bibitem {Bendikov-Pote}A. Bendikov, P. Krupski, On the spectrum of the
hierarchical Laplacian,\textit{\ Potential Analysis }\textbf{41} (2014), no.
4, 1247-1266.

\bibitem {Berkolaiko}Gregory Berkolaiko, Kuchment Peter,\textit{\ Introduction
to quantum graphs}, American Mathematical Society, Providence, RI, 2013.

\bibitem {Blumenthal}R. M. Blumenthal, R. K. Getoor, \textit{Markov processes
and potential Theory}, Academic Press, New York and London, 1968.

\bibitem {Boccaletti et al}S. Boccaletti, V. Latora, Y. Moreno, M. Chavez,
D.-U. Hwang, Complex networks: structure and dynamics, \textit{Phys. Rep.}
\textbf{424} (2006), no. 4-5, 175--308.

\bibitem {Chung}Soon-Yeong Chung, Jae-Hwang Lee, Blow-up for discrete
reaction-diffusion equations on networks, \textit{Appl. Anal. Discrete Math.}
\textbf{9 } (2015), No. 1, 103-119.

\bibitem {E-K}Stewart N. Ethier, Thomas G. Kurtz,\textit{\ Markov Processes -
Characterization and convergence}, Wiley Series in Probability and
Mathematical Statistics, John Wiley $\And$ Sons, New York, 1986.

\bibitem {Fraunfelder et al}H. Frauenfelder, S. S. Chan, W. S. Chan (eds),
\textit{The Physics of Proteins}, Springer-Verlag, 2010.

\bibitem {Halmos}Paul R. Halmos, \textit{Measure Theor}y, D. Van Nostrand
Company, 1950.

\bibitem {Ide}Yusuke Ide, Hirofumi Izuhara, Takuya Machida, Turing instability
in reaction-diffusion models on complex networks, \textit{Physica A}
\textbf{457 }(2016), 331-347.

\bibitem {Khrennikov et al 1}Andrei Khrennikov, Klaudia Oleschko, Maria de
Jes\'{u}s Correa L\'{o}pez, Application of p-adic wavelets to model
reaction-diffusion dynamics in random porous media, \textit{J. Fourier Anal.
Appl.} 22 (2016), no. 4, 809--822.

\bibitem {Khrennikov et al}Andrei Khrennikov, Klaudia Oleschko, Mar\'{\i}a de
Jes\'{u}s Correa L\'{o}pez, Modeling fluid's dynamics with master equations in
ultrametric spaces representing the treelike structure of capillary networks,
\textit{Entropy} 18 (2016), no. 7, Paper No. 249, 28 pp.

\bibitem {Kh-Kochubei}Andrei Yu. Khrennikov, Anatoly N. Kochubei, $p$ -adic
analogue of the porous medium equation, \textit{J. Fourier Anal. Appl.}
\textbf{24 }(2018), No. 5, 1401-1424 .

\bibitem {KKZuniga}Andrei Khrennikov, Sergei Kozyrev, W. A. Z{\'{u}}%
{\~{n}}iga-Galindo, \textit{Ultrametric Equations and its Applications},
Encyclopedia of Mathematics and its Applications (168), Cambridge University
Press, 2018.

\bibitem {Koblitz}Neal Koblitz,\textit{\ }$p$\textit{-adic Numbers, }%
$p$\textit{-adic Analysis, and Zeta-Functions}, Graduate Texts in Mathematics
No. 58, Springer-Verlag, 1984.

\bibitem {Koch}Anatoly N. Kochubei,\textit{\ Pseudo-differential equations and
stochastics over non-Archimedean fields, }Marcel Dekker, Inc., New York, 2001.

\bibitem {Kozyrev-SB}S. V. Kozyrev, Wavelets and spectral analysis of
ultrametric pseudodifferential operators, \textit{Mat. Sb}.\textbf{198 }
(2007), no. 1, 97-116.

\bibitem {Kozyrev  SV}S. V. Kozyrev, Methods and Applications of Ultrametric
and $p$-Adic Analysis: From Wavelet Theory to Biophysics, \textit{Proceedings
of the Steklov Institute of Mathematics}, \textbf{274} (2011), 1-84.

\bibitem {Lengyel et al}I. Lengyel, I. R. Epstein, Modeling of Turing
structure in the chlorite-iodide-malonic acid-starch reaction system,
\textit{Science} \textbf{251 }(1991), 650--652.

\bibitem {Milan}Milan Miklav\v{c}i\v{c}, A\textit{pplied functional analysis
and partial differential equations}, World Scientific Publishing Co., Inc.,
River Edge, NJ, 1998.

\bibitem {Mocarlo}M. Mocarlo Zheng, Bin Shao, Qi Ouyang, Identifying network
topologies that can generate Turing pattern, \textit{J. Theor. Biol.}
\textbf{408} (2016), 88-96 .

\bibitem {Mugnolo}Delio Mugnolo, \textit{Semigroup methods for evolution
equations on networks. Understanding Complex Systems}, Springer, Cham, 2014.

\bibitem {Murray}J. D. Murray, Mathematical biology. II. Spatial models and
biomedical applications. Third edition. Springer-Verlag, New York, 2003.

\bibitem {Nakao-Mikhailov}Nakao Hiroya and Mikhailov Alexander S., Turing
patterns in network-organized activator-inhibitor systems, \textit{Nature
Physics} \textbf{6 }(2010), 544-550.

\bibitem {Nicolis-Prig}G. Nicolis, I. Prigogine, \textit{Self-organization in
Non-equilibrium Systems.} Wiley Interscience, New-York (1977).

\bibitem {OL-Kh-1}K. Oleschko, Andrei Khrennikov, Transport through a network
of capillaries from ultrametric diffusion equation with quadratic
nonlinearity, \textit{Russ. J. Math. Phys.} \textbf{24} (2017), No. 4, 505-516.

\bibitem {OL-Kh-2}K. Oleschko, A. Yu. Khrennikov, Applications of $p$-adics to
geophysics: linear and quasilinear diffusion of water-in-oil and oil-in-water
emulsions, \textit{Theor. Math. Phys.} \textbf{190 }(2017), No. 1, 154-163.

\bibitem {Othmer et al 1}H. G. Othmer, L. E. Scriven, Instability and dynamic
pattern in cellular networks, \textit{J. Theor. Biol.} \textbf{32 }(1971), 507-537.

\bibitem {Othmer et al 2}H. G. \ Othmer, L. E. Scriven, Nonlinear aspects of
dynamic pattern in cellular networks, \textit{J. Theor. Biol.} \textbf{43
}(1974), 83-112.

\bibitem {Pazy}A. Pazy, \textit{Semigroups of linear operators and
applications to partial differential equations}. Applied Mathematical
Sciences, 44. Springer-Verlag, New York, 1983.

\bibitem {Perthame}Beno\^{\i}t Perthame, \textit{Parabolic equations in
biology. Growth, reaction, movement and diffusion}. Lecture Notes on
Mathematical Modelling in the Life Sciences. Springer, Cham, 2015.

\bibitem {Prig-Lefe}I. Prigogine, R. Lefever, Symmetry breaking instabilities
in dissipative systems II., \textit{J. Chem. Phys.} \textbf{48 }(1968), 1695--1700.

\bibitem {Slavova et al}Angela Slavova, Pietro Zecca, Complex behavior of
polynomial FitzHugh-Nagumo cellular neural network model, \textit{Nonlinear
Anal., Real World Appl.} \textbf{8 } (2007), No. 4, 1331-1340.

\bibitem {Taibleson}M. H. Taibleson, \textit{Fourier analysis on local
fields}, Princeton University Press, 1975.

\bibitem {Taira}Kazuaki Taira, \textit{Boundary value problems and Markov
processes}. Second edition. Lecture Notes in Mathematics, 1499.
Springer-Verlag, 2009.

\bibitem {To-Zuniga-1}Anselmo Torresblanca-Badillo, W. A.
Z\'{u}\~{n}iga-Galindo, Ultrametric diffusion, exponential landscapes, and the
first passage time problem, \textit{Acta Appl. Math}. \textbf{157} (2018), 93--116.

\bibitem {To-Zuniga-2}Anselmo Torresblanca-Badillo, W. A.
Z\'{u}\~{n}iga-Galindo, \textit{Non-Archimedean pseudodifferential operators
and Feller semigroups}, p-Adic Numbers Ultrametric Anal. Appl. \textbf{10}
(2018), no. 1, 57--73.

\bibitem {Turing}A. M. Turing, The chemical basis of morphogenesis,
\textit{Phil. Trans. R. Soc. Lond. B }\textbf{237} (1952), 37-72.

\bibitem {Van Mieghem}Piet Van Mieghem, \textit{Graph spectra for complex
networks}, Cambridge University Press, Cambridge, 2011.

\bibitem {V-V-Z}V. S. Vladimirov, I. V. Volovich, E. I. Zelenov,\textit{\ }$p
$\textit{-adic analysis and mathematical physics}, World Scientific, 1994.

\bibitem {von Below}Joachim von Below, Jos\'{e} A. Lubary, Instability of
stationary solutions of reaction-diffusion-equations on graphs,
\textit{Result. Math.} \textbf{68 }(2015), No. 1-2, 171-201.

\bibitem {Zhao}Hongyong Zhao, Xuanxuan Huang, Xuebing Zhang, Turing
instability and pattern formation of neural networks with reaction-diffusion
terms, \textit{Nonlinear Dyn.} \textbf{76 }(2014), No. 1, 115-124.

\bibitem {Zuniga-LNM-2016}W. A. Z\'{u}\~{n}iga-Galindo,
\textit{Pseudodifferential equations over non-Archimedean spaces}. Lectures
Notes in Mathematics 2174, Springer, Cham, 2016.

\bibitem {Zuniga-JPA}W. A. Z\'{u}\~{n}iga-Galindo, Non-archimedean replicator
dynamics and Eigen's paradox, \textit{J. Phys. A} \textbf{51} (2018), no. 50,
505601, 26 pp.

\bibitem {Zuniga-Nonlinear}W. A. Z{\'{u}}{\~{n}}iga-Galindo, Non-Archimedean
Reaction-Ultradiffusion Equations and Complex Hierarchic
Systems,\ \textit{Nonlinearity} \textbf{31} (2018), no. 6, 2590-2616.

\bibitem {Zuniga-Galindo-PNAS}W. A. Z\'{u}\~{n}iga-Galindo, Eigen's Paradox
and the Quasispecies Model in a Non-Archimedean Framework. arXiv:2111.01745.
\end{thebibliography}
\end{document}